\documentclass[11pt]{amsart}

\usepackage{mathrsfs}
\usepackage{amsmath,amssymb,amsthm}
\usepackage{graphicx}
\usepackage{latexsym}
\usepackage{mathrsfs}
\usepackage{tikz}
\usetikzlibrary{cd,arrows,matrix,backgrounds,positioning,calc,decorations.markings,decorations.pathmorphing,decorations.pathreplacing}
\tikzset{black/.style={circle,fill=black,inner sep=3pt,outer sep=3pt},
         white/.style={circle,fill=white,draw=black,inner sep=3pt,outer sep=3pt},
}
\input xy  

\xyoption{all}
\newtheorem{theorem}{Theorem}[section]

\newtheorem{thm}[theorem]{Theorem}
\newtheorem{cor}[theorem]{Corollary}
\newtheorem{lemm}[theorem]{Lemma}
\newtheorem{prop}[theorem]{Proposition}
\newtheorem{question}[theorem]{Question}

\theoremstyle{definition}
\newtheorem{definition-theorem}[theorem]{Definition-Theorem}
\newtheorem{defi}[theorem]{Definition}
\newtheorem{remk}[theorem]{Remark}
\newtheorem{exam}[theorem]{Example}

\newcommand{\Db}{\mathsf{D}^{\rm b}}

\newcommand{\T}{\operatorname{\mathsf{T}}\nolimits}
\newcommand{\F}{\operatorname{\mathsf{F}}\nolimits}
\newcommand{\FF}{\operatorname{\mathcal F}\nolimits}
\newcommand{\PP}{\operatorname{\mathcal P}\nolimits}

\newcommand{\AD}{\mathsf{AD}}
\newcommand{\NAD}{\mathsf{NAD}}
\newcommand{\DAD}{\mathsf{DAD}}
\newcommand{\RNAD}{\mathsf{RNAD}}
\newcommand{\ANAD}{\mathsf{ANAD}}
\newcommand{\G}{\mathcal{G}}
\newcommand{\R}{\mathcal{R}}
\newcommand{\D}{\mathcal{D}}
\newcommand{\E}{\mathcal{E}}
\newcommand{\sS}{\mathcal{S}}

\newcommand{\xX}{\mathcal{X}}

\DeclareMathOperator{\smc}{\mathsf{smc}}
\DeclareMathOperator{\twosmc}{\mathsf{2-smc}}

\DeclareMathOperator{\brick}{\mathsf{brick}}
\DeclareMathOperator{\sbrick}{\mathsf{sbrick}}

\DeclareMathOperator{\fLsbrick}{\mathsf{f_L-sbrick}}
\DeclareMathOperator{\fRsbrick}{\mathsf{f_R-sbrick}}

\DeclareMathOperator{\ftors}{\mathsf{f-tors}}

\DeclareMathOperator{\ftorf}{\mathsf{f-torf}}
\newcommand{\Filt}{\mathsf{Filt}}

\newcommand{\add}{\mathsf{add}\nolimits}

\newcommand{\id}{\operatorname{id}\nolimits}

\newcommand{\Hom}{\operatorname{Hom}\nolimits}

\newcommand{\Ext}{\operatorname{Ext}\nolimits}

\newcommand{\rad}{\operatorname{rad}\nolimits}
\newcommand{\Soc}{\operatorname{Soc}\nolimits}

\newcommand{\RHom}{\mathbf{R}\strut\kern-.2em\operatorname{Hom}\nolimits}

\def\dim{\mathop{\mathrm{dim}}\nolimits}

\def\Ker{\mathop{\mathrm{Ker}}\nolimits}
\def\Coker{\mathop{\mathrm{Coker}}\nolimits}
\def\Hom{\mathop{\mathrm{Hom}}\nolimits}
\def\End{\mathop{\mathrm{End}}\nolimits}
\def\Ext{\mathop{\mathrm{Ext}}\nolimits}

\def\RHom{\mathop{\mathbb R\mathrm{Hom}}\nolimits}
\DeclareMathOperator{\moduleCategory}{\mathsf{mod}} \renewcommand{\mod}{\moduleCategory}

\newcommand{\Fac}{\mathsf{Fac}\hspace{.01in}}
\def\add{{\mathsf{add}}}

\begin{document}
\title
[Arc diagrams and 2-s.m.c for preprojective algebras]{Arc diagrams and 2-term simple-minded collections of  preprojective algebras of type $A$}

\author{Yuya Mizuno}
\address{Faculty of Liberal Arts and Sciences, 	Osaka Prefecture University, 1-1 Gakuen-cho, Naka-ku, Sakai, Osaka 599-8531, Japan}
\email{yuya.mizuno@las.osakafu-u.ac.jp}
\thanks{The author is supported by 
Grant-in-Aid for Scientific Research 20K03539.}
\begin{abstract}
We study an explicit description of semibricks and 2-term simple-minded collections over preprojective algebras of type $A$ via arc diagrams. 
We provide a bijection between the set of noncrossing arc diagrams (resp. the set of double arc diagrams), which is in bijective correspondence with elements of the symmetric group, and the set of semibricks (resp. the set of 2-term simple-minded collections) over the algebra. 
Moreover we define a mutation and a partial order on the set of double arc diagrams. 
In particular, we obtain a poset isomorphism between the symmetric group and the set of 2-term simple-minded collections. 
As an application of our results, we study semibricks of some quotient algebras of the preprojective algebras of type $A$ and  
we reprove some important results shown by the other authors. 

\end{abstract}
\maketitle
\tableofcontents

\section{Introduction}

\subsection{Background}
Preprojective algebras are one of the important classes of algebras not only in representation theory of algebras but also in many areas of mathematics. 
One of the remarkable properties of this class is that they can unify path algebras of all orientations of the given quiver. 
It has recently turned out this property provides a close connection between 
tilting theory of preprojective algebras and the corresponding Weyl groups \cite{BIRS,IR,M1}. 
This strong link leads to a lot of fruitful consequences to analyze categorical structures using combinatorics of Weyl groups, for example \cite{AM,AIRT,A2,BIRS,GLS,IRRT,IRTT,IZ,M1,M2,MT}. 
In this recent development, a theory of semibricks and $\tau$-tilting modules have been playing an crucial role \cite{AIR,A1}. 
This theory allows us to give a systematic way to study several key categories such as torsion classes and wide subcategories. 

By the result of \cite{M1,A1}, we have established a one-to-one correspondence between semibricks 
over preprojective algebras of Dynkin type and the corresponding Weyl group through $\tau$-tilting theory. 
One of the main aims of this paper is to give an explicit description of semibricks for type $A$ via 
arc diagrams, and provide a direct and simple proof 
of the bijection.

\subsection{Main results}
Fix a positive integer $n$. 
Let $\Pi$ be the preprojective algebra of type $A_n$, 
$\mod \Pi$ the category of finitely generated right $\Pi$-modules and  $\Db(\mod \Pi)$ the bounded derived category of $\mod\Pi$.
We denote by $\sbrick\Pi$ the set of semibricks (Definition \ref{semibricks}) of $\mod\Pi$ and by $\twosmc\Pi$ the set of the 2-term simple-minded collections (SMCs for short) on $\Db(\mod \Pi)$ (Definition \ref{def sms}). 

Our fundamental tool is the notion of arc diagrams introduced by Reading \cite{R3}. 
The set of noncrossing arc diagrams of $n+1$ points (Definition \ref{def nad}, $\NAD$ for short), which is bijective correspondence with elements of the symmetric group of degree $n+1$, provides a combinatorial model for canonical join representations \cite{R3}. 
We give a map from each arc to a module, which is a brick, 
and we will show that this map can be extended to a bijective map from $\NAD$ to $\sbrick\Pi$ (Theorem \ref{bij nad}). 
Moreover, we attach two gradings on arcs, 
which we call \emph{green} and \emph{red}, 
and introduce the notion of the set of double arc diagrams (Definition \ref{def dad}, $\DAD$ for short), which is a set of  bigraded arc diagrams enhanced from  $\NAD$. 
Then we give an interpretation of the gradings as shift functors of the derived category of $\Db(\mod \Pi)$ and 
extend the above map from $\DAD$ to $\twosmc\Pi$, which turned out to be bijection. 
For this purpose, we define a mutation on $\DAD$ (Definition \ref{muta}). The mutation is compatible with the action of a simple generator on the Weyl group (Proposition \ref{comm}), and this fact induces a partial order on $\DAD$ (Corollary \ref{dad partial order}). 

The following picture shows the mutation behavior of $\NAD$ of 4 points. 
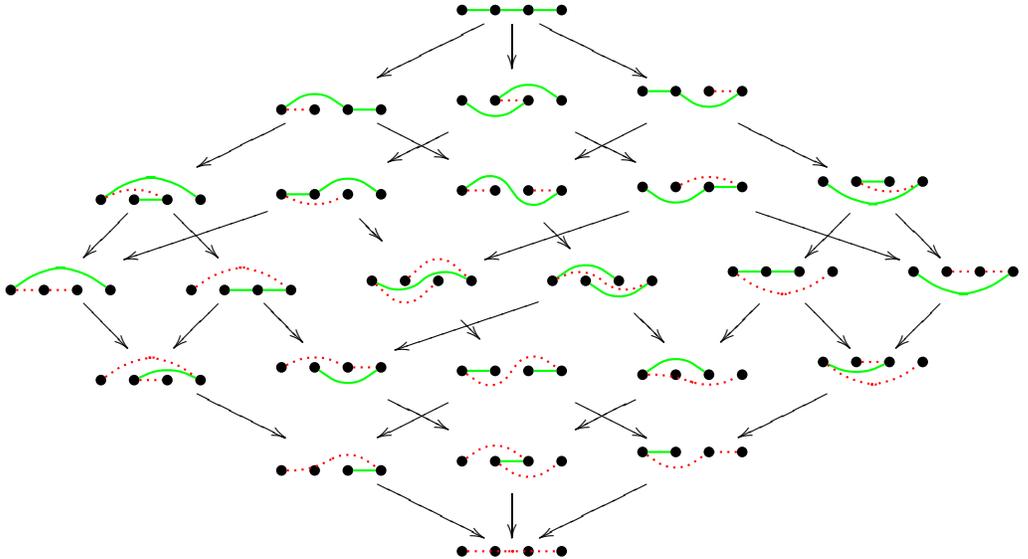
\begin{align*}
\begin{xy}
(  0,-36) *+{\begin{tikzpicture}
      [ mycell/.style={draw, minimum size=1em},
        dot/.style={mycell,
            append after command={\pgfextra \fill (\tikzlastnode) circle[radius=.2em]; \endpgfextra}}]
\begin{scope}[every node/.style={circle, fill=black, inner sep=.5mm, outer sep=0}]
\node (1) {};
\node[right=3mm of 1] (2) {};
\node[right=3mm of 2] (3) {};
\node[right=3mm of 3] (4) {};
      \end{scope}
      \begin{scope}
        [thick, rounded corners=8pt]
         \draw[red,dotted] (1)--($(2)$)--($(3)$)-- (4);
      \end{scope}
      \end{tikzpicture}}="1234",
( 24,-24) *+{\begin{tikzpicture}
      [ mycell/.style={draw, minimum size=1em},
        dot/.style={mycell,
            append after command={\pgfextra \fill (\tikzlastnode) circle[radius=.2em]; \endpgfextra}}]
\begin{scope}[every node/.style={circle, fill=black, inner sep=.5mm, outer sep=0}]
\node (1) {};
\node[right=3mm of 1] (2) {};
\node[right=3mm of 2] (3) {};
\node[right=3mm of 3] (4) {};
      \end{scope}
      \begin{scope}
        [thick, rounded corners=8pt]
        \draw[green]
        (1)--(2); 
         \draw[red,dotted] (4)-- (3);
          \draw[red,dotted] (1)--($(2) - (0,3mm)$)-- (3);
      \end{scope}
      \end{tikzpicture}}="2134",
(  0,-24) *+{\begin{tikzpicture}
      [ mycell/.style={draw, minimum size=1em},
        dot/.style={mycell,
            append after command={\pgfextra \fill (\tikzlastnode) circle[radius=.2em]; \endpgfextra}}]
\begin{scope}[every node/.style={circle, fill=black, inner sep=.5mm, outer sep=0}]
\node (1) {};
\node[right=3mm of 1] (2) {};
\node[right=3mm of 2] (3) {};
\node[right=3mm of 3] (4) {};
      \end{scope}
      \begin{scope}
        [thick, rounded corners=8pt]
        \draw[green]
        (2)-- (3); 
        \draw[red,dotted] (1)--($(2) + (0,3mm)$)-- (3);
        \draw[red,dotted] (2)-- ($(3) - (0,3mm)$)-- (4);
      \end{scope}
      \end{tikzpicture}}="1324",
(-24,-24) *+{\begin{tikzpicture}
      [ mycell/.style={draw, minimum size=1em},
        dot/.style={mycell,
            append after command={\pgfextra \fill (\tikzlastnode) circle[radius=.2em]; \endpgfextra}}]
\begin{scope}[every node/.style={circle, fill=black, inner sep=.5mm, outer sep=0}]
\node (1) {};
\node[right=3mm of 1] (2) {};
\node[right=3mm of 2] (3) {};
\node[right=3mm of 3] (4) {};
      \end{scope}
      \begin{scope}
        [thick, rounded corners=8pt]
        \draw[green]
        (3)-- (4); 
        \draw[red,dotted] (1)-- ($(2)$)--($(3) + (0,3mm)$)-- (4);
      \end{scope}
      \end{tikzpicture}}="1243",
( 48,-12) *+{\begin{tikzpicture}
      [ mycell/.style={draw, minimum size=1em},
        dot/.style={mycell,
            append after command={\pgfextra \fill (\tikzlastnode) circle[radius=.2em]; \endpgfextra}}]
\begin{scope}[every node/.style={circle, fill=black, inner sep=.5mm, outer sep=0}]
\node (1) {};
\node[right=3mm of 1] (2) {};
\node[right=3mm of 2] (3) {};
\node[right=3mm of 3] (4) {};
      \end{scope}
      \begin{scope}
        [thick, rounded corners=8pt]
        \draw[green]
        (1)-- ($(2) - (0,2mm)$)-- (3); 
                  \draw[red,dotted] (1)-- ($(2) - (0,3mm)$)--($(3) - (0,3mm)$)-- (4);
  \draw[red,dotted] (2)-- (3);
      \end{scope}
      \end{tikzpicture}}="2314",
( 24,-12) *+{\begin{tikzpicture}
      [ mycell/.style={draw, minimum size=1em},
        dot/.style={mycell,
            append after command={\pgfextra \fill (\tikzlastnode) circle[radius=.2em]; \endpgfextra}}]
\begin{scope}[every node/.style={circle, fill=black, inner sep=.5mm, outer sep=0}]
\node (1) {};
\node[right=3mm of 1] (2) {};
\node[right=3mm of 2] (3) {};
\node[right=3mm of 3] (4) {};
      \end{scope}
      \begin{scope}
        [thick, rounded corners=8pt]
        \draw[green]
        (1)-- ($(2) + (0,3mm)$)-- (3); 
          \draw[red,dotted] (1)-- ($(2)$)--($(3) - (0,2mm)$)-- (4);
      \end{scope}
      \end{tikzpicture}}="3124",
(  0,-12) *+{\begin{tikzpicture}
      [ mycell/.style={draw, minimum size=1em},
        dot/.style={mycell,
            append after command={\pgfextra \fill (\tikzlastnode) circle[radius=.2em]; \endpgfextra}}]
\begin{scope}[every node/.style={circle, fill=black, inner sep=.5mm, outer sep=0}]
\node (1) {};
\node[right=3mm of 1] (2) {};
\node[right=3mm of 2] (3) {};
\node[right=3mm of 3] (4) {};
      \end{scope}
      \begin{scope}
        [thick, rounded corners=8pt]
        \draw[green]
        (1)-- (2); 
         \draw[green] (3)-- (4); 
          \draw[red,dotted] (1)-- ($(2) - (0,3mm)$)--($(3) + (0,3mm)$)-- (4);
      \end{scope}
      \end{tikzpicture}}="2143",
(-24,-12) *+{\begin{tikzpicture}
      [ mycell/.style={draw, minimum size=1em},
        dot/.style={mycell,
            append after command={\pgfextra \fill (\tikzlastnode) circle[radius=.2em]; \endpgfextra}}]
\begin{scope}[every node/.style={circle, fill=black, inner sep=.5mm, outer sep=0}]
\node (1) {};
\node[right=3mm of 1] (2) {};
\node[right=3mm of 2] (3) {};
\node[right=3mm of 3] (4) {};
      \end{scope}
      \begin{scope}
        [thick, rounded corners=8pt]
        \draw[green]        (2)--($(3) - (0,3mm)$)-- (4); 
      \draw[red,dotted] (3)-- (4); 
        \draw[red,dotted] (1)-- ($(2) + (0,2mm)$)-- (3);
      \end{scope}
      \end{tikzpicture}}="1342",
(-48,-12) *+{\begin{tikzpicture}
      [ mycell/.style={draw, minimum size=1em},
        dot/.style={mycell,
            append after command={\pgfextra \fill (\tikzlastnode) circle[radius=.2em]; \endpgfextra}}]
\begin{scope}[every node/.style={circle, fill=black, inner sep=.5mm, outer sep=0}]
\node (1) {};
\node[right=3mm of 1] (2) {};
\node[right=3mm of 2] (3) {};
\node[right=3mm of 3] (4) {};
      \end{scope}
      \begin{scope}
        [thick, rounded corners=8pt]
        \draw[green]        (2)-- ($(3) + (0,2mm)$)-- (4); 
        \draw[red,dotted] (2)-- (3); 
        \draw[red,dotted] (1)-- ($(2) + (0,3mm)$)--($(3) + (0,3mm)$)-- (4);
      \end{scope}
      \end{tikzpicture}}="1423",
( 60,  0) *+{\begin{tikzpicture}
      [ mycell/.style={draw, minimum size=1em},
        dot/.style={mycell,
            append after command={\pgfextra \fill (\tikzlastnode) circle[radius=.2em]; \endpgfextra}}]
\begin{scope}[every node/.style={circle, fill=black, inner sep=.5mm, outer sep=0}]
\node (1) {};
\node[right=3mm of 1] (2) {};
\node[right=3mm of 2] (3) {};
\node[right=3mm of 3] (4) {};
      \end{scope}
      \begin{scope}
        [thick, rounded corners=8pt]
        \draw[green]
        (1)-- ($(2) - (0,3mm)$)--($(3) - (0,3mm)$)-- (4); 
        \draw[red,dotted] (2)-- (3)-- (4); 
      \end{scope}
      \end{tikzpicture}}="2341",
( 36,  0) *+{\begin{tikzpicture}
      [ mycell/.style={draw, minimum size=1em},
        dot/.style={mycell,
            append after command={\pgfextra \fill (\tikzlastnode) circle[radius=.2em]; \endpgfextra}}]
\begin{scope}[every node/.style={circle, fill=black, inner sep=.5mm, outer sep=0}]
\node (1) {};
\node[right=3mm of 1] (2) {};
\node[right=3mm of 2] (3) {};
\node[right=3mm of 3] (4) {};
      \end{scope}
      \begin{scope}
        [thick, rounded corners=8pt]
        \draw[green]
        (1)-- (2)-- (3); 
        \draw[red,dotted] (1)-- ($(2) - (0,3mm)$)-- ($(3) - (0,3mm)$)-- (4); 
      \end{scope}
      \end{tikzpicture}}="3214",
( 12,  0) *+{\begin{tikzpicture}
      [ mycell/.style={draw, minimum size=1em},
        dot/.style={mycell,
            append after command={\pgfextra \fill (\tikzlastnode) circle[radius=.2em]; \endpgfextra}}]
\begin{scope}[every node/.style={circle, fill=black, inner sep=.5mm, outer sep=0}]
\node (1) {};
\node[right=3mm of 1] (2) {};
\node[right=3mm of 2] (3) {};
\node[right=3mm of 3] (4) {};
      \end{scope}
      \begin{scope}
        [thick, rounded corners=8pt]
        \draw[green]
        (1)-- ($(2) + (0,3mm)$)-- (3); 
          \draw[green] (2)-- ($(3) - (0,3mm)$)-- (4); 
           \draw[red,dotted](1)-- ($(2) + (0,2mm)$)-- ($(3) - (0,2mm)$)-- (4); 
      \end{scope}
      \end{tikzpicture}}="3142",
(-12,  0) *+{\begin{tikzpicture}
      [ mycell/.style={draw, minimum size=1em},
        dot/.style={mycell,
            append after command={\pgfextra \fill (\tikzlastnode) circle[radius=.2em]; \endpgfextra}}]
\begin{scope}[every node/.style={circle, fill=black, inner sep=.5mm, outer sep=0}]
\node (1) {};
\node[right=3mm of 1] (2) {};
\node[right=3mm of 2] (3) {};
\node[right=3mm of 3] (4) {};
      \end{scope}
      \begin{scope}
        [thick, rounded corners=8pt]
         \draw[green] (1)-- ($(2) - (0,2mm)$)-- ($(3) + (0,2mm)$)-- (4); 
          \draw[red,dotted](2)-- ($(3) + (0,4mm)$)-- (4);
           \draw[red,dotted](1)-- ($(2) - (0,4mm)$)-- (3);
      \end{scope}
      \end{tikzpicture}}="2413",
(-36,  0) *+{\begin{tikzpicture}
      [ mycell/.style={draw, minimum size=1em},
        dot/.style={mycell,
            append after command={\pgfextra \fill (\tikzlastnode) circle[radius=.2em]; \endpgfextra}}]
\begin{scope}[every node/.style={circle, fill=black, inner sep=.5mm, outer sep=0}]
\node (1) {};
\node[right=3mm of 1] (2) {};
\node[right=3mm of 2] (3) {};
\node[right=3mm of 3] (4) {};
      \end{scope}
      \begin{scope}
        [thick, rounded corners=8pt]
        \draw[green]
        (2)-- (3)-- (4); 
         \draw[red,dotted](1)-- ($(2) + (0,3mm)$)-- ($(3) + (0,3mm)$)-- (4);
      \end{scope}
      \end{tikzpicture}}="1432",
(-60,  0) *+{\begin{tikzpicture}
      [ mycell/.style={draw, minimum size=1em},
        dot/.style={mycell,
            append after command={\pgfextra \fill (\tikzlastnode) circle[radius=.2em]; \endpgfextra}}]
\begin{scope}[every node/.style={circle, fill=black, inner sep=.5mm, outer sep=0}]
\node (1) {};
\node[right=3mm of 1] (2) {};
\node[right=3mm of 2] (3) {};
\node[right=3mm of 3] (4) {};
      \end{scope}
      \begin{scope}
        [thick, rounded corners=8pt]
        \draw[green] (1)-- ($(2) + (0,3mm)$)-- ($(3) + (0,3mm)$)-- (4) ; 
         \draw[red,dotted] (1)-- (2);
          \draw[red,dotted] (2)-- (3);
      \end{scope}
      \end{tikzpicture}}="4123",
( 48, 12) *+{\begin{tikzpicture}
      [ mycell/.style={draw, minimum size=1em},
        dot/.style={mycell,
            append after command={\pgfextra \fill (\tikzlastnode) circle[radius=.2em]; \endpgfextra}}]
\begin{scope}[every node/.style={circle, fill=black, inner sep=.5mm, outer sep=0}]
\node (1) {};
\node[right=3mm of 1] (2) {};
\node[right=3mm of 2] (3) {};
\node[right=3mm of 3] (4) {};
      \end{scope}
      \begin{scope}
        [thick, rounded corners=8pt]
        \draw[green]        (2)-- (3); 
        \draw[green] (1)-- ($(2) - (0,3mm)$)-- ($(3) - (0,3mm)$)-- (4);
        \draw[red,dotted] (2)-- ($(3) - (0,2mm)$)-- (4);
      \end{scope}
      \end{tikzpicture}}="3241",
( 24, 12) *+{\begin{tikzpicture}
      [ mycell/.style={draw, minimum size=1em},
        dot/.style={mycell,
            append after command={\pgfextra \fill (\tikzlastnode) circle[radius=.2em]; \endpgfextra}}]
\begin{scope}[every node/.style={circle, fill=black, inner sep=.5mm, outer sep=0}]
\node (1) {};
\node[right=3mm of 1] (2) {};
\node[right=3mm of 2] (3) {};
\node[right=3mm of 3] (4) {};
      \end{scope}
      \begin{scope}
        [thick, rounded corners=8pt]
        \draw[green]   (3)-- (4);
     \draw[green] (1)-- ($(2) - (0,3mm)$)--  (3); 
     \draw[red,dotted] (2)-- ($(3) + (0,2mm)$)-- (4);
      \end{scope}
      \end{tikzpicture}}="2431",
(  0, 12) *+{\begin{tikzpicture}
      [ mycell/.style={draw, minimum size=1em},
        dot/.style={mycell,
            append after command={\pgfextra \fill (\tikzlastnode) circle[radius=.2em]; \endpgfextra}}]
\begin{scope}[every node/.style={circle, fill=black, inner sep=.5mm, outer sep=0}]
\node (1) {};
\node[right=3mm of 1] (2) {};
\node[right=3mm of 2] (3) {};
\node[right=3mm of 3] (4) {};
      \end{scope}
      \begin{scope}
        [thick, rounded corners=8pt]
         \draw[green] (1)-- ($(2) + (0,3mm)$)-- ($(3) - (0,3mm)$)-- (4);
         \draw[red,dotted] (1)-- (2);
         \draw[red,dotted] (3)-- (4);
      \end{scope}
      \end{tikzpicture}}="3412",
(-24, 12) *+{\begin{tikzpicture}
      [ mycell/.style={draw, minimum size=1em},
        dot/.style={mycell,
            append after command={\pgfextra \fill (\tikzlastnode) circle[radius=.2em]; \endpgfextra}}]
\begin{scope}[every node/.style={circle, fill=black, inner sep=.5mm, outer sep=0}]
\node (1) {};
\node[right=3mm of 1] (2) {};
\node[right=3mm of 2] (3) {};
\node[right=3mm of 3] (4) {};
      \end{scope}
      \begin{scope}
        [thick, rounded corners=8pt]
        \draw[green]   (1)-- (2);
        \draw[green] (2)-- ($(3) + (0,3mm)$)-- (4);
        \draw[red,dotted] (1)-- ($(2) - (0,2mm)$)-- (3);
      \end{scope}
      \end{tikzpicture}}="4213",
(-48, 12) *+{\begin{tikzpicture}
      [ mycell/.style={draw, minimum size=1em},
        dot/.style={mycell,
            append after command={\pgfextra \fill (\tikzlastnode) circle[radius=.2em]; \endpgfextra}}]
\begin{scope}[every node/.style={circle, fill=black, inner sep=.5mm, outer sep=0}]
\node (1) {};
\node[right=3mm of 1] (2) {};
\node[right=3mm of 2] (3) {};
\node[right=3mm of 3] (4) {};
      \end{scope}
      \begin{scope}
        [thick, rounded corners=8pt]
        \draw[green]       (2)-- (3);
         \draw[green] (1)-- ($(2) + (0,3mm)$)--($(3) + (0,3mm)$)-- (4);
         \draw[red,dotted] (1)-- ($(2) + (0,2mm)$)-- (3);
      \end{scope}
      \end{tikzpicture}}="4132",
( 24, 24) *+{\begin{tikzpicture}
      [ mycell/.style={draw, minimum size=1em},
        dot/.style={mycell,
            append after command={\pgfextra \fill (\tikzlastnode) circle[radius=.2em]; \endpgfextra}}]
\begin{scope}[every node/.style={circle, fill=black, inner sep=.5mm, outer sep=0}]
\node (1) {};
\node[right=3mm of 1] (2) {};
\node[right=3mm of 2] (3) {};
\node[right=3mm of 3] (4) {};
      \end{scope}
      \begin{scope}
        [thick, rounded corners=8pt]
        \draw[green]
        (1)-- (2); 
        \draw[green] (2)-- ($(3) - (0,3mm)$)-- (4);
        \draw[red,dotted] (3)-- (4);
      \end{scope}
      \end{tikzpicture}}="3421",
(  0, 24) *+{\begin{tikzpicture}
      [ mycell/.style={draw, minimum size=1em},
        dot/.style={mycell,
            append after command={\pgfextra \fill (\tikzlastnode) circle[radius=.2em]; \endpgfextra}}]
\begin{scope}[every node/.style={circle, fill=black, inner sep=.5mm, outer sep=0}]
\node (1) {};
\node[right=3mm of 1] (2) {};
\node[right=3mm of 2] (3) {};
\node[right=3mm of 3] (4) {};
      \end{scope}
      \begin{scope}
        [thick, rounded corners=8pt]
        \draw[green] (2)-- ($(3) + (0,3mm)$)-- (4);
        \draw[green] (1)-- ($(2) - (0,3mm)$)-- (3);
        \draw[red,dotted] (2)-- (3);
      \end{scope}
      \end{tikzpicture}}="4231",
(-24, 24) *+{\begin{tikzpicture}
      [ mycell/.style={draw, minimum size=1em},
        dot/.style={mycell,
            append after command={\pgfextra \fill (\tikzlastnode) circle[radius=.2em]; \endpgfextra}}]
\begin{scope}[every node/.style={circle, fill=black, inner sep=.5mm, outer sep=0}]
\node (1) {};
\node[right=3mm of 1] (2) {};
\node[right=3mm of 2] (3) {};
\node[right=3mm of 3] (4) {};
      \end{scope}
      \begin{scope}
        [thick, rounded corners=8pt]
        \draw[green] (3)-- (4);
        \draw[green] (1)-- ($(2) + (0,3mm)$)-- (3);
        \draw[red,dotted] (1)-- (2);
      \end{scope}
      \end{tikzpicture}}="4312",
(  0, 36) *+{\begin{tikzpicture}
      [ mycell/.style={draw, minimum size=1em},
        dot/.style={mycell,
            append after command={\pgfextra \fill (\tikzlastnode) circle[radius=.2em]; \endpgfextra}}]
\begin{scope}[every node/.style={circle, fill=black, inner sep=.5mm, outer sep=0}]
\node (1) {};
\node[right=3mm of 1] (2) {};
\node[right=3mm of 2] (3) {};
\node[right=3mm of 3] (4) {};
      \end{scope}
      \begin{scope}
        [thick, rounded corners=8pt]
        \draw[green]
        (1)-- (2)-- (3)--(4); 
      \end{scope}
      \end{tikzpicture}}="4321",
\ar "2134";"1234"
\ar "1324";"1234"
\ar "1243";"1234"
\ar "2314";"2134"
\ar
"2143";"2134"
\ar "3124";"1324"
\ar "1342";"1324"
\ar
"2143";"1243"
\ar "1423";"1243"
\ar "2341";"2314"
\ar "3214";"2314"
\ar "3214";"3124"
\ar "3142";"3124"
\ar
"2413";"2143"
\ar "3142";"1342"
\ar "1432";"1342"
\ar "1432";"1423"
\ar "4123";"1423"
\ar "3241";"2341"
\ar
"2431";"2341"
\ar "3241";"3214"
\ar "3412";"3142"
\ar
"2431";"2413"
\ar "4213";"2413"
\ar "4132";"1432"
\ar
"4213";"4123"
\ar "4132";"4123"
\ar "3421";"3241"
\ar
"4231";"2431"
\ar "3421";"3412"
\ar "4312";"3412"
\ar
"4231";"4213"
\ar "4312";"4132"
\ar "4321";"3421"
\ar "4321";"4231"
\ar "4321";"4312"
\end{xy}.
\end{align*}

Furthermore, we show that mutation on $\DAD$ is also compatible with mutation on $\twosmc\Pi$ (Proposition \ref{main2}). 
Our main results are summarized as follows.

\begin{thm}(see Proposition \ref{h^0}, Theorems \ref{bij nad}, \ref{two poset iso})
We have the following commutative diagrams, and all maps are poset  isomorphisms.
\begin{align*}
\begin{xy}
( 0,  8) *+{\DAD}   ="01",
(45,  8) *+{\twosmc\Pi }   ="11",
( 0, -8) *+{\NAD}    ="00",
(45, -8) *+{\sbrick \Pi,}     ="10",
(-45,  8) *+{W}   ="02",
(-45,  -8) *+{W}   ="12".
\ar^{\Psi} "01";"11"
\ar^{\Phi}  "00";"10"
\ar^{\mathbb{G}}^{}   "01";"00"
\ar^{H^0}   "11";"10"
\ar^{\D} "02";"01"
\ar^{\G} "12";"00"
\ar@{=}_{}  "02";"12"
\end{xy}
\end{align*}
\end{thm}

We remark that (semi)bricks over the preprojective algebras also have been studied  
by \cite{A2,DIRRT} in a complete different way: 
Asai gave a classification of (semi)bricks in terms of Young diagram-like notation \cite{A2} and Demonet-Iyama-Reading-Reiten-Thomas applied their reduction technique to the algebras and reduced this problem to  gentle algebras \cite{DIRRT}. 
On the other hand, Barnard-Carroll-Zhu \cite{BCZ} studied torsion classes for a quotient of the preprojective algebras via arc diagrams. 
An explicit relationship between this quotient algebra and the preprojective algebra can be explained by using  results of \cite{AIR,A1,M1,DIRRT}. In this paper, we will establish the above result without depending on these results. 



As an application, we study semibricks over quotient algebras of preprojective algebras of type $A_n$. 
Let $I$ be a two-sided ideal of $\Pi$. 
Using the above map, we define a subposet of $\NAD$ as follows 
$$\NAD_I:=\{\G\in\NAD\ |\ \Phi(\G)\subset\mod(\Pi/I)\}.$$

Then we have the following corollary, which recovers some part of results shown by the other authors such as 
\cite{DIRRT},\cite{BCZ}, \cite{IT}, \cite{Ao} and  \cite{AMN}.

\begin{cor}
Let $I$ be a two-sided ideal of $\Pi$. 
We have a poset isomorphism
$$\NAD_I\longrightarrow\sbrick(\Pi/I).$$
In particular, we have the following results. 
\begin{itemize}
\item[(i)] Let $I_{\textnormal{cyc}}$ be the ideal of $\Pi$ generated by all 2-cycles. 
Then we have a poset isomorphism
$$\NAD\longrightarrow\sbrick (\Pi/I_{\textnormal{cyc}}).$$

\item[(ii)] 
Let $\vec{Q}$ be a linear quiver of type $A_n$ and 
$\RNAD$ the set of right noncrossing arc diagrams (Definition \ref{RNAD}). 
Then we have a poset isomorphism
$$\RNAD\longrightarrow\sbrick K\vec{Q}.$$

\item[(iii)] 
Let 
$\ANAD$ be the set of alternating noncrossing arc diagrams (Definition \ref{ANAD}). 
Then we have a poset isomorphism
$$\ANAD\longrightarrow\sbrick (\Pi/\rad^2(\Pi)).$$
\end{itemize}
\end{cor}
Finally we expect that the techniques of this  paper shed a new light on the study of SMCs
and can be widely applied to  several classes of algebras formulated by geometric models   such as gentle algebras.

\section{Preliminaries}
\textbf{Notation}
Fix a natural number $n\geq1$. 
For a finite dimensional algebra $A$ over an algebraically closed field $K$, we denote by $\mod A$ the category of finitely generated right $A$-modules and by $\Db(\mod A)$ the bounded derived category of $\mod A$.

\subsection{Symmetric groups and canonical join representations} 
We recall basic definitions and terminologies about the symmetric group.
\begin{defi}\label{def join}
Let $L$ be a finite lattice.
\begin{itemize}
\item[(1)] An element of $x\in L$ is called \emph{join-irreducible} 
if it is not the minimum element of $L$ and if $x=y\vee z$ 
for some $y,z \in L$, then $y = x$ or $z = x$, where $\vee$ denotes the join of $y$ and $z$, or equivalently, the join-irreducible elements are those which cover precisely one element.

\item[(2)] 
We call $C\subset L$ a \emph{canonical join representation} if
\begin{itemize}
\item[(i)] $x=\bigvee_{c \in C} c$. 
\item[(ii)] For any proper subset $C' \subsetneq C$,
the join $\bigvee_{c \in C'} c$ never coincides with $x$. 
\item[(iii)] If $U \subset L$ satisfies the properties (i) and (ii), then, 
for every $c \in C$, there exists $u \in U$ such that $ c\le u$.
\end{itemize}
\end{itemize}
In this case, we also call $x=\bigvee_{c \in C} c$ a canonical join representation.
\end{defi}

Note that if $x=\bigvee_{c \in C} c$ is a canonical join representation, then it is unique and each element $c\in C$ is join-irreducible. 
As a finite lattice $L$, we consider the symmetric group in this paper. 

Fix a natural number $n\geq1$ and let $[n+1]:=\{1,2,\ldots,n+1\}$.  
A permutation $w$ of $[n+1]$ is a sequence $w=w_1w_2\ldots w_{n+1}$ such that $\{w_1,\ldots,w_{n+1}\}=[n+1]$. 
The weak order on permutations is a partial order whose cover relations are $w_1\cdots w_{n+1}\lessdot v_1\cdots v_{n+1}$ whenever there exists $i$ such that $w_i=v_{i+1}<v_i=w_{i+1}$ and such that $w_j=v_j$ for $j\notin\{i,i+1\}$. 
This is the Weyl group of type $A_{n}$ and 
it is generated by the transpositions $s_i = (i\ i+1)$ ($1\leq i\leq n$), which satisfy the relations $s^2_i = 1$, $s_is_j = s_js_i$ $(|i-j| \geq 2)$ and $s_is_{i+1}s_i = s_{i+1}s_is_{i+1}$. 
We denote by $W=W_{n}$ for the set of permutations of $[n+1]$ and we regard $W$ as a partially ordered set defined by the weak order. 

In this paper, for $w=w_1w_2\ldots w_{n+1}$, we define the left action of $s_i$ by
\begin{eqnarray*}
s_iw&=&w_{s_i(1)}w_{s_i(2)}\cdots w_{s_i(n+1)}\\
&=&w_1w_2\cdots w_{i-1}w_{i+1}w_iw_{i+2}\cdots w_{n+1}.
\end{eqnarray*}


We say that a \emph{descent} of a permutation is a pair $w_i$ and $w_{i+1}$ of adjacent entries such that $w_i>w_{i+1}$. 
A permutation is join-irreducible if and only if it has exactly one descent (see \cite{R3} for more detail).

\subsection{Arc diagrams}
Next we introduce arc diagrams, which is the main subject of this paper. We slightly modify the original definition of \cite{R3} for our convenience. In particular, we arrange points of an arc diagram horizontally.

\begin{defi}\label{def nad}
Put $n+1$ distinct points $1,2,\ldots,n+1$
on a horizontal line from left to right. 
\begin{itemize}
\item[(1)] An \emph{arc} is a curve which connects a point $p\in[n+1]$ to a strictly higher point $q\in[n+1]$, moving monotone upwards from $p$ to $q$ and passing either to the left (=above) or to the right (=below) of each point between $p$ and $q$.
\item[(2)] 
An \emph{arc diagram} consists of some (or no) arcs connecting the points. 
We identify an arc (or arcs) with an arc diagram. 
\item[(3)] An arc diagram is called \emph{noncrossing} if it 
satisfies the following two conditions:
\begin{enumerate}
\item[(nc1)]
No two arcs intersect, except possibly at their endpoints.
\item[(nc2)] No two arcs share the same right endpoint or the same left endpoint.
\end{enumerate}
\end{itemize}
\end{defi}

In this paper, by \emph{arcs intersect}, we always mean that arcs cross in their interiors. 
We also treat arcs and arc diagrams up to isotopies, 
or equivalently, 
an arc diagram is determined by which pairs of points are joined by an arc and which points are above and below of each arc. 
We denote the set of arc diagrams (resp. noncrossing arc diagrams) by $\AD=\AD_n$ (resp. $\NAD=\NAD_n$).

\begin{exam}
This is an example of a noncrossing arc diagram consisting of three arcs. 
$$
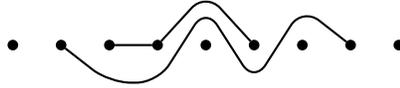
\begin{tikzpicture}
      [ mycell/.style={draw, minimum size=1em},
        dot/.style={mycell,
            append after command={\pgfextra \fill (\tikzlastnode) circle[radius=.2em]; \endpgfextra}}]

\begin{scope}[every node/.style={circle, fill=black, inner sep=.5mm, outer sep=0}]
        \node (1) {};
        \node[right=5mm of 1] (2) {};
        \node[right=5mm of 2] (3) {};
        \node[right=5mm of 3] (4) {};
        \node[right=5mm of 4] (5) {};
        \node[right=5mm of 5] (6) {};
        \node[right=5mm of 6] (7) {};
        \node[right=5mm of 7] (8) {};
        \node[right=5mm of 8] (9) {};
      \end{scope}
      \begin{scope}
        [thick, rounded corners=8pt]
        \draw
        (2)-- ($(3) - (0,5mm)$) -- ($(4) - (0,5mm)$) -- ($(5) + (0,5mm)$) -- 
        ($(6) - (0,5mm)$) -- ($(7) + (0,5mm)$) -- (8) ;
         \draw        (3)-- (4) ;
         \draw        (4)-- ($(5) + (0,7mm)$) --(6) ;
\end{scope}
\end{tikzpicture}$$
The following arc diagrams does not satisfy (nc1) and (nc2), respectively. 

$$
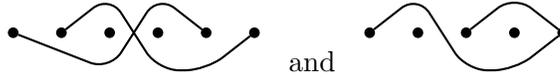
\begin{tikzpicture}
      [ mycell/.style={draw, minimum size=1em},
        dot/.style={mycell,
            append after command={\pgfextra \fill (\tikzlastnode) circle[radius=.2em]; \endpgfextra}}]

      \begin{scope}[every node/.style={circle, fill=black, inner sep=.5mm, outer sep=0}]
        \node (1) {};
        \node[left=5mm of 1] (0) {};
        \node[right=5mm of 1] (2) {};
        \node[right=5mm of 2] (3) {};
        \node[right=5mm of 3] (4) {};
        \node[right=5mm of 4] (5) {};
      \end{scope}
      \begin{scope}
        [thick, rounded corners=8pt]
        \draw  (0)-- ($(2) - (0,5mm)$) -- ($(3) + (0,5mm)$) -- (4) ;
        \draw  (1)-- ($(2) + (0,5mm)$) -- ($(3) - (0,5mm)$) --($(4) - (0,5mm)$) -- (5) ;
      \end{scope}
      \end{tikzpicture}\ \ \ \textnormal{and}\ \ \ 
      \begin{tikzpicture}
      [ mycell/.style={draw, minimum size=1em},
        dot/.style={mycell,
            append after command={\pgfextra \fill (\tikzlastnode) circle[radius=.2em]; \endpgfextra}}]

      \begin{scope}[every node/.style={circle, fill=black, inner sep=.5mm, outer sep=0}]
        \node (1) {};
        \node[right=5mm of 1] (2) {};
        \node[right=5mm of 2] (3) {};
        \node[right=5mm of 3] (4) {};
        \node[right=5mm of 4] (5) {};
      \end{scope}
      \begin{scope}
        [thick, rounded corners=8pt]
     \draw  (1)-- ($(2) + (0,5mm)$) -- ($(3) - (0,5mm)$) --($(4) - (0,5mm)$) -- (5) ;
        \draw  (3)-- ($(4) + (0,5mm)$) -- (5) ;
      \end{scope}
      \end{tikzpicture}$$

\end{exam}

Next we relate arc diagrams with elements of $W$ following \cite{R3}. 
We first define some set of lines. 

\begin{defi}\label{def dad}
Given a permutation $w=w_1 \ldots w_{n+1}$, write each entry $w_i$ at the point $(i,w_i)$ in the plane $\{(a,b)\ |\ 1\leq a,b\leq n+1\}$. 
\begin{itemize}
\item[(1)]
We draw \emph{green lines} between $(i,w_i)$ and $(i+1,w_{i+1})$ if 
$w_i>w_{i+1}$. We denote the set of green lines by $\G(w)$. 
\item[(2)]
We draw \emph{red lines} between $(i,w_i)$ and $(i+1,w_{i+1})$ if 
$w_i<w_{i+1}$. We denote the set of red lines by $\R(w)$. 
\item[(3)] We let $\D(w):=\G(w)\cup\R(w)$.
\end{itemize}
\end{defi}
 
For $w\in W$, define $\G(w)$ as above, and 
move all of the points into a single vertical line, allowing the lines 
to curve but not to pass through any of the points. Then rotate it 90 degrees clockwise. 
These lines become the arcs in an arc diagram. 
By abuse of notation, 
we also write this map by $\G:W\to\AD_n$. 
Similarly, we define 
$\R:W\to\AD_n$ and $\D:W\to\AD_n$. 
We call $\G(w)$ (resp. $\R(w)$, $\D(w)$) \emph{a green arc diagram} (resp. \emph{a red arc diagram}, \emph{a double arc diagram}) of $w$.
We denote the set of double arc diagrams by $\DAD=\DAD_n$, that is, 
$\DAD=\{\D(w)\ |\ w\in W\}$. 

\begin{remk}
The name of green and red arcs comes from \emph{maximal green sequences} in the sense of Keller (we refer to \cite{K}). This name will be justified later (subsection \ref{mutation}).  
\end{remk}

\begin{exam}\label{exam1}
Let $w=53271468$. The left figure shows the construction of green lines and the right one is the corresponding green arc diagram $\G(w)$. 
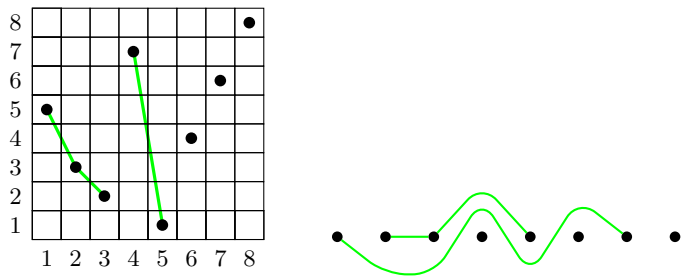
 \[ \begin{tikzpicture}[
      mycell/.style={draw, minimum size=1em},
      dot/.style={mycell,
          append after command={\pgfextra \fill (\tikzlastnode) circle[radius=.2em]; \endpgfextra}}]

  \matrix (m) [matrix of nodes, row sep=-\pgflinewidth, column sep=-\pgflinewidth,
      nodes={mycell}, nodes in empty cells]
  {
  &&&&&&&|[dot]|\\
  &&&|[dot]|&&&&\\
  &&&&&&|[dot]|&\\
  |[dot]|&&&&&&&\\
  &&&&&|[dot]|&&\\
    &|[dot]|&&&&&&\\
      &&|[dot]|&&&&&\\
      &&&&|[dot]|&&&\\
  };
  \begin{scope}[on background layer, every path/.style={very thick}]
    \draw[green] (m-4-1.center) -- (m-6-2.center);
\draw[green] (m-7-3.center) -- (m-6-2.center);
\draw[green] (m-2-4.center) -- (m-8-5.center);
  \end{scope}
  \foreach \i [count=\xi from 1] in  {8,...,1}{
      \node[mycell,label=left:\footnotesize\xi] at (m-\i-1) {};
  }
  \foreach \i [count=\xi from 1] in  {1,...,8}{
      \node[mycell, label=below:\footnotesize\xi] at (m-8-\i) {};  }
  \end{tikzpicture} \quad \quad 
 \begin{tikzpicture}
      [ mycell/.style={draw, minimum size=1em},
        dot/.style={mycell,
            append after command={\pgfextra \fill (\tikzlastnode) circle[radius=.2em]; \endpgfextra}}]

      \begin{scope}[every node/.style={circle, fill=black, inner sep=.5mm, outer sep=0}]
       \node (1) {};
        \node[right=5mm of 1] (2) {};
        \node[right=5mm of 2] (3) {};
        \node[right=5mm of 3] (4) {};
        \node[right=5mm of 4] (5) {};
        \node[right=5mm of 5] (6) {};
        \node[right=5mm of 6] (7) {};
        \node[right=5mm of 7] (8) {};
      \end{scope}
      \begin{scope}
        [thick, rounded corners=8pt]
        \draw[green]
        (1)-- ($(2) - (0,5mm)$) -- ($(3) - (0,5mm)$) -- ($(4) + (0,5mm)$) -- 
        ($(5) - (0,5mm)$) -- ($(6) + (0,5mm)$) -- (7) ;
         \draw[green] (2)-- (3); 
        \draw[green]        (3)-- ($(4) + (0,7mm)$) -- (5); 
      \end{scope}
      \end{tikzpicture}\] 
Moreover, the red arc diagram $\R(w)$ and the double arc diagram $\D(w)$ are, respectively, illustrated as follows.
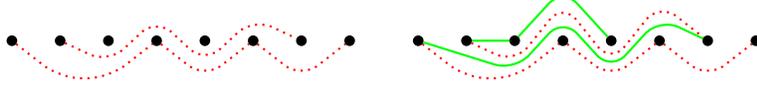
\[\begin{tikzpicture}
      [ mycell/.style={draw, minimum size=1em},
        dot/.style={mycell,
            append after command={\pgfextra \fill (\tikzlastnode) circle[radius=.2em]; \endpgfextra}}]

      \begin{scope}[every node/.style={circle, fill=black, inner sep=.5mm, outer sep=0}]
       \node (1) {};
        \node[right=5mm of 1] (2) {};
        \node[right=5mm of 2] (3) {};
        \node[right=5mm of 3] (4) {};
        \node[right=5mm of 4] (5) {};
        \node[right=5mm of 5] (6) {};
        \node[right=5mm of 6] (7) {};
        \node[right=5mm of 7] (8) {};
      \end{scope}
      \begin{scope}
        [thick, rounded corners=8pt]
        \draw[red,dotted]
        (2)-- ($(3) - (0,3mm)$) -- ($(4) + (0,3mm)$) -- ($(5) - (0,3mm)$) -- 
        ($(6) + (0,3mm)$) -- (7) ;
         \draw[red,dotted]        (1)-- ($(2) - (0,5mm)$) -- ($(3) - (0,5mm)$) -- (4); 
        \draw[red,dotted]        (4)-- ($(5) - (0,5mm)$) -- (6); 
              \draw[red,dotted]        (6)-- ($(7) - (0,5mm)$) -- (8); 
      \end{scope}
      \end{tikzpicture}\quad \quad  \begin{tikzpicture}
      [ mycell/.style={draw, minimum size=1em},
        dot/.style={mycell,
            append after command={\pgfextra \fill (\tikzlastnode) circle[radius=.2em]; \endpgfextra}}]

      \begin{scope}[every node/.style={circle, fill=black, inner sep=.5mm, outer sep=0}]
       \node (1) {};
        \node[right=5mm of 1] (2) {};
        \node[right=5mm of 2] (3) {};
        \node[right=5mm of 3] (4) {};
        \node[right=5mm of 4] (5) {};
        \node[right=5mm of 5] (6) {};
        \node[right=5mm of 6] (7) {};
        \node[right=5mm of 7] (8) {};
      \end{scope}
      \begin{scope}
        [thick, rounded corners=8pt]
        \draw[red,dotted]
        (2)-- ($(3) - (0,3mm)$) -- ($(4) + (0,5mm)$) -- ($(5) - (0,3mm)$) -- 
        ($(6) + (0,5mm)$) -- (7) ;
         \draw[red,dotted]        (1)-- ($(2) - (0,5mm)$) -- ($(3) - (0,5mm)$) -- (4); 
        \draw[red,dotted]       (4)-- ($(5) - (0,5mm)$) -- (6); 
              \draw[red,dotted]        (6)-- ($(7) - (0,5mm)$) -- (8); 
            \draw[green]
        (1)-- ($(2) - (0,2mm)$) -- ($(3) - (0,4mm)$) -- ($(4) + (0,3mm)$) -- 
        ($(5) - (0,4mm)$) -- ($(6) + (0,3mm)$) -- (7) ;
         \draw[green] (2)-- (3); 
        \draw[green][green]        (3)-- ($(4) + (0,7mm)$) -- (5);      
      \end{scope}
      \end{tikzpicture}\]
Here we write green arcs as solid lines and red arcs as dashed lines. 
\end{exam}

The following lemma follows immediately from its construction. 

\begin{lemm}\label{basic proper}
\begin{itemize}
\item[(1)] An element $w\in W$ has exactly one decent (or equivalently, join-irreducible) 
if and only if 
$\G(w)$ consists of one arc.
\item[(2)] A double arc diagram always consists of $n$ arcs and any two arcs never intersect. 
\end{itemize}
\end{lemm}

Moreover we recall the following important result due to Reading. 

\begin{thm}\cite{R3}\label{Reading}
The map $\G:W\to\AD$ and $\R:W\to\AD$
gives a bijection 
$$W\to \NAD.$$

Moreover $\G(w)$ gives a canonical join representation of $w$ by identifying the arcs with join-irreducible elements.
\end{thm}

\begin{proof}
\cite[Theorems 2.4 and 3.1]{R3} shows that the map $\G:W\to\NAD$, where $\G$ is denoted by $\delta$ in \cite{R3}, gives a bijection and $\G(w)$
gives a canonical join representation of $w$. 
Moreover, it is easy to check a bijection $\R:W\to\NAD$ by 
the similar argument of \cite[Theorem 3.1]{R3}. 
\end{proof}


The following lemma also follows immediately from its construction and Theorem \ref{Reading}. 
\begin{prop}\label{h^0}
Define the maps 
$$\mathbb{G}:\DAD\to\NAD, \  \D(w)\mapsto\G(w),\ \ \  \mathbb{R}:\DAD\to\NAD, \   \D(w)\mapsto\R(w)$$
by the restriction of arcs of $\D(w)$ to green arcs and red arcs, respectively. 
Then the following diagram commutes, and all maps are bijections.
\begin{align*}
\begin{xy}
( 0,  8) *+{\DAD}   ="01",
( 0, -8) *+{\NAD}    ="00",
(-45,  8) *+{W}   ="02",
(-45,  -8) *+{W}   ="12",
\ar_{\mathbb{G}}^{\mathbb{R}}   "01";"00"
\ar^{\D} "02";"01"
\ar^{\G} "12";"00"
\ar_{\R} "12";"00"
\ar@{=}_{}  "02";"12"
\end{xy}
\end{align*}
\end{prop}

\begin{remk}\label{G,R,D} 
In \cite{R3}, an explicit bijective map $\NAD\to W$ is explained. 
It implies that, for a given $\G(w),\R(w)$ or $\D(w)$, it is possible to calculate one of the others. 
\end{remk}

\begin{exam}
We illustrate the correspondence $W$ and the set of 
arc diagrams. 
Here we write $W$ by the Hasse quiver. 
The corresponding quiver consisting of $\NAD$ and $\DAD$ is justified from the view point of mutation and partial orders, which will be explained in section \ref{mutation}.

\begin{enumerate}
\item[(1)] We give the correspondence $\G:W\to \NAD$ for $n=2$. 
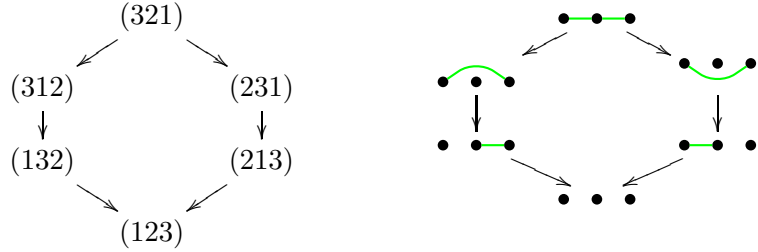
\[\xymatrix@C10pt@R10pt{
&&(321)\ar[dr]\ar[dl]&&\\
&(312)\ar[d]&&(231)\ar[d]&\\
&(132)\ar[dr]&&(213)\ar[dl]&\\
&&(123)&&}\ \ \ \ 
\xymatrix@C10pt@R10pt{&&
 \begin{tikzpicture}
      [ mycell/.style={draw, minimum size=1em},
        dot/.style={mycell,
            append after command={\pgfextra \fill (\tikzlastnode) circle[radius=.2em]; \endpgfextra}}]
\begin{scope}[every node/.style={circle, fill=black, inner sep=.5mm, outer sep=0}]
       \node (1) {};
        \node[right=3mm of 1] (2) {};
        \node[right=3mm of 2] (3) {};
      \end{scope}
      \begin{scope}
        [thick, rounded corners=8pt]
        \draw[green]
        (1)-- (2)-- (3); 
      \end{scope}
      \end{tikzpicture}
      \ar[dr]\ar[dl]&&\\
& \begin{tikzpicture}
      [ mycell/.style={draw, minimum size=1em},
        dot/.style={mycell,
            append after command={\pgfextra \fill (\tikzlastnode) circle[radius=.2em]; \endpgfextra}}]
\begin{scope}[every node/.style={circle, fill=black, inner sep=.5mm, outer sep=0}]
       \node (1) {};
        \node[right=3mm of 1] (2) {};
        \node[right=3mm of 2] (3) {};
      \end{scope}
      \begin{scope}
        [thick, rounded corners=8pt]
        \draw[green]
        (1)-- ($(2) + (0,3mm)$) -- (3); 
      \end{scope}
      \end{tikzpicture}\ar[d]&& \begin{tikzpicture}
      [ mycell/.style={draw, minimum size=1em},
        dot/.style={mycell,
            append after command={\pgfextra \fill (\tikzlastnode) circle[radius=.2em]; \endpgfextra}}]
\begin{scope}[every node/.style={circle, fill=black, inner sep=.5mm, outer sep=0}]
       \node (1) {};
        \node[right=3mm of 1] (2) {};
        \node[right=3mm of 2] (3) {};
      \end{scope}
      \begin{scope}
        [thick, rounded corners=8pt]
        \draw[green]
        (1)-- ($(2) - (0,3mm)$) -- (3); 
      \end{scope}
      \end{tikzpicture}\ar[d]&\\
& \begin{tikzpicture}
      [ mycell/.style={draw, minimum size=1em},
        dot/.style={mycell,
            append after command={\pgfextra \fill (\tikzlastnode) circle[radius=.2em]; \endpgfextra}}]
\begin{scope}[every node/.style={circle, fill=black, inner sep=.5mm, outer sep=0}]
       \node (1) {};
        \node[right=3mm of 1] (2) {};
        \node[right=3mm of 2] (3) {};
      \end{scope}
      \begin{scope}
        [thick, rounded corners=8pt]
        \draw[green] (2)-- (3); 
      \end{scope}
      \end{tikzpicture}\ar[dr]&& \begin{tikzpicture}
      [ mycell/.style={draw, minimum size=1em},
        dot/.style={mycell,
            append after command={\pgfextra \fill (\tikzlastnode) circle[radius=.2em]; \endpgfextra}}]
\begin{scope}[every node/.style={circle, fill=black, inner sep=.5mm, outer sep=0}]
       \node (1) {};
        \node[right=3mm of 1] (2) {};
        \node[right=3mm of 2] (3) {};
      \end{scope}
      \begin{scope}
        [thick, rounded corners=8pt]
        \draw[green]
        (1)-- (2); 
      \end{scope}
      \end{tikzpicture}\ar[dl]&\\
&& \begin{tikzpicture}
      [ mycell/.style={draw, minimum size=1em},
        dot/.style={mycell,
            append after command={\pgfextra \fill (\tikzlastnode) circle[radius=.2em]; \endpgfextra}}]
\begin{scope}[every node/.style={circle, fill=black, inner sep=.5mm, outer sep=0}]
       \node (1) {};
        \node[right=3mm of 1] (2) {};
        \node[right=3mm of 2] (3) {};
      \end{scope}
      \begin{scope} [thick, rounded corners=8pt]
      \end{scope}
      \end{tikzpicture}&&}\]
Moreover, the correspondence $\D:W\to \DAD$ is illustrated as follows. 
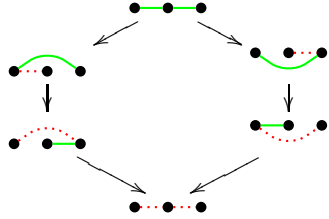
\[
\xymatrix@C10pt@R10pt{&&
 \begin{tikzpicture}
      [ mycell/.style={draw, minimum size=1em},
        dot/.style={mycell,
            append after command={\pgfextra \fill (\tikzlastnode) circle[radius=.2em]; \endpgfextra}}]
\begin{scope}[every node/.style={circle, fill=black, inner sep=.5mm, outer sep=0}]
       \node (1) {};
        \node[right=3mm of 1] (2) {};
        \node[right=3mm of 2] (3) {};
      \end{scope}
      \begin{scope}
        [thick, rounded corners=8pt]
        \draw[green]
        (1)-- (2)-- (3); 
      \end{scope}
      \end{tikzpicture}
      \ar[dr]\ar[dl]&&\\
& \begin{tikzpicture}
      [ mycell/.style={draw, minimum size=1em},
        dot/.style={mycell,
            append after command={\pgfextra \fill (\tikzlastnode) circle[radius=.2em]; \endpgfextra}}]
\begin{scope}[every node/.style={circle, fill=black, inner sep=.5mm, outer sep=0}]
       \node (1) {};
        \node[right=3mm of 1] (2) {};
        \node[right=3mm of 2] (3) {};
      \end{scope}
      \begin{scope}
        [thick, rounded corners=8pt]
        \draw[green]
        (1)-- ($(2) + (0,3mm)$) -- (3); 
        \draw[red,dotted] (1)-- (2);
      \end{scope}
      \end{tikzpicture}\ar[d]&& \begin{tikzpicture}
      [ mycell/.style={draw, minimum size=1em},
        dot/.style={mycell,
            append after command={\pgfextra \fill (\tikzlastnode) circle[radius=.2em]; \endpgfextra}}]
\begin{scope}[every node/.style={circle, fill=black, inner sep=.5mm, outer sep=0}]
       \node (1) {};
        \node[right=3mm of 1] (2) {};
        \node[right=3mm of 2] (3) {};
      \end{scope}
      \begin{scope}
        [thick, rounded corners=8pt]
        \draw[green]
        (1)-- ($(2) - (0,3mm)$) -- (3); 
        \draw[red,dotted] (2)-- (3);
      \end{scope}
      \end{tikzpicture}\ar[d]&\\
& \begin{tikzpicture}
      [ mycell/.style={draw, minimum size=1em},
        dot/.style={mycell,
            append after command={\pgfextra \fill (\tikzlastnode) circle[radius=.2em]; \endpgfextra}}]
\begin{scope}[every node/.style={circle, fill=black, inner sep=.5mm, outer sep=0}]
       \node (1) {};
        \node[right=3mm of 1] (2) {};
        \node[right=3mm of 2] (3) {};
      \end{scope}
      \begin{scope}
        [thick, rounded corners=8pt]
        \draw[red,dotted]
        (1)-- ($(2) + (0,3mm)$) -- (3);
        \draw[green] (2)-- (3); 
      \end{scope}
      \end{tikzpicture}\ar[dr]&& \begin{tikzpicture}
      [ mycell/.style={draw, minimum size=1em},
        dot/.style={mycell,
            append after command={\pgfextra \fill (\tikzlastnode) circle[radius=.2em]; \endpgfextra}}]
\begin{scope}[every node/.style={circle, fill=black, inner sep=.5mm, outer sep=0}]
       \node (1) {};
        \node[right=3mm of 1] (2) {};
        \node[right=3mm of 2] (3) {};
      \end{scope}
      \begin{scope}
        [thick, rounded corners=8pt]
        \draw[green]
        (1)-- (2); 
        \draw[red,dotted] (1)-- ($(2) - (0,3mm)$) -- (3);
      \end{scope}
      \end{tikzpicture}\ar[dl]&\\
&& \begin{tikzpicture}
      [ mycell/.style={draw, minimum size=1em},
        dot/.style={mycell,
            append after command={\pgfextra \fill (\tikzlastnode) circle[radius=.2em]; \endpgfextra}}]
\begin{scope}[every node/.style={circle, fill=black, inner sep=.5mm, outer sep=0}]
       \node (1) {};
        \node[right=3mm of 1] (2) {};
        \node[right=3mm of 2] (3) {};
      \end{scope}
      \begin{scope} [thick, rounded corners=8pt]
        \draw[red,dotted] (1)-- (2)-- (3); 
      \end{scope}
      \end{tikzpicture}&&}\]
Here we write green arcs as normal lines  and red arcs as dotted lines.

\item[(2)] We give the correspondence $\D:W\to \DAD$ for $n=3$. 

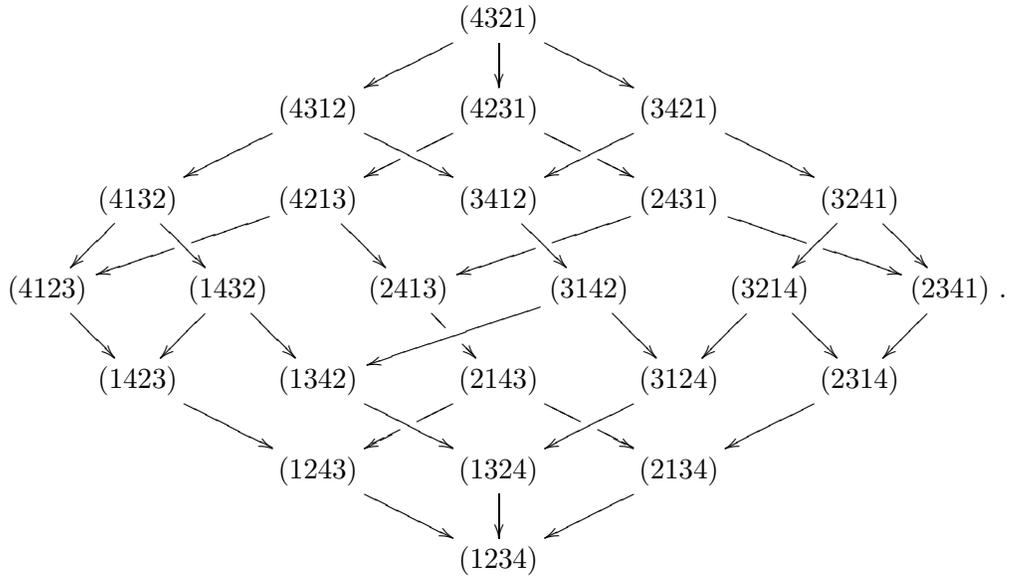
\begin{align*}
\begin{xy}
(  0,-36) *+{(1234)}="1234",
( 24,-24) *+{(2134)}="2134",
(  0,-24) *+{(1324)}="1324",
(-24,-24) *+{(1243)}="1243",
( 48,-12) *+{(2314)}="2314",
( 24,-12) *+{(3124)}="3124",
(  0,-12) *+{(2143)}="2143",
(-24,-12) *+{(1342)}="1342",
(-48,-12) *+{(1423)}="1423",
( 60,  0) *+{(2341)}="2341",
( 36,  0) *+{(3214)}="3214",
( 12,  0) *+{(3142)}="3142",
(-12,  0) *+{(2413)}="2413",
(-36,  0) *+{(1432)}="1432",
(-60,  0) *+{(4123)}="4123",
( 48, 12) *+{(3241)}="3241",
( 24, 12) *+{(2431)}="2431",
(  0, 12) *+{(3412)}="3412",
(-24, 12) *+{(4213)}="4213",
(-48, 12) *+{(4132)}="4132",
( 24, 24) *+{(3421)}="3421",
(  0, 24) *+{(4231)}="4231",
(-24, 24) *+{(4312)}="4312",
(  0, 36) *+{(4321)}="4321",
\ar "2134";"1234"
\ar "1324";"1234"
\ar "1243";"1234"
\ar "2314";"2134"
\ar|\hole "2143";"2134"
\ar "3124";"1324"
\ar "1342";"1324"
\ar|\hole "2143";"1243"
\ar "1423";"1243"
\ar "2341";"2314"
\ar "3214";"2314"
\ar "3214";"3124"
\ar "3142";"3124"
\ar|\hole "2413";"2143"
\ar "3142";"1342"
\ar "1432";"1342"
\ar "1432";"1423"
\ar "4123";"1423"
\ar "3241";"2341"
\ar|\hole "2431";"2341"
\ar "3241";"3214"
\ar "3412";"3142"
\ar|\hole "2431";"2413"
\ar "4213";"2413"
\ar "4132";"1432"
\ar|\hole "4213";"4123"
\ar "4132";"4123"
\ar "3421";"3241"
\ar|\hole "4231";"2431"
\ar "3421";"3412"
\ar "4312";"3412"
\ar|\hole "4231";"4213"
\ar "4312";"4132"
\ar "4321";"3421"
\ar "4321";"4231"
\ar "4321";"4312"
\end{xy}.
\end{align*}  

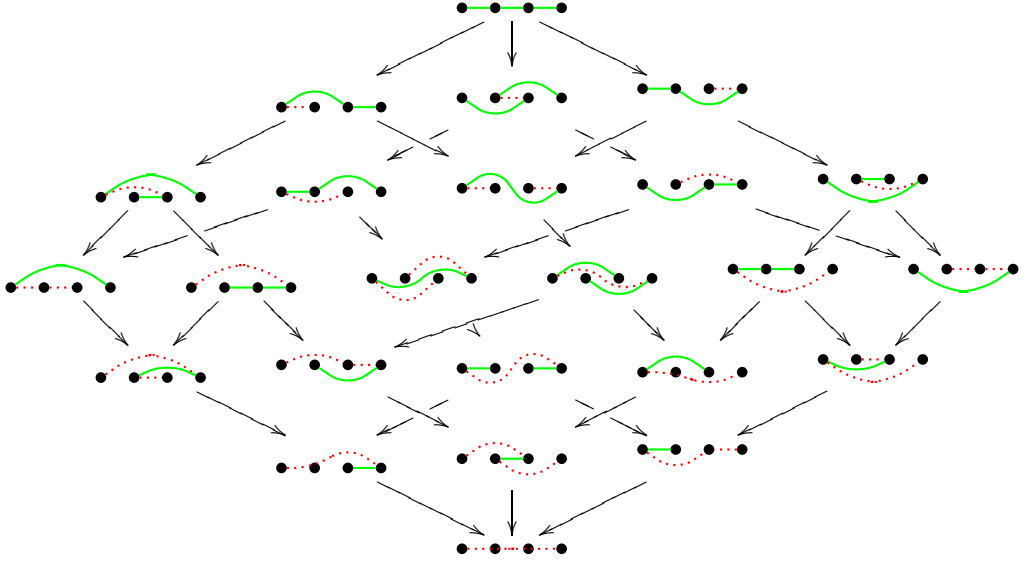
\begin{align*}
\begin{xy}
(  0,-36) *+{\begin{tikzpicture}
      [ mycell/.style={draw, minimum size=1em},
        dot/.style={mycell,
            append after command={\pgfextra \fill (\tikzlastnode) circle[radius=.2em]; \endpgfextra}}]
\begin{scope}[every node/.style={circle, fill=black, inner sep=.5mm, outer sep=0}]
\node (1) {};
\node[right=3mm of 1] (2) {};
\node[right=3mm of 2] (3) {};
\node[right=3mm of 3] (4) {};
      \end{scope}
      \begin{scope}
        [thick, rounded corners=8pt]
         \draw[red,dotted] (1)--($(2)$)--($(3)$)-- (4);
      \end{scope}
      \end{tikzpicture}}="1234",
( 24,-24) *+{\begin{tikzpicture}
      [ mycell/.style={draw, minimum size=1em},
        dot/.style={mycell,
            append after command={\pgfextra \fill (\tikzlastnode) circle[radius=.2em]; \endpgfextra}}]
\begin{scope}[every node/.style={circle, fill=black, inner sep=.5mm, outer sep=0}]
\node (1) {};
\node[right=3mm of 1] (2) {};
\node[right=3mm of 2] (3) {};
\node[right=3mm of 3] (4) {};
      \end{scope}
      \begin{scope}
        [thick, rounded corners=8pt]
        \draw[green]
        (1)--(2); 
         \draw[red,dotted] (4)-- (3);
          \draw[red,dotted] (1)--($(2) - (0,3mm)$)-- (3);
      \end{scope}
      \end{tikzpicture}}="2134",
(  0,-24) *+{\begin{tikzpicture}
      [ mycell/.style={draw, minimum size=1em},
        dot/.style={mycell,
            append after command={\pgfextra \fill (\tikzlastnode) circle[radius=.2em]; \endpgfextra}}]
\begin{scope}[every node/.style={circle, fill=black, inner sep=.5mm, outer sep=0}]
\node (1) {};
\node[right=3mm of 1] (2) {};
\node[right=3mm of 2] (3) {};
\node[right=3mm of 3] (4) {};
      \end{scope}
      \begin{scope}
        [thick, rounded corners=8pt]
        \draw[green]
        (2)-- (3); 
        \draw[red,dotted] (1)--($(2) + (0,3mm)$)-- (3);
        \draw[red,dotted] (2)-- ($(3) - (0,3mm)$)-- (4);
      \end{scope}
      \end{tikzpicture}}="1324",
(-24,-24) *+{\begin{tikzpicture}
      [ mycell/.style={draw, minimum size=1em},
        dot/.style={mycell,
            append after command={\pgfextra \fill (\tikzlastnode) circle[radius=.2em]; \endpgfextra}}]
\begin{scope}[every node/.style={circle, fill=black, inner sep=.5mm, outer sep=0}]
\node (1) {};
\node[right=3mm of 1] (2) {};
\node[right=3mm of 2] (3) {};
\node[right=3mm of 3] (4) {};
      \end{scope}
      \begin{scope}
        [thick, rounded corners=8pt]
        \draw[green]
        (3)-- (4); 
        \draw[red,dotted] (1)-- ($(2)$)--($(3) + (0,3mm)$)-- (4);
      \end{scope}
      \end{tikzpicture}}="1243",
( 48,-12) *+{\begin{tikzpicture}
      [ mycell/.style={draw, minimum size=1em},
        dot/.style={mycell,
            append after command={\pgfextra \fill (\tikzlastnode) circle[radius=.2em]; \endpgfextra}}]
\begin{scope}[every node/.style={circle, fill=black, inner sep=.5mm, outer sep=0}]
\node (1) {};
\node[right=3mm of 1] (2) {};
\node[right=3mm of 2] (3) {};
\node[right=3mm of 3] (4) {};
      \end{scope}
      \begin{scope}
        [thick, rounded corners=8pt]
        \draw[green]
        (1)-- ($(2) - (0,2mm)$)-- (3); 
                  \draw[red,dotted] (1)-- ($(2) - (0,3mm)$)--($(3) - (0,3mm)$)-- (4);
  \draw[red,dotted] (2)-- (3);
      \end{scope}
      \end{tikzpicture}}="2314",
( 24,-12) *+{\begin{tikzpicture}
      [ mycell/.style={draw, minimum size=1em},
        dot/.style={mycell,
            append after command={\pgfextra \fill (\tikzlastnode) circle[radius=.2em]; \endpgfextra}}]
\begin{scope}[every node/.style={circle, fill=black, inner sep=.5mm, outer sep=0}]
\node (1) {};
\node[right=3mm of 1] (2) {};
\node[right=3mm of 2] (3) {};
\node[right=3mm of 3] (4) {};
      \end{scope}
      \begin{scope}
        [thick, rounded corners=8pt]
        \draw[green]
        (1)-- ($(2) + (0,3mm)$)-- (3); 
          \draw[red,dotted] (1)-- ($(2)$)--($(3) - (0,2mm)$)-- (4);
      \end{scope}
      \end{tikzpicture}}="3124",
(  0,-12) *+{\begin{tikzpicture}
      [ mycell/.style={draw, minimum size=1em},
        dot/.style={mycell,
            append after command={\pgfextra \fill (\tikzlastnode) circle[radius=.2em]; \endpgfextra}}]
\begin{scope}[every node/.style={circle, fill=black, inner sep=.5mm, outer sep=0}]
\node (1) {};
\node[right=3mm of 1] (2) {};
\node[right=3mm of 2] (3) {};
\node[right=3mm of 3] (4) {};
      \end{scope}
      \begin{scope}
        [thick, rounded corners=8pt]
        \draw[green]
        (1)-- (2); 
         \draw[green] (3)-- (4); 
          \draw[red,dotted] (1)-- ($(2) - (0,3mm)$)--($(3) + (0,3mm)$)-- (4);
      \end{scope}
      \end{tikzpicture}}="2143",
(-24,-12) *+{\begin{tikzpicture}
      [ mycell/.style={draw, minimum size=1em},
        dot/.style={mycell,
            append after command={\pgfextra \fill (\tikzlastnode) circle[radius=.2em]; \endpgfextra}}]
\begin{scope}[every node/.style={circle, fill=black, inner sep=.5mm, outer sep=0}]
\node (1) {};
\node[right=3mm of 1] (2) {};
\node[right=3mm of 2] (3) {};
\node[right=3mm of 3] (4) {};
      \end{scope}
      \begin{scope}
        [thick, rounded corners=8pt]
        \draw[green]        (2)--($(3) - (0,3mm)$)-- (4); 
      \draw[red,dotted] (3)-- (4); 
        \draw[red,dotted] (1)-- ($(2) + (0,2mm)$)-- (3);
      \end{scope}
      \end{tikzpicture}}="1342",
(-48,-12) *+{\begin{tikzpicture}
      [ mycell/.style={draw, minimum size=1em},
        dot/.style={mycell,
            append after command={\pgfextra \fill (\tikzlastnode) circle[radius=.2em]; \endpgfextra}}]
\begin{scope}[every node/.style={circle, fill=black, inner sep=.5mm, outer sep=0}]
\node (1) {};
\node[right=3mm of 1] (2) {};
\node[right=3mm of 2] (3) {};
\node[right=3mm of 3] (4) {};
      \end{scope}
      \begin{scope}
        [thick, rounded corners=8pt]
        \draw[green]        (2)-- ($(3) + (0,2mm)$)-- (4); 
        \draw[red,dotted] (2)-- (3); 
        \draw[red,dotted] (1)-- ($(2) + (0,3mm)$)--($(3) + (0,3mm)$)-- (4);
      \end{scope}
      \end{tikzpicture}}="1423",
( 60,  0) *+{\begin{tikzpicture}
      [ mycell/.style={draw, minimum size=1em},
        dot/.style={mycell,
            append after command={\pgfextra \fill (\tikzlastnode) circle[radius=.2em]; \endpgfextra}}]
\begin{scope}[every node/.style={circle, fill=black, inner sep=.5mm, outer sep=0}]
\node (1) {};
\node[right=3mm of 1] (2) {};
\node[right=3mm of 2] (3) {};
\node[right=3mm of 3] (4) {};
      \end{scope}
      \begin{scope}
        [thick, rounded corners=8pt]
        \draw[green]
        (1)-- ($(2) - (0,3mm)$)--($(3) - (0,3mm)$)-- (4); 
        \draw[red,dotted] (2)-- (3)-- (4); 
      \end{scope}
      \end{tikzpicture}}="2341",
( 36,  0) *+{\begin{tikzpicture}
      [ mycell/.style={draw, minimum size=1em},
        dot/.style={mycell,
            append after command={\pgfextra \fill (\tikzlastnode) circle[radius=.2em]; \endpgfextra}}]
\begin{scope}[every node/.style={circle, fill=black, inner sep=.5mm, outer sep=0}]
\node (1) {};
\node[right=3mm of 1] (2) {};
\node[right=3mm of 2] (3) {};
\node[right=3mm of 3] (4) {};
      \end{scope}
      \begin{scope}
        [thick, rounded corners=8pt]
        \draw[green]
        (1)-- (2)-- (3); 
        \draw[red,dotted] (1)-- ($(2) - (0,3mm)$)-- ($(3) - (0,3mm)$)-- (4); 
      \end{scope}
      \end{tikzpicture}}="3214",
( 12,  0) *+{\begin{tikzpicture}
      [ mycell/.style={draw, minimum size=1em},
        dot/.style={mycell,
            append after command={\pgfextra \fill (\tikzlastnode) circle[radius=.2em]; \endpgfextra}}]
\begin{scope}[every node/.style={circle, fill=black, inner sep=.5mm, outer sep=0}]
\node (1) {};
\node[right=3mm of 1] (2) {};
\node[right=3mm of 2] (3) {};
\node[right=3mm of 3] (4) {};
      \end{scope}
      \begin{scope}
        [thick, rounded corners=8pt]
        \draw[green]
        (1)-- ($(2) + (0,3mm)$)-- (3); 
          \draw[green] (2)-- ($(3) - (0,3mm)$)-- (4); 
           \draw[red,dotted](1)-- ($(2) + (0,2mm)$)-- ($(3) - (0,2mm)$)-- (4); 
      \end{scope}
      \end{tikzpicture}}="3142",
(-12,  0) *+{\begin{tikzpicture}
      [ mycell/.style={draw, minimum size=1em},
        dot/.style={mycell,
            append after command={\pgfextra \fill (\tikzlastnode) circle[radius=.2em]; \endpgfextra}}]
\begin{scope}[every node/.style={circle, fill=black, inner sep=.5mm, outer sep=0}]
\node (1) {};
\node[right=3mm of 1] (2) {};
\node[right=3mm of 2] (3) {};
\node[right=3mm of 3] (4) {};
      \end{scope}
      \begin{scope}
        [thick, rounded corners=8pt]
         \draw[green] (1)-- ($(2) - (0,2mm)$)-- ($(3) + (0,2mm)$)-- (4); 
          \draw[red,dotted](2)-- ($(3) + (0,4mm)$)-- (4);
           \draw[red,dotted](1)-- ($(2) - (0,4mm)$)-- (3);
      \end{scope}
      \end{tikzpicture}}="2413",
(-36,  0) *+{\begin{tikzpicture}
      [ mycell/.style={draw, minimum size=1em},
        dot/.style={mycell,
            append after command={\pgfextra \fill (\tikzlastnode) circle[radius=.2em]; \endpgfextra}}]
\begin{scope}[every node/.style={circle, fill=black, inner sep=.5mm, outer sep=0}]
\node (1) {};
\node[right=3mm of 1] (2) {};
\node[right=3mm of 2] (3) {};
\node[right=3mm of 3] (4) {};
      \end{scope}
      \begin{scope}
        [thick, rounded corners=8pt]
        \draw[green]
        (2)-- (3)-- (4); 
         \draw[red,dotted](1)-- ($(2) + (0,3mm)$)-- ($(3) + (0,3mm)$)-- (4);
      \end{scope}
      \end{tikzpicture}}="1432",
(-60,  0) *+{\begin{tikzpicture}
      [ mycell/.style={draw, minimum size=1em},
        dot/.style={mycell,
            append after command={\pgfextra \fill (\tikzlastnode) circle[radius=.2em]; \endpgfextra}}]
\begin{scope}[every node/.style={circle, fill=black, inner sep=.5mm, outer sep=0}]
\node (1) {};
\node[right=3mm of 1] (2) {};
\node[right=3mm of 2] (3) {};
\node[right=3mm of 3] (4) {};
      \end{scope}
      \begin{scope}
        [thick, rounded corners=8pt]
        \draw[green] (1)-- ($(2) + (0,3mm)$)-- ($(3) + (0,3mm)$)-- (4) ; 
         \draw[red,dotted] (1)-- (2);
          \draw[red,dotted] (2)-- (3);
      \end{scope}
      \end{tikzpicture}}="4123",
( 48, 12) *+{\begin{tikzpicture}
      [ mycell/.style={draw, minimum size=1em},
        dot/.style={mycell,
            append after command={\pgfextra \fill (\tikzlastnode) circle[radius=.2em]; \endpgfextra}}]
\begin{scope}[every node/.style={circle, fill=black, inner sep=.5mm, outer sep=0}]
\node (1) {};
\node[right=3mm of 1] (2) {};
\node[right=3mm of 2] (3) {};
\node[right=3mm of 3] (4) {};
      \end{scope}
      \begin{scope}
        [thick, rounded corners=8pt]
        \draw[green]        (2)-- (3); 
        \draw[green] (1)-- ($(2) - (0,3mm)$)-- ($(3) - (0,3mm)$)-- (4);
        \draw[red,dotted] (2)-- ($(3) - (0,2mm)$)-- (4);
      \end{scope}
      \end{tikzpicture}}="3241",
( 24, 12) *+{\begin{tikzpicture}
      [ mycell/.style={draw, minimum size=1em},
        dot/.style={mycell,
            append after command={\pgfextra \fill (\tikzlastnode) circle[radius=.2em]; \endpgfextra}}]
\begin{scope}[every node/.style={circle, fill=black, inner sep=.5mm, outer sep=0}]
\node (1) {};
\node[right=3mm of 1] (2) {};
\node[right=3mm of 2] (3) {};
\node[right=3mm of 3] (4) {};
      \end{scope}
      \begin{scope}
        [thick, rounded corners=8pt]
        \draw[green]   (3)-- (4);
     \draw[green] (1)-- ($(2) - (0,3mm)$)--  (3); 
     \draw[red,dotted] (2)-- ($(3) + (0,2mm)$)-- (4);
      \end{scope}
      \end{tikzpicture}}="2431",
(  0, 12) *+{\begin{tikzpicture}
      [ mycell/.style={draw, minimum size=1em},
        dot/.style={mycell,
            append after command={\pgfextra \fill (\tikzlastnode) circle[radius=.2em]; \endpgfextra}}]
\begin{scope}[every node/.style={circle, fill=black, inner sep=.5mm, outer sep=0}]
\node (1) {};
\node[right=3mm of 1] (2) {};
\node[right=3mm of 2] (3) {};
\node[right=3mm of 3] (4) {};
      \end{scope}
      \begin{scope}
        [thick, rounded corners=8pt]
         \draw[green] (1)-- ($(2) + (0,3mm)$)-- ($(3) - (0,3mm)$)-- (4);
         \draw[red,dotted] (1)-- (2);
         \draw[red,dotted] (3)-- (4);
      \end{scope}
      \end{tikzpicture}}="3412",
(-24, 12) *+{\begin{tikzpicture}
      [ mycell/.style={draw, minimum size=1em},
        dot/.style={mycell,
            append after command={\pgfextra \fill (\tikzlastnode) circle[radius=.2em]; \endpgfextra}}]
\begin{scope}[every node/.style={circle, fill=black, inner sep=.5mm, outer sep=0}]
\node (1) {};
\node[right=3mm of 1] (2) {};
\node[right=3mm of 2] (3) {};
\node[right=3mm of 3] (4) {};
      \end{scope}
      \begin{scope}
        [thick, rounded corners=8pt]
        \draw[green]   (1)-- (2);
        \draw[green] (2)-- ($(3) + (0,3mm)$)-- (4);
        \draw[red,dotted] (1)-- ($(2) - (0,2mm)$)-- (3);
      \end{scope}
      \end{tikzpicture}}="4213",
(-48, 12) *+{\begin{tikzpicture}
      [ mycell/.style={draw, minimum size=1em},
        dot/.style={mycell,
            append after command={\pgfextra \fill (\tikzlastnode) circle[radius=.2em]; \endpgfextra}}]
\begin{scope}[every node/.style={circle, fill=black, inner sep=.5mm, outer sep=0}]
\node (1) {};
\node[right=3mm of 1] (2) {};
\node[right=3mm of 2] (3) {};
\node[right=3mm of 3] (4) {};
      \end{scope}
      \begin{scope}
        [thick, rounded corners=8pt]
        \draw[green]       (2)-- (3);
         \draw[green] (1)-- ($(2) + (0,3mm)$)--($(3) + (0,3mm)$)-- (4);
         \draw[red,dotted] (1)-- ($(2) + (0,2mm)$)-- (3);
      \end{scope}
      \end{tikzpicture}}="4132",
( 24, 24) *+{\begin{tikzpicture}
      [ mycell/.style={draw, minimum size=1em},
        dot/.style={mycell,
            append after command={\pgfextra \fill (\tikzlastnode) circle[radius=.2em]; \endpgfextra}}]
\begin{scope}[every node/.style={circle, fill=black, inner sep=.5mm, outer sep=0}]
\node (1) {};
\node[right=3mm of 1] (2) {};
\node[right=3mm of 2] (3) {};
\node[right=3mm of 3] (4) {};
      \end{scope}
      \begin{scope}
        [thick, rounded corners=8pt]
        \draw[green]
        (1)-- (2); 
        \draw[green] (2)-- ($(3) - (0,3mm)$)-- (4);
        \draw[red,dotted] (3)-- (4);
      \end{scope}
      \end{tikzpicture}}="3421",
(  0, 24) *+{\begin{tikzpicture}
      [ mycell/.style={draw, minimum size=1em},
        dot/.style={mycell,
            append after command={\pgfextra \fill (\tikzlastnode) circle[radius=.2em]; \endpgfextra}}]
\begin{scope}[every node/.style={circle, fill=black, inner sep=.5mm, outer sep=0}]
\node (1) {};
\node[right=3mm of 1] (2) {};
\node[right=3mm of 2] (3) {};
\node[right=3mm of 3] (4) {};
      \end{scope}
      \begin{scope}
        [thick, rounded corners=8pt]
        \draw[green] (2)-- ($(3) + (0,3mm)$)-- (4);
        \draw[green] (1)-- ($(2) - (0,3mm)$)-- (3);
        \draw[red,dotted] (2)-- (3);
      \end{scope}
      \end{tikzpicture}}="4231",
(-24, 24) *+{\begin{tikzpicture}
      [ mycell/.style={draw, minimum size=1em},
        dot/.style={mycell,
            append after command={\pgfextra \fill (\tikzlastnode) circle[radius=.2em]; \endpgfextra}}]
\begin{scope}[every node/.style={circle, fill=black, inner sep=.5mm, outer sep=0}]
\node (1) {};
\node[right=3mm of 1] (2) {};
\node[right=3mm of 2] (3) {};
\node[right=3mm of 3] (4) {};
      \end{scope}
      \begin{scope}
        [thick, rounded corners=8pt]
        \draw[green] (3)-- (4);
        \draw[green] (1)-- ($(2) + (0,3mm)$)-- (3);
        \draw[red,dotted] (1)-- (2);
      \end{scope}
      \end{tikzpicture}}="4312",
(  0, 36) *+{\begin{tikzpicture}
      [ mycell/.style={draw, minimum size=1em},
        dot/.style={mycell,
            append after command={\pgfextra \fill (\tikzlastnode) circle[radius=.2em]; \endpgfextra}}]
\begin{scope}[every node/.style={circle, fill=black, inner sep=.5mm, outer sep=0}]
\node (1) {};
\node[right=3mm of 1] (2) {};
\node[right=3mm of 2] (3) {};
\node[right=3mm of 3] (4) {};
      \end{scope}
      \begin{scope}
        [thick, rounded corners=8pt]
        \draw[green]
        (1)-- (2)-- (3)--(4); 
      \end{scope}
      \end{tikzpicture}}="4321",
\ar "2134";"1234"
\ar "1324";"1234"
\ar "1243";"1234"
\ar "2314";"2134"
\ar|\hole "2143";"2134"
\ar "3124";"1324"
\ar "1342";"1324"
\ar|\hole "2143";"1243"
\ar "1423";"1243"
\ar "2341";"2314"
\ar "3214";"2314"
\ar "3214";"3124"
\ar "3142";"3124"
\ar|\hole "2413";"2143"
\ar "3142";"1342"
\ar "1432";"1342"
\ar "1432";"1423"
\ar "4123";"1423"
\ar "3241";"2341"
\ar|\hole "2431";"2341"
\ar "3241";"3214"
\ar "3412";"3142"
\ar|\hole "2431";"2413"
\ar "4213";"2413"
\ar "4132";"1432"
\ar|\hole "4213";"4123"
\ar "4132";"4123"
\ar "3421";"3241"
\ar|\hole "4231";"2431"
\ar "3421";"3412"
\ar "4312";"3412"
\ar|\hole "4231";"4213"
\ar "4312";"4132"
\ar "4321";"3421"
\ar "4321";"4231"
\ar "4321";"4312"
\end{xy}.
\end{align*}

\end{enumerate}
\end{exam}

The aim of this paper is to give a natural  interpretation of these arc diagrams in the representation theory of preprojective algebras. 


\subsection{Torsion classes and semibricks}
In this subsection, we recall the notion of torsion classes and semibricks. 
Let $A$ be a basic finite dimensional algebra. 
We denote by $|A|$ the number of isoclasses of indecomposable direct summands of $A$ and we assume that $|A|=n$. 

First we recall torsion classes and torsion-free classes.

\begin{defi}
A full subcategory $\mathcal{T} \subset \mod A$ is called a \emph{torsion class} 
if $\mathcal{T}$ is closed under taking extensions and factor modules. 
A torsion class $\mathcal{T} \subset \mod A$ is called \emph{functorially finite} if there exists $M\in\mod A$ such that $\mathcal{T}=\Fac M$, where $\Fac M$ is the full subcategory of $\mod A$ consisting of factor modules of finite direct sums of copies of $M$. 
Dually we can define a full subcategory $\mathcal{F} \subset \mod A$ \emph{torsion-free class} and \emph{functorially finite} torsion-free class. 
We denote $\ftors A$ by the set of functorially finite torsion classes in $\mod A$ 
and $\ftorf A$ by the set of functorially finite torsion-free classes in $\mod A$. 
\end{defi}

Next we recall the notion of bricks and semibricks.

\begin{defi}\label{semibricks}
\begin{itemize}
\item[(1)] 
A module $S$ in $\mod A$ is called a \textit{brick} 
if $\End_A(S)$ is a division $K$-algebra.
We denote by $\brick A$ the set of isoclasses of bricks in $\mod A$.  
\item[(2)] 
A subset $\sS \subset \brick A$ is called a \emph{semibrick} 
if $\Hom_A(S_i,S_j)=0$ holds for any $S_i \ne S_j \in \sS$.
We denote by $\sbrick A$ the set of semibricks in $\mod A$.   
\item[(3)] We say that a semibrick $\sS$ is \emph{left finite} (resp. \emph{right finite}) if the smallest torsion class (resp. torsion-free class)  $\T(\sS)\subset\mod A$ (resp. $\F(\sS)\subset\mod A$) containing $\sS$ is functorially finite. 
We denote by $\fLsbrick A$ (resp. $\fRsbrick A$) the set of left finite (resp. right finite) semibricks in $\mod A$.   
\end{itemize}
\end{defi}

Then we briefly recall some important results about torsion classes and semibricks. 

\begin{definition-theorem}\cite[Proposition 1.9]{A1}\cite{Ri}
\label{poset}
We have bijections 
$$\T:\fLsbrick A\to\ftors A,\ \ \ \F:\fRsbrick A\to\ftorf A.$$
For $\sS,\sS'\in\fLsbrick A$, 
we write $\sS\leq\sS'$ if $\T(\sS)\subset\T(\sS')$, and define a partial order on $\fLsbrick A$ by  $\sS\leq\sS'$. 
A partial order on $\fRsbrick A$ is similarly defined.
\end{definition-theorem}

Finally, we recall the following \emph{$\tau$-tilting finite property} 
in terms of bricks. 

\begin{thm}\label{finite brick}\cite{DIJ,A1}
Let $A$ be a finite dimensional algebra. If $\brick A$ is finite, then we have 
$$\brick A=\fLsbrick A=\fRsbrick A.$$
\end{thm}

\subsection{Simple-minded collections}

We next recall the definition of simple-minded collections in $\Db(\mod A)$. We keep the same notation as the previous subsection.

\begin{defi}\label{def sms}
A set $\xX:=\{X_1,X_2,\ldots,X_r\}$ of isoclasses of objects in $\Db(\mod A)$ is called 
a \emph{simple-minded collection} (SMC for short) if it satisfies the following conditions.
\begin{itemize}
\item[(sm1)] For any $X_i \in \xX$, the endomorphism ring 
$\End_{\Db(\mod A)}(X_i)$ is a division $K$-algebra.
\item[(sm2)]  For any $X_i \ne X_j \in \xX$,
we have $\Hom_{\Db(\mod A)}(X_i,X_j)=0$.
\item[(sm3)]  For any $X_i, X_j \in \xX$ and $m < 0$,
we have $\Hom_{\Db(\mod A)}(X_i,X_j[m])=0$.
\item[(sm4)]  The smallest thick subcategory of $\Db(\mod A)$ containing $\xX$ 
is $\Db(\mod A)$.
\end{itemize}
In this case, we always have $r=n$ \cite[Corollary 5.5]{KY}. 
We denote by $\smc A$ for the set of SMCs in $\Db(\mod A)$.
\end{defi}

Moreover we recall a partial order on $\smc A$.

\begin{prop}\label{poset sms}\cite[Proposition 7.9]{KY}
For $\xX,\xX'\in\smc A$, we write $\xX\leq \xX'$ if $$\Hom_{\Db(\mod A)}(X,X'[m])=0$$ for any $m<0$, $X\in\xX$ and $X'\in\xX'$. 
Then $\leq$ gives a partial order on $\smc A$.
\end{prop}

Moreover, we call a complex $X\in\Db(\mod A)$ \emph{2-term}
if the $i$-th cohomology $H^i(X)$ is $0$ for any $i \ne -1,0$.
A SMC $\xX$ in $\Db(\mod A)$ is called  \emph{2-term}
if any $X\in\xX$ is 2-term. 
We denote by $\twosmc A$ for the set of 2-term SMCs in $\Db(\mod A)$.

The following simple remark is also quite important.

\begin{remk}\label{rem}\cite[Remark 4.11]{BY}
If $\xX$ is a 2-term SMC, then 
every $X \in \xX$ belongs to either $\mod A$ or $(\mod A)[1]$ in $\Db(\mod A)$.
\end{remk}

Finally we recall an important relationship between 
semibricks and SMCs. 

\begin{thm}\cite{BY,KY,A1}\label{smc-fsbrick bij}
There exists a poset isomorphism and anti-isomorphism 
\begin{align*}
H^0 \colon \twosmc A \to \fLsbrick A, \quad \mathrm{}\quad H^{-1} \colon \twosmc A \to \fRsbrick A
\end{align*}
given by $\xX \mapsto H^0(\xX) = \xX \cap \mod A$ and
$\xX \mapsto H^{-1}(\xX) = \xX[-1] \cap \mod A$, respectively.
\end{thm}



\section{Double arc diagrams and 2-term SMCs}

In this section, we establish a direct relationship between the set of noncrossing arc diagrams (resp. double arc diagrams) and the set of semibricks (resp. 2-term SMCs) of a preprojective algebra of type $A_n
$. 
In particular, we establish a bijection between $W$ and $\twosmc\Pi$, which 
gives another proof of \cite{A2}. 
For this purpose, we relate each arc (with a color) to a brick (with a shift functor). Moreover, we study hom-spaces between two bricks in terms of combinatorics of arcs. 

From now on, fix an integer $n\geq 1$. 
As before, we denote by $W=W_{n}$ for the set of permutations of $\{1,2,\ldots,n+1\}$.
Let $Q$ be the following quiver 
$$Q=(\xymatrix@C30pt@R10pt{
v_1 \ar@<-0.5ex>[r]_{a_1} &v_2 \ar@<-0.5ex>[r]_{a_2}\ar@<-0.5ex>[l]_{a_1^{-}}&\ar@<-0.5ex>[l]_{a_2^{-}}\cdots\ar@<-0.5ex>[r]_{a_{n-2}}&v_{n-1}\ar@<-0.5ex>[l]_{a_{n-2}^{-}}\ar@<-0.5ex>[r]_{a_{n-1}}& \ar@<-0.5ex>[l]_{a_{n-1}^{-}} v_n}).$$
We denote by $Q_0$ the set of vertices of $Q$, that is, 
$Q_0=\{
v_1,
\ldots,
v_n
\}$, and  $Q_1$ the set of arrows of $Q$, that is, 
$Q_1=\{
a_1,
\ldots,
a_{n-1}, a_1^{-},
\ldots,
a_{n-1}^{-}
\}$.
For the convenience, we write $a_i=a_i^+$ and hence 
we can write any arrow of $Q$ as $a_i^{\epsilon_i}$ for some $\epsilon_i\in\{+,-\}$ and $1\leq i\leq n-1.$ 
Let $\Pi=\Pi_n$ be the preprojective algebra of type $A_{n}$, that is, $\Pi=KQ/\langle\sum_{i=1}^{n-1} (a_i^{-}a_i-a_{i}a_{i}^{-} )\rangle$.

\subsection{Bricks and arc modules}
In this subsection, we give a correspondence between arcs and bricks, which is also explained in \cite[Appendix]{E}. 

\begin{defi}
For an arc $\alpha$, we define a $\Pi$-module $S(\alpha)$ by the following  steps. 
\begin{enumerate}
  \item Draw $n$ vertical dashed lines between adjacent points in the arc diagram, and name these lines as $v_1,v_2,\dots,v_n$ from left to right.
  \item Define a subquiver $Q_\alpha$ of $Q$ as follows.
  \begin{itemize}
    \item The vertex set of $Q_\alpha$ consists of $v_i \in Q_0$ such that $\alpha$ and the line $v_i$ intersect.
    \item Suppose that we have $v_i, v_{i+1} \in (Q_\alpha)_0$.
    If the segment of $\alpha$ cut by the lines $v_i$ and $v_{i+1}$ is below the point $i+1$, 
    then we put an arrow $a_i:v_{i} \to v_{i+1}$
    , and put $a_i^{-}:v_{i+1} \to v_{i}$ if the segment is above the point.
   \end{itemize}
  \item  Construct $S(\alpha)$ as a representation of $Q$ as follows. 
  \begin{itemize}
    \item To each $v_i \in Q_0$, we assign $K$ if $i \in (Q_\alpha)_0$, and $0$ otherwise.
    \item To each arrow $v_i \to v_j \in Q$, we assign the identity map if $v_i \to v_j$ belongs to $Q_\alpha$, and $0$ otherwise.
  \end{itemize}
\end{enumerate}
\end{defi}

By this correspondence, 
we define a map  
$$S:\{\textnormal{arcs}\}\to\mod\Pi, \alpha \mapsto S(\alpha),$$
and we call $S(\alpha)$ an \emph{arc module}.

Moreover, we can give another formulation of arc modules as follows.

The conditions (1) and (2) imply that 
if an arc $\alpha$ has the left endpoint $i$ and the right endpoint $j+1$, then we can define the sequence   $a_i^{\epsilon_i}a_{i+1}^{\epsilon_{i+1}}a_{i+2}^{\epsilon_{i+2}}\cdots a_{j-1}^{\epsilon_{j-1}}$ ($\epsilon\in\{+,-\}$) of arrows (possibly an idempotent) 
of $Q$. 
More precisely, if $i<j$, let ${}_k|\alpha|_{k+1}$ be the central segment of $\alpha$ cut by the lines $v_k$ and $v_{k+1}$ ($i\leq k<j$). Then ${}_k|\alpha|_{k+1}$ corresponds to the arrow $a_k:v_{k} \to v_{k+1}$ if the segment is below the point, and $a_k^{-}:v_{k+1} \to v_{k}$ if the segment is above the point. 
We repeat this procedure from $i$ to $j+1$ (left to right) and we have a unique sequence of arrows.
If $i=j$, then we associate $\alpha$ with the idempotent $e_i$.

 
By this correspondence, we associate an arc with the sequence of arrows and we call it the \emph{arrow  
sequence} of $\alpha$ (which is not necessary a path) and write $\alpha=a_i^{\epsilon_i}a_{i+1}^{\epsilon_{i+1}}\cdots a_{j-1}^{\epsilon_{j-1}}$. 

Then we define $S(a_i^{\epsilon_i}a_{i+1}^{\epsilon_{i+1}}\cdots a_{j-1}^{\epsilon_{j-1}})$ by assigning $K$ to the vertices $v_i,v_{i+1},\ldots,v_j$ and assigning the identity maps  for the arrows $a_i^{\epsilon_i},a_{i+1}^{\epsilon_{i+1}},\cdots, a_{j-1}^{\epsilon_{j-1}}$. 
Then we have  $S(\alpha)=S(a_i^{\epsilon_i}a_{i+1}^{\epsilon_{i+1}}\cdots a_{j-1}^{\epsilon_{j-1}})$.

We remark that arc modules can be defined by \emph{non-revisiting walks} as \cite{DIRRT} and by a \emph{Young diagram-like notation} as \cite{A2}. 
It is also defined by using an arc in \cite{BCZ} for a quotient algebra of the preprojective algebra.



Now we give an example. 

\begin{exam}
Let $\alpha$ be the following arc. 
    
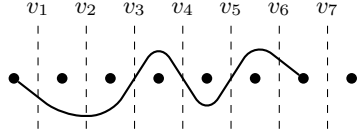
\label{fig:bex1}
\[\begin{tikzpicture}
      [ mycell/.style={draw, minimum size=1em},
        dot/.style={mycell,
            append after command={\pgfextra \fill (\tikzlastnode) circle[radius=.2em]; \endpgfextra}}]

      \begin{scope}[every node/.style={circle, fill=black, inner sep=.5mm, outer sep=0}]
        \node[right=5mm of 1] (2) {};
        \node[right=5mm of 2] (3) {};
        \node[right=5mm of 3] (4) {};
        \node[right=5mm of 4] (5) {};
        \node[right=5mm of 5] (6) {};
        \node[right=5mm of 6] (7) {};
        \node[right=5mm of 7] (8) {};
        \node[right=5mm of 8] (9) {};
      \end{scope}

      \begin{scope}[every path/.style={dashed}, every node/.style={font=\footnotesize}]
        \draw ($(2)!0.5!(3) + (0,7mm)$) node[above] {$v_1$} -- +(0,-14mm);
        \draw ($(3)!0.5!(4) + (0,7mm)$) node[above] {$v_2$} -- +(0,-14mm);
        \draw ($(4)!0.5!(5) + (0,7mm)$) node[above] {$v_3$} -- +(0,-14mm);
        \draw ($(5)!0.5!(6) + (0,7mm)$) node[above] {$v_4$} -- +(0,-14mm);
        \draw ($(6)!0.5!(7) + (0,7mm)$) node[above] {$v_5$} -- +(0,-14mm);
        \draw ($(7)!0.5!(8) + (0,7mm)$) node[above] {$v_6$} -- +(0,-14mm);
        \draw ($(8)!0.5!(9) + (0,7mm)$) node[above] {$v_7$} -- +(0,-14mm);
      \end{scope}
      \begin{scope}
        [thick, rounded corners=8pt]
        \draw
        (2)-- ($(3) - (0,5mm)$) -- ($(4) - (0,5mm)$) -- ($(5) + (0,5mm)$) -- 
        ($(6) - (0,5mm)$) -- ($(7) + (0,5mm)$) -- (8) ;
      \end{scope} 
      \end{tikzpicture}\]
Then $a_1a_2a_3^{-1}a_4a_5^{-1}$ is the arrow sequence of $\alpha$.       
Then the corresponding arc module $S(\alpha)$ is defined as follows 
$$
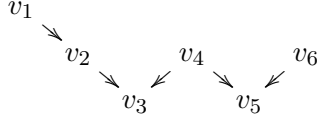
\xymatrix@C5pt@R5pt{
  v_1\ar[rd]&&&&&&&\\
  &v_2\ar[rd]&&v_4\ar[rd]\ar[ld]&&v_6\ar[ld]&\\
  &&v_3&&v_5&&\\
  &&&&&&
  }$$
where $v_i$ shows a $K$-vector space $K$ lying on the vertex $v_i$, and each arrow is the identity map. 
\end{exam}

\begin{remk}
We can also interpret the above construction in terms of  surface models as follows. 
Let $S$ be a disk with $n+1$ punctures
with a set of distinguished points $1,2,3,\cdots,n+1$ from left to right in the interior of $S$. 
Assume that the boundary has counterclockwise orientation.
  
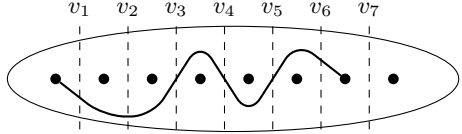
\label{fig:be2}
\[\begin{tikzpicture}
      [ mycell/.style={draw, minimum size=1em},
        dot/.style={mycell,
            append after command={\pgfextra \fill (\tikzlastnode) circle[radius=.2em]; \endpgfextra}}]

      \begin{scope}[every node/.style={circle, fill=black, inner sep=.5mm, outer sep=0}]
        \node[right=5mm of 1] (2) {};
        \node[right=5mm of 2] (3) {};
        \node[right=5mm of 3] (4) {};
        \node[right=5mm of 4] (5) {};
        \node[right=5mm of 5] (6) {};
        \node[right=5mm of 6] (7) {};
        \node[right=5mm of 7] (8) {};
        \node[right=5mm of 8] (9) {};
      \end{scope}
  \draw[] (3,0) ellipse (3 and 0.7);
      \begin{scope}[every path/.style={dashed}, every node/.style={font=\footnotesize}]
        \draw ($(2)!0.5!(3) + (0,7mm)$) node[above] {$v_1$} -- +(0,-14mm);
        \draw ($(3)!0.5!(4) + (0,7mm)$) node[above] {$v_2$} -- +(0,-14mm);
        \draw ($(4)!0.5!(5) + (0,7mm)$) node[above] {$v_3$} -- +(0,-14mm);
        \draw ($(5)!0.5!(6) + (0,7mm)$) node[above] {$v_4$} -- +(0,-14mm);
        \draw ($(6)!0.5!(7) + (0,7mm)$) node[above] {$v_5$} -- +(0,-14mm);
        \draw ($(7)!0.5!(8) + (0,7mm)$) node[above] {$v_6$} -- +(0,-14mm);
        \draw ($(8)!0.5!(9) + (0,7mm)$) node[above] {$v_7$} -- +(0,-14mm);
      \end{scope}
      \begin{scope}
        [thick, rounded corners=8pt]
        \draw
        (2)-- ($(3) - (0,5mm)$) -- ($(4) - (0,5mm)$) -- ($(5) + (0,5mm)$) -- 
        ($(6) - (0,5mm)$) -- ($(7) + (0,5mm)$) -- (8) ;
      \end{scope} 
      \end{tikzpicture}\]
Then each segment ${}_k|\alpha|_{k+1}$ is isotopy to the boundary of $S$ and can be oriented uniquely as same as the boundary of $S$. This orientation determines an arrow sequence of given arc. 
\end{remk}

From now on, for given arcs $\alpha$ and $\beta$, we will study hom-space 
$\Hom_\Pi(S(\alpha),S(\beta))$. 
Since arc modules are nothing but representations of some path algebra of type $A_n$, it is not difficult to calculate it directly.  
Here we will formulate its calculation using a similar technique of gentle algebras by \cite{CB1,S,BDMTY}.

For an arc $\alpha$, consider a  factorization of the arrow sequence. 
We denote the set of all factorizations of the arrow sequence $\alpha=a_i^{\epsilon_i}a_{i+1}^{\epsilon_{i+1}}\cdots a_{j-1}^{\epsilon_{j-1}}$ by 
$$\PP(\alpha) :=\{(b,c,d)\ |\ a_i^{\epsilon_i}a_{i+1}^{\epsilon_{i+1}}\cdots a_{j-1}^{\epsilon_{j-1}}=bcd \}.$$

We call a triple $(b,c,d)\in \PP(\alpha)$ a \emph{quotient factorization} of $\alpha$ if the following conditions are satisfied 
\begin{itemize}
\item $b$ is empty or $b=b'a_s^{-}$ for some $i\leq s\leq j-1$.
\item $d$ is empty  or $d=a_td'$ for some $i\leq t\leq j-1$.  
\end{itemize}

It is easy to see that a quotient factorization $(b,c,d)\in \PP(\alpha)$ induces a surjective map $\xymatrix@C15pt@R5pt{S(\alpha) \ar@{->>}[r]& S(c). }$
We denote by $\FF(\alpha)$ the set of all quotient factorizations of $\alpha$.

Dually, we can define a \emph{submodule factorization} of $\alpha$ and denote by $\sS(\alpha)$ the set of all submodule factorizations of $\alpha$. 
Similarly, we have an injective map
 $\xymatrix@C15pt@R5pt{S(c) \ar@{^{(}->}[r]& S(\alpha) }$
for a submodule factorization $(b,c,d)\in \sS(\alpha)$. 

Let $(b,c,d)\in\FF(\alpha)$ and $(b',c,d')\in\sS(\beta)$. 
For a pair $T:=((b,c,d),(b',c,d'))$,
we have a natural homomorphism 
$$f_T:S(\alpha)\to S(\beta),$$ 
defined as the composition 

\[
\xymatrix{
S(\alpha) \ar@{->>}[r]
& S(c) \ar@{^{(}->}[r] & S(\beta).
}
\]
We call these homomorphisms $f_T$ \emph{graph maps} following \cite{BDMTY}.

Note that a preprojective algebra is not necessary a gentle (string) algebra in general, 
but we can easily check that the similar result works for arc modules since $S(\alpha)\in\mod\Pi/I_{\textnormal{cyc}}$, where $I_{\textnormal{cyc}}$ is the ideal generated by all 2-cycles and $\Pi/I_{\textnormal{cyc}}$ is gentle. The following result is due to \cite{CB1}.

\begin{thm}\cite{CB1}\label{CWthm} 
The set of 
graph maps from $S(\alpha)$ to $S(\beta)$
is a basis for $\Hom_\Pi(S(\alpha),S(\beta))$. 
\end{thm}

From this theorem, we can formulate  $\Hom_\Pi(S(\alpha),S(\beta))$ 
by a simple combinatorics of $\alpha$ and $\beta$.

We start with the following easy lemma about arc modules.

\begin{lemm}\label{S is brick}
For an arc $\alpha$, $S(\alpha)$ is a brick (i.e. arc modules are brick).
\end{lemm}

\begin{proof}
This follows from Theorem \ref{CWthm} (or it is easy to check it directly since $S(\alpha)$ is an indecomposable $KQ_\alpha$-module whose entries of the dimension vector consists of 1). 
\end{proof}

Next we will show that any brick is an arc module. For this purpose, we recall some basic result of preprojective algebras, which comes from 2-Calabi-Yau property.

Let $(-,-)$ be the symmetric bilinear form on $\mathbb{Z}^n$ defined by 
$$(\mathbf{x},\mathbf{y}) :=\sum_{i\in Q_0}2x_iy_i-\sum_{a:i\to j\in Q_1} x_iy_j,$$

where $\mathbf{x}=(x_1,\ldots,x_n)$ and $\mathbf{y}=(y_1,\ldots,y_n)$. 
In particular, 
we have the associated quadratic form 
$$q(\mathbf{x}):=(\mathbf{x},\mathbf{x})=x_1^2 +(x_1-x_2)^2+(x_2-x_3)^2+\ldots+(x_{n-1}-x_n)^2 +x_n^2.$$


\begin{prop}\cite{CB2}\label{cw}
Let  $\underline{\dim}X$ (resp. $\underline{\dim}Y$)  
be the dimension vector of $X\in\mod\Pi$ (resp. $Y\in\mod\Pi$). 
Then we have $$(\underline{\dim}X,\underline{\dim}Y)=
\dim\Hom_\Pi(X,Y)+\dim\Hom_\Pi(Y,X)-\dim\Ext_\Pi^1(X,Y).$$
\end{prop}

Using this result, we determine all bricks in terms of arcs.

\begin{prop}\label{arc-brick bij}
\begin{itemize}
\item[(1)] 
The map 
$$S:\{\textnormal{arcs}\}\to\brick\Pi$$ is bijection. 
\item[(2)] 
Any brick does not admit non-trivial self-extension.
\item[(3)] 
We have 
$$\sbrick\Pi=\fLsbrick\Pi=\fRsbrick\Pi.$$
\end{itemize}
\end{prop}

\begin{proof}
By Lemma \ref{S is brick}, the map is well-defined and clearly injective. 
We will show the surjectivity. 
Let $X\neq0\in\mod\Pi$ be a  brick, which is clearly indecomposable, and 
$\underline{\dim}X=(x_1,\ldots,x_n)$ the dimension vector of $X$. 
Then, Proposition \ref{cw} 
implies $$q(\underline{\dim}X) \leq 2\dim\Hom_\Pi(X,X)=2.$$

Therefore we have 
$$x_1^2+(x_1-x_2)^2+\cdots+(x_{n-1}-x_n)^2+x_n^2\leq2.$$

We can easily check that neither $q(\underline{\dim}X)=0$ nor  $q(\underline{\dim}X)=1$ occur, and hence we have $q(\underline{\dim}X)=2$. 
Thus we get $\Ext^1_\Pi(X,X)=0$ and (2) follows.  
Moreover, by $x_1^2+(x_1-x_2)^2+\cdots+(x_{n-1}-x_n)^2+x_n^2=2$, 
$x_1$ and $x_n$ are both 0 or 1. 
Then it is easy to check that $x_i$ is also 0 or 1 for any $2\leq i\leq n-1$ and hence $\underline{\dim}X$ has entries only 0 or 1. 
Since $X$ is indecomposable, any non-zero entry is adjacent to each other, that is, $\underline{\dim}X$ has the form 
$(0,\ldots,0,1,\ldots,1,0,\ldots,0)$ possibly without zeroes. 
Therefore it is isomorphic to $S(\alpha)$ for some arc $\alpha$. 
Finally (3) follows from (1) and  Theorem \ref{finite brick}.
\end{proof}

\begin{remk}
A classification of $\sbrick\Pi$ is also formulated in \cite{A2,DIRRT} in a different way. 
\end{remk}

\subsection{Hom-spaces of arc modules}

In this subsection, we study hom-spaces of arc modules in terms of combinatorics of arcs. Moreover, we establish a bijection between the set of noncrossing arc diagrams and the set of semibricks.

\begin{lemm}\label{hom lemm}
Let $\alpha$ and $\beta$ be arcs. Assume that $\alpha$ and $\beta$ do not intersect.
\begin{itemize} 
\item[(1)] Assume that $\alpha$ and $\beta$ share the same left endpoint but not the right endpoint.
Then we have $\dim\Hom_\Pi(S(\alpha),S(\beta))=0$ or $\dim\Hom_\Pi(S(\alpha),S(\beta))=1$. 
Moreover, $\dim\Hom_\Pi(S(\alpha),S(\beta))=0$ (resp. $\dim\Hom_\Pi(S(\alpha),S(\beta))=1$) if and only if $\dim\Hom_\Pi(S(\beta),S(\alpha))=1$ (resp. $\dim\Hom_\Pi(S(\beta),S(\alpha))=0$).

\item[(2)]
Assume that $\alpha$ and $\beta$ share the same right endpoint but not the left endpoint.
Then we have $\dim\Hom_\Pi(S(\alpha),S(\beta))=0$ or $\dim\Hom_\Pi(S(\alpha),S(\beta))=1$. 
Moreover, $\dim\Hom_\Pi(S(\alpha),S(\beta))=0$ (resp. $\dim\Hom_\Pi(S(\alpha),S(\beta))=1$) if and only if $\dim\Hom_\Pi(S(\beta),S(\alpha))=1$ (resp. $\dim\Hom_\Pi(S(\beta),S(\alpha))=0$). 
\item[(3)] Assume that the right endpoint of $\alpha$ coincides with 
the left endpoint of $\beta$, or left endpoint of $\alpha$ coincides with 
the right endpoint of $\beta$. 
Then we have 
$$\dim\Hom_\Pi(S(\alpha),S(\beta))=0=\dim\Hom_\Pi(S(\beta),S(\alpha)).$$

\item[(4)] Assume that $\alpha$ and $\beta$ share both the right endpoint and  the left endpoint. 
Then we have $$\dim\Hom_\Pi(S(\alpha),S(\beta))\neq 0 \ \textnormal{and}\ 
\dim\Hom_\Pi(S(\alpha),S(\beta))\neq 0.$$ 
\end{itemize}
\end{lemm}

\begin{proof}
We will show (1); (2) is similar.

We assume that $\alpha$ and $\beta$ share the left endpoint $i$. 
Since $\alpha$ and $\beta$ does not intersect, 
there exists $i<j$ such that $\alpha$ and $\beta$ are isotopy 
between the interval $v_i$ and $v_{j-1}$ 
satisfying one of the following situations 
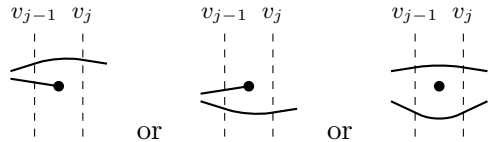
\[\begin{tikzpicture}
      [ mycell/.style={draw, minimum size=1em},
        dot/.style={mycell,
            append after command={\pgfextra \fill (\tikzlastnode) circle[radius=.2em]; \endpgfextra}}]
      \begin{scope}[every node/.style={circle, fill=black, inner sep=.5mm, outer sep=0}]
        \node[right=5mm of 1] (2) {};
      \end{scope}

      \begin{scope}[every path/.style={dashed}, every node/.style={font=\footnotesize}]
        \draw ($(1)!0.5!(2) + (0,7mm)$) node[above] {$v_{j-1}$} -- +(0,-14mm);
        \draw ($(2)!0.5!(3) + (0,7mm)$) node[above] {$v_j$} -- +(0,-14mm);
      
      \end{scope}
      \begin{scope}
        [thick, rounded corners=8pt]
    \draw   ($(1) + (0,2mm)$)-- ($(2) + (0,4mm)$) --($(3)+ (0,3mm)$);
    \draw   ($(1)+ (0,1mm)$)--($(2)$);
      \end{scope} 
      \end{tikzpicture}
  \quad \mathrm{or} \quad  
  \begin{tikzpicture}
      [ mycell/.style={draw, minimum size=1em},
        dot/.style={mycell,
            append after command={\pgfextra \fill (\tikzlastnode) circle[radius=.2em]; \endpgfextra}}]

      \begin{scope}[every node/.style={circle, fill=black, inner sep=.5mm, outer sep=0}]
        \node[right=5mm of 1] (2) {};
      \end{scope}

      \begin{scope}[every path/.style={dashed}, every node/.style={font=\footnotesize}]
        \draw ($(1)!0.5!(2) + (0,7mm)$) node[above] {$v_{j-1}$} -- +(0,-14mm);
        \draw ($(2)!0.5!(3) + (0,7mm)$) node[above] {$v_j$} -- +(0,-14mm);
      
      \end{scope}
      \begin{scope}
        [thick, rounded corners=8pt]
    \draw   ($(1) - (0,2mm)$)-- ($(2) - (0,4mm)$) --($(3)- (0,3mm)$);
    \draw   ($(1)- (0,1mm)$)--($(2)$);
      \end{scope} 
      \end{tikzpicture}
  \quad  \mathrm{or} \quad      
      \begin{tikzpicture}
      [ mycell/.style={draw, minimum size=1em},
        dot/.style={mycell,
            append after command={\pgfextra \fill (\tikzlastnode) circle[radius=.2em]; \endpgfextra}}]

      \begin{scope}[every node/.style={circle, fill=black, inner sep=.5mm, outer sep=0}]
        \node[right=5mm of 1] (2) {};
      \end{scope}

      \begin{scope}[every path/.style={dashed}, every node/.style={font=\footnotesize}]
        \draw ($(1)!0.5!(2) + (0,7mm)$) node[above] {$v_{j-1}$} -- +(0,-14mm);
        \draw ($(2)!0.5!(3) + (0,7mm)$) node[above] {$v_j$} -- +(0,-14mm);
      
      \end{scope}
      \begin{scope}
        [thick, rounded corners=8pt]
  \draw      
    ($(1) - (0,2mm)$)-- ($(2) - (0,5mm)$) --($(3)- (0,2mm)$);
    \draw        ($(1) + (0,2mm)$)-- ($(2) + (0,3mm)$) --($(3)+ (0,2mm)$);
      \end{scope} 
      \end{tikzpicture}\]
      
{\bf (Case 1).} 
Consider the following case. 

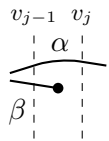
\[\begin{tikzpicture}
      [ mycell/.style={draw, minimum size=1em},
        dot/.style={mycell,
            append after command={\pgfextra \fill (\tikzlastnode) circle[radius=.2em]; \endpgfextra}}]
      \begin{scope}[every node/.style={circle, fill=black, inner sep=.5mm, outer sep=0}]
        \node[right=5mm of 1] (2) {};
      \end{scope}

      \begin{scope}[every path/.style={dashed}, every node/.style={font=\footnotesize}]
        \draw ($(1)!0.5!(2) + (0,7mm)$) node[above] {$v_{j-1}$} -- +(0,-14mm);
        \draw ($(2)!0.5!(3) + (0,7mm)$) node[above] {$v_j$} -- +(0,-14mm);
      
      \end{scope}
      \begin{scope}
        [thick, rounded corners=8pt]
   \coordinate (a) at (0.66666,0.8) node at (a) [below=0] {$\alpha$};
    \coordinate (b) at (0.1,0) node at (b) [below=0] {$\beta$};
    \draw   ($(1) + (0,2mm)$)-- ($(2) + (0,4mm)$) --($(3)+ (0,3mm)$);
    \draw   ($(1)+ (0,1mm)$)--($(2)$);
      \end{scope} 
      \end{tikzpicture}\]
Then, by Theorem \ref{CWthm}, we have an inclusion $S(\beta)\to S(\alpha)$ and there is no graph map from $S(\alpha)$ to $S(\beta)$. 
Thus 
we have $\dim\Hom_\Pi(S(\alpha),S(\beta))=0$ and $\dim\Hom_\Pi(S(\beta),S(\alpha))=1$. 
Clearly, if we replace $\alpha$ and $\beta$, then we have 
$\dim\Hom_\Pi(S(\beta),S(\alpha))=0$ and $\dim\Hom_\Pi(S(\alpha),S(\beta))=1$.

{\bf (Case 2).} 
Consider the following case.
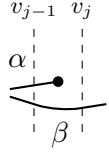
\[\begin{tikzpicture}
      [ mycell/.style={draw, minimum size=1em},
        dot/.style={mycell,
            append after command={\pgfextra \fill (\tikzlastnode) circle[radius=.2em]; \endpgfextra}}]

      \begin{scope}[every node/.style={circle, fill=black, inner sep=.5mm, outer sep=0}]
        \node[right=5mm of 1] (2) {};
      \end{scope}

      \begin{scope}[every path/.style={dashed}, every node/.style={font=\footnotesize}]
        \draw ($(1)!0.5!(2) + (0,7mm)$) node[above] {$v_{j-1}$} -- +(0,-14mm);
        \draw ($(2)!0.5!(3) + (0,7mm)$) node[above] {$v_j$} -- +(0,-14mm);
      
      \end{scope}
      \begin{scope}
        [thick, rounded corners=8pt]
      \coordinate (a) at (0.66666,-0.4) node at (a) [below=0] {$\beta$};
    \coordinate (b) at (0.1,0.5) node at (b) [below=0] {$\alpha$};    
    \draw   ($(1) - (0,2mm)$)-- ($(2) - (0,4mm)$) --($(3)- (0,3mm)$);
    \draw   ($(1)- (0,1mm)$)--($(2)$);
      \end{scope} 
      \end{tikzpicture}\]
Then, by Theorem \ref{CWthm}, we have an surjection $S(\beta)\to S(\alpha)$ and there is no graph map from $S(\alpha)$ to $S(\beta)$. 
Thus 
we have $\dim\Hom_\Pi(S(\alpha),S(\beta))=0$ and $\dim\Hom_\Pi(S(\beta),S(\alpha))=1$. 
Clearly, if we replace $\alpha$ and $\beta$, then we have 
$\dim\Hom_\Pi(S(\beta),S(\alpha))=0$ and $\dim\Hom_\Pi(S(\alpha),S(\beta))=1$.

{\bf (Case 3).} 
Consider the following case.
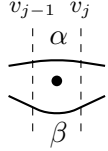
\[\begin{tikzpicture}
      [ mycell/.style={draw, minimum size=1em},
        dot/.style={mycell,
            append after command={\pgfextra \fill (\tikzlastnode) circle[radius=.2em]; \endpgfextra}}]

      \begin{scope}[every node/.style={circle, fill=black, inner sep=.5mm, outer sep=0}]
        \node[right=5mm of 1] (2) {};
      \end{scope}

      \begin{scope}[every path/.style={dashed}, every node/.style={font=\footnotesize}]
        \draw ($(1)!0.5!(2) + (0,7mm)$) node[above] {$v_{j-1}$} -- +(0,-14mm);
        \draw ($(2)!0.5!(3) + (0,7mm)$) node[above] {$v_j$} -- +(0,-14mm);
      
      \end{scope}
      \begin{scope}
        [thick, rounded corners=8pt]
        \coordinate (a) at (0.66666,-0.4) node at (a) [below=0] {$\beta$};
    \coordinate (b) at (0.66666,0.8) node at (b) [below=0] {$\alpha$};      
  \draw      
    ($(1) - (0,2mm)$)-- ($(2) - (0,5mm)$) --($(3)- (0,2mm)$);
    \draw        ($(1) + (0,2mm)$)-- ($(2) + (0,3mm)$) --($(3)+ (0,2mm)$);
      \end{scope} 
      \end{tikzpicture}\]

Then, for the arrow sequence of $\alpha$ and $\beta$,  
we have factorizations $\alpha=ba_{j-1}^{-}d$ and  $\beta=ba_{j-1}d'$. 
Then we have 
$$\xymatrix{
S(\beta) \ar@{->>}[r]
& S(b) \ar@{^{(}->}[r] & S(\alpha).
}$$ 
Moreover, because $\alpha$ and $\beta$ do not share the right endpoint, the right endpoint of $\alpha$ is bigger or smaller than the right endpoint of $\beta$. Therefore, there exists $k>j$ such that

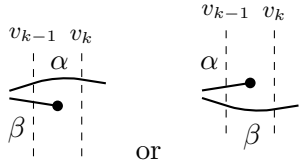
\[\begin{tikzpicture}
      [ mycell/.style={draw, minimum size=1em},
        dot/.style={mycell,
            append after command={\pgfextra \fill (\tikzlastnode) circle[radius=.2em]; \endpgfextra}}]
      \begin{scope}[every node/.style={circle, fill=black, inner sep=.5mm, outer sep=0}]
        \node[right=5mm of 1] (2) {};
      \end{scope}

      \begin{scope}[every path/.style={dashed}, every node/.style={font=\footnotesize}]
        \draw ($(1)!0.5!(2) + (0,7mm)$) node[above] {$v_{k-1}$} -- +(0,-14mm);
        \draw ($(2)!0.5!(3) + (0,7mm)$) node[above] {$v_k$} -- +(0,-14mm);
      
      \end{scope}
      \begin{scope}
        [thick, rounded corners=8pt]
   \coordinate (a) at (0.66666,0.8) node at (a) [below=0] {$\alpha$};
    \coordinate (b) at (0.1,0) node at (b) [below=0] {$\beta$};
    \draw   ($(1) + (0,2mm)$)-- ($(2) + (0,4mm)$) --($(3)+ (0,3mm)$);
    \draw   ($(1)+ (0,1mm)$)--($(2)$);
      \end{scope} 
      \end{tikzpicture}  \quad \mathrm{or} \quad  
      \begin{tikzpicture}
      [ mycell/.style={draw, minimum size=1em},
        dot/.style={mycell,
            append after command={\pgfextra \fill (\tikzlastnode) circle[radius=.2em]; \endpgfextra}}]

      \begin{scope}[every node/.style={circle, fill=black, inner sep=.5mm, outer sep=0}]
        \node[right=5mm of 1] (2) {};
      \end{scope}

      \begin{scope}[every path/.style={dashed}, every node/.style={font=\footnotesize}]
        \draw ($(1)!0.5!(2) + (0,7mm)$) node[above] {$v_{k-1}$} -- +(0,-14mm);
        \draw ($(2)!0.5!(3) + (0,7mm)$) node[above] {$v_k$} -- +(0,-14mm);
      
      \end{scope}
      \begin{scope}
        [thick, rounded corners=8pt]
      \coordinate (a) at (0.66666,-0.4) node at (a) [below=0] {$\beta$};
    \coordinate (b) at (0.1,0.5) node at (b) [below=0] {$\alpha$};    
    \draw   ($(1) - (0,2mm)$)-- ($(2) - (0,4mm)$) --($(3)- (0,3mm)$);
    \draw   ($(1)- (0,1mm)$)--($(2)$);
      \end{scope} 
      \end{tikzpicture}\]

Then we can check that the above map $S(\beta) \twoheadrightarrow S(b) \hookrightarrow S(\alpha)$ is a unique graph map and there is no graph map from $S(\alpha)$ to $S(\beta)$. 
Thus we have $\dim\Hom_\Pi(S(\beta),S(\alpha))=1$ and 
$\dim\Hom_\Pi(S(\alpha),S(\beta))=0$ by Theorem \ref{CWthm}. Clearly, if we replace $\alpha$ and $\beta$, then we have 
$\dim\Hom_\Pi(S(\beta),S(\alpha))=0$ and $\dim\Hom_\Pi(S(\alpha),S(\beta))=1$.

Using Theorem \ref{CWthm}, it is easy to check 
(3) and (4) by the same argument.
\end{proof}

Next we give the following simple observation.  

\begin{prop}\label{noncrossing-hom}
Two arcs $\alpha$ and $\beta$ are noncrossing (i.e. $\alpha$ and $\beta$ satisfy (nc1),(nc2)) if and only if  $\dim\Hom_\Pi(S(\alpha),S(\beta))=0=\dim\Hom_\Pi(S(\beta),S(\alpha))$.
\end{prop}

\begin{proof}
The only if part follows from Lemma \ref{hom lemm} (3). 
We will show the if part. 
Assume that $\alpha$ and $\beta$ are crossing. 

{\bf (nc1).} 
First we assume that  $\alpha$ and $\beta$ intersect between the points $i$ and $i+1$. 
Then it is easy to check that 
the simple module $S_{v_i}$ associated to the vertex $v_i$ 
belongs to one of the top of $S(\alpha)$ and $S(\beta)$. 
We assume that it belongs to the top of $S(\alpha)$. 
Then since $\alpha$ and $\beta$ intersect, $S_{v_i}$ belongs to 
the socle of $S(\beta)$. 
Therefore have $\dim\Hom_\Pi(S(\alpha),S(\beta))\neq 0.$ Clearly, if we replace $\alpha$ and $\beta$, then we have $\dim\Hom_\Pi(S(\beta),S(\alpha))\neq 0.$

{\bf (nc2).} Next we assume that $\alpha$ and $\beta$ share either the same right endpoint or the same left endpoint, or both. 
Then Lemma \ref{hom lemm} implies that $\dim\Hom_\Pi(S(\alpha),S(\beta))\neq0$ or $\dim\Hom_\Pi(S(\beta),S(\alpha))\neq0$. 
\end{proof}

\begin{remk}
We remark that Proposition \ref{noncrossing-hom} is essentially identified with \cite[Lemma 4.2.8]{BCZ}, where the authors studied some quotient algebra $\Pi/I$ of $\Pi$. Since there exists a natural fully faithful functor from $\mod(\Pi/I)$ to $\mod\Pi$, we can apply this result to our case.
\end{remk}


For $\G(w)\in\NAD$, we define the map $\Phi$ as follows 
$$\Phi(\G(w))=\{S(\alpha)\ |\ \alpha\in\G(w) \}\subset\mod\Pi.$$


\begin{exam}
Let $w=53271468$. Then, by Example \ref{exam1}, 
we have the following green arc diagram $\G(w)$

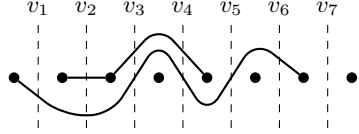
\[ \begin{tikzpicture}
      [ mycell/.style={draw, minimum size=1em},
        dot/.style={mycell,
            append after command={\pgfextra \fill (\tikzlastnode) circle[radius=.2em]; \endpgfextra}}]

      \begin{scope}[every node/.style={circle, fill=black, inner sep=.5mm, outer sep=0}]
       \node (1) {};
        \node[right=5mm of 1] (2) {};
        \node[right=5mm of 2] (3) {};
        \node[right=5mm of 3] (4) {};
        \node[right=5mm of 4] (5) {};
        \node[right=5mm of 5] (6) {};
        \node[right=5mm of 6] (7) {};
        \node[right=5mm of 7] (8) {};
      \end{scope}
      \begin{scope}
        [thick, rounded corners=8pt]
        \draw
        (1)-- ($(2) - (0,5mm)$) -- ($(3) - (0,5mm)$) -- ($(4) + (0,5mm)$) -- 
        ($(5) - (0,5mm)$) -- ($(6) + (0,5mm)$) -- (7) ;
         \draw(2)-- (3); 
        \draw        (3)-- ($(4) + (0,7mm)$) -- (5); 
      \end{scope}
      \begin{scope}[every path/.style={dashed}, every node/.style={font=\footnotesize}]
        \draw ($(1)!0.5!(2) + (0,7mm)$) node[above] {$v_1$} -- +(0,-14mm);
        \draw ($(2)!0.5!(3) + (0,7mm)$) node[above] {$v_2$} -- +(0,-14mm);
        \draw ($(3)!0.5!(4) + (0,7mm)$) node[above] {$v_3$} -- +(0,-14mm);
        \draw ($(4)!0.5!(5) + (0,7mm)$) node[above] {$v_4$} -- +(0,-14mm);
        \draw ($(5)!0.5!(6) + (0,7mm)$) node[above] {$v_5$} -- +(0,-14mm);
        \draw ($(6)!0.5!(7) + (0,7mm)$) node[above] {$v_6$} -- +(0,-14mm);
        \draw ($(7)!0.5!(8) + (0,7mm)$) node[above] {$v_7$} -- +(0,-14mm);
      \end{scope}
      \end{tikzpicture} \]
Then $\Phi(\G(w))$ consist of three indecomposable $\Pi$-modules 

$${\begin{smallmatrix}
  v_1&&&&&&&\\
  &v_2&&v_4&&v_6&\\
  &&v_3&&v_5&&\\
  &&&&&&
  \end{smallmatrix}}\oplus
  {\begin{smallmatrix}
  v_2
  \end{smallmatrix}}\oplus{\begin{smallmatrix}
  &&v_4&\\
  &v_3&&
  \end{smallmatrix}}.$$
\end{exam}

Using the above notion, we have the following consequence, which provides a bijection between $W$ and $\sbrick\Pi$.

\begin{thm}\label{bij nad}
We have a bijection 
$$\Phi:\NAD\to\sbrick\Pi.$$
In particular, there exists a bijection $W$ and $\sbrick\Pi$.
\end{thm}

\begin{proof}
From Propositions \ref{arc-brick bij} and 
\ref{noncrossing-hom}, the map is well-defined and surjective, and it is clearly injective. 
The second statement follows from Theorem \ref{Reading}.
\end{proof}

We prepare the following lemma for later use.

\begin{lemm}\label{ext lemm}
Assume that two arcs $\alpha$ and $\beta$ do not intersect nor share any endpoint. Then we have 
$$\dim\Ext^1_\Pi(S(\alpha),S(\beta))=0=\dim\Ext^1_\Pi(S(\beta),S(\alpha)).$$
\end{lemm}

\begin{proof}
Since the dimension vectors of $S(\alpha)$ and $S(\beta)$ have entries only 0 or 1, we have $(\underline{\dim}S(\alpha),\underline{\dim}S(\beta))=0$ by the assumption. 
On the other hand, by Lemma \ref{hom lemm} (3), we have 
$\dim\Hom_\Pi(S(\alpha),S(\beta))=0=\dim\Hom_\Pi(S(\beta),S(\alpha))$. 
Therefore, by Proposition \ref{cw}, we have 
$$0=(\underline{\dim}S(\alpha),\underline{\dim}S(\beta))=-\dim\Ext_\Pi^1(S(\alpha),S(\beta)).$$ 
Thus, we get  $\dim\Ext^1_\Pi(S(\alpha),S(\beta))=0$ and, similarly, $\dim\Ext^1_\Pi(S(\beta),S(\alpha))=0$. 
\end{proof}


\subsection{A partial order and mutation}\label{mutation}
In this subsection, we discuss mutation and partial orders 
of arc diagrams. 
We define mutation of $\DAD$ such that the left action of a transposition on $w$ is compatible with the mutation of the corresponding double arc diagrams. 
Moreover we will show the compatibility of the mutation of double arc diagrams and 2-term SMCs. 
Using this result, we show a poset isomorphism $W\to\twosmc\Pi$.

For our purpose, we give a more precise correspondence $\D:W\to\DAD.$ 
Let $w=(w_1w_2\cdots w_{n+1})\in W$. By the map $\D:W\to\DAD$, 
for a pair $(w_i,w_{i+1})$, we have an arc whose endpoints are $w_i$ and $w_{i+1}$. We denote this arc by $\alpha_{w_iw_{i+1}}$ (note that the arc $\alpha_{w_iw_{i+1}}$ depends not only on  $(w_i,w_{i+1})$ but also on $(w_1w_2\cdots w_{n+1})$).

Recall that $\alpha_{w_iw_{i+1}}$ is green (resp. red) if $w_i>w_{i+1}$ 
(resp. $w_i<w_{i+1}$). 
For simplicity, we write 
$\alpha_{w_iw_{i+1}}(0)$ (resp. $\alpha_{w_iw_{i+1}}(1)$) if $\alpha_{w_iw_{i+1}}$ is 
green (resp. red). 
We denote by $c_i\in\{0,1\}$ 
the color of $\alpha_{w_iw_{i+1}}$. 
In this notation, we can write 
$$\D(w) = \{\alpha_{w_1w_{2}}(c_1),\alpha_{w_2w_{3}}(c_2),\cdots,\alpha_{w_{n}w_{n+1}}(c_{n})\},$$

and we define  
$\Psi:\DAD\to\{\textnormal{the set of 2-term complexes of }\Db(\mod\Pi)\}$ $$\Psi(\D(w))=\Psi(\{\alpha_{w_{i}w_{i+1}}(c_i)\})=\{S(\alpha_{w_{i}w_{i+1}})[c_i]\}\  \ \ ({1\leq i\leq n}).$$

It is also written as 
$\Psi(\D(w))=\Phi(\G(w))\sqcup\Phi(\R(w))[1]$. 
By abuse of notation, for an arc $\alpha_{w_{i}w_{i+1}}(c_i)$ of $\D(w)$, we denote by $$\Psi(\alpha_{w_{i}w_{i+1}}(c_i))=S(\alpha_{w_{i}w_{i+1}})[c_i]$$ the corresponding 2-term indecomposable complex obtained by $\Psi$. 
We will show that $\Psi$ gives a bijection between $\DAD$ and $\twosmc\Pi$.


First,  
we will define a mutation on $\DAD$. 
This is defined by applying a \emph{half-twist} to a diagram.  To avoid an confusion, 
we give another formulation here.  

We introduce the following terminology.
For an arc $\alpha$, we denote by $L(\alpha)\in\{1,2,\cdots,n\}$ the left endpoint of $\alpha$ and $R(\alpha)\in\{2,3,\cdots,n+1\}$ the right endpoint of $\alpha$. 
Recall that, for an arc $\alpha$ and $1\leq i\leq n$, we denote by $\alpha|_{v_{i}}$ (resp.  ${}_{v_{i}}|\alpha$,  ${}_{v_{i}}|\alpha|_{v_{i+1}}$) the left segment (resp. the right segment, the central segment) of $\alpha$ cut by the line $v_i$ (resp. $v_i$,  $v_i$ and $v_{i+1}$).
Then we define the following transformation of arcs.

\begin{defi}
Fix $w\in W$. 
Let $\alpha,\beta$ ($\alpha\neq\beta$) be arcs of $\D(w)$. 
Assume that $\alpha$ and $\beta$ share one of the endpoint $i$.  
Then, by transforming $\beta$, we define a new arc $\gamma=\sigma(\beta;\alpha)$ as follows.

\begin{itemize}
\item[(1-i)] 
Assume that $R(\alpha)=L(\beta)=i$. 
Then we define $\gamma$, where $\gamma|_{v_{i-1}}$ is isotopy to  $\alpha|_{v_{i-1}}$,   ${}_{v_{i}}|\gamma$ is isotopy to  ${}_{v_{i}}|\beta$ and  ${}_{v_{i-1}}|\gamma|_{v_{i}}$ is the arc above to the point $i$.
This is illustrated by the following example.

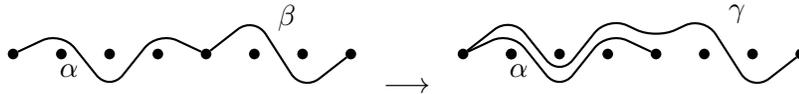
\[\begin{tikzpicture}
      [ mycell/.style={draw, minimum size=1em},
        dot/.style={mycell,
            append after command={\pgfextra \fill (\tikzlastnode) circle[radius=.2em]; \endpgfextra}}]
      \begin{scope}[every node/.style={circle, fill=black, inner sep=.5mm, outer sep=0}]
        \node[left=5mm of 1] (-1) {};
        \node (0) {};
        \node[right=5mm of 0] (1) {};
        \node[right=5mm of 1] (2) {};
        \node[right=5mm of 2] (3) {};
        \node[right=5mm of 3] (4) {};
        \node[right=5mm of 4] (5) {};
         \node[right=5mm of 5] (6) {};
      \end{scope}
      \begin{scope}
        [thick, rounded corners=8pt]
        \coordinate (a) at (3,0.8) node at (a) [below=0] {$\beta$};
    \coordinate (b) at (0.1,0) node at (b) [below=0] {$\alpha$};
     \draw (-1)--($(0) + (0,3mm)$)-- ($(1) - (0,5mm)$)-- ($(2) + (0,3mm)$) -- (3) ;
        \draw  (3)-- ($(4) + (0,5mm)$) --($(5) - (0,5mm)$) --(6) ;
      \end{scope}
      \end{tikzpicture}\ \ \longrightarrow \ \ 
      \begin{tikzpicture}
      [ mycell/.style={draw, minimum size=1em},
        dot/.style={mycell,
            append after command={\pgfextra \fill (\tikzlastnode) circle[radius=.2em]; \endpgfextra}}]
      \begin{scope}[every node/.style={circle, fill=black, inner sep=.5mm, outer sep=0}]
      \node[left=5mm of 1] (-2) {};
        \node[left=5mm of 0] (-1) {};
        \node (0) {};
        \node[right=5mm of 0] (1) {};
        \node[right=5mm of 1] (2) {};
        \node[right=5mm of 2] (3) {};
        \node[right=5mm of 3] (4) {};
        \node[right=5mm of 4] (5) {};
         \node[right=5mm of 5] (6) {};
      \end{scope}
      \begin{scope}
        [thick, rounded corners=8pt]
         \coordinate (a) at (3,0.8) node at (a) [below=0] {$\gamma$};
          \coordinate (b) at (0.1,0) node at (b) [below=0] {$\alpha$};
     \draw (-1)--($(0) + (0,3mm)$)-- ($(1) - (0,5mm)$)-- ($(2) + (0,3mm)$) -- (3) ;
        \draw  (-1)--($(0) + (0,5mm)$)-- ($(1) - (0,3mm)$)-- ($(2) + (0,5mm)$) --($(3) + (0,2mm)$)-- ($(4) + (0,5mm)$) --($(5) - (0,5mm)$) --(6) ;
      \end{scope}
      \end{tikzpicture}\] 
We remark that we have a natural interpretation as an exact sequence. 
For example, if $\alpha$ and $\beta$ are green, then we have the exact sequence
$$\xymatrix{0\ar[r]&
S(\alpha) \ar[r]^{\ \ }
& S(\gamma) \ar[r] & S(\beta)\ar[r]&0.
}$$ 

\item[(1-ii)]
If $L(\alpha)>L(\beta)$, then 
we define $\gamma$ such that $\gamma|_{v_{i-1}}$ is isotopy to  $\beta|_{v_{i-1}}$ and $i$ is the right endpoint of $\gamma$.
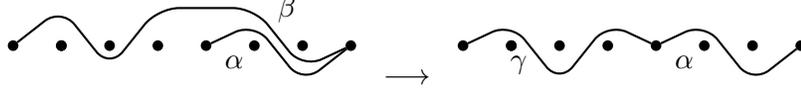
\[   \begin{tikzpicture}
      [ mycell/.style={draw, minimum size=1em},
        dot/.style={mycell,
            append after command={\pgfextra \fill (\tikzlastnode) circle[radius=.2em]; \endpgfextra}}]
      \begin{scope}[every node/.style={circle, fill=black, inner sep=.5mm, outer sep=0}]
      \node[left=5mm of 1] (-2) {};
        \node[left=5mm of 0] (-1) {};
        \node (0) {};
        \node[right=5mm of 0] (1) {};
        \node[right=5mm of 1] (2) {};
        \node[right=5mm of 2] (3) {};
        \node[right=5mm of 3] (4) {};
        \node[right=5mm of 4] (5) {};
         \node[right=5mm of 5] (6) {};
      \end{scope}
      \begin{scope}
        [thick, rounded corners=8pt]
\coordinate (a) at (3,0.8) node at (a) [below=0] {$\beta$};
    \coordinate (b) at (2.3,0) node at (b) [below=0] {$\alpha$};        
 \draw  (-1)--($(0) + (0,5mm)$)-- ($(1) - (0,3mm)$)-- ($(2) + (0,5mm)$) -- ($(3) + (0,5mm)$)-- ($(4) + (0,5mm)$) --($(5) - (0,3mm)$) --(6);
        \draw  (3)-- ($(4) + (0,3mm)$) --($(5) - (0,5mm)$) --(6) ;
      \end{scope}
      \end{tikzpicture}\ \ \longrightarrow \ \ 
      \begin{tikzpicture}
      [ mycell/.style={draw, minimum size=1em},
        dot/.style={mycell,
            append after command={\pgfextra \fill (\tikzlastnode) circle[radius=.2em]; \endpgfextra}}]
      \begin{scope}[every node/.style={circle, fill=black, inner sep=.5mm, outer sep=0}]
          \node[left=5mm of 1] (-2) {};
        \node[left=5mm of 0] (-1) {};
        \node (0) {};
        \node[right=5mm of 0] (1) {};
        \node[right=5mm of 1] (2) {};
        \node[right=5mm of 2] (3) {};
        \node[right=5mm of 3] (4) {};
        \node[right=5mm of 4] (5) {};
         \node[right=5mm of 5] (6) {};
      \end{scope}
      \begin{scope}
        [thick, rounded corners=8pt]
   \coordinate (a) at (2.3,0) node at (a) [below=0] {$\alpha$};
    \coordinate (b) at (0.1,0) node at (b) [below=0] {$\gamma$};
     \draw (-1)--($(0) + (0,3mm)$)-- ($(1) - (0,5mm)$)-- ($(2) + (0,3mm)$) -- (3) ;
        \draw  
     (3)-- ($(4) + (0,3mm)$) --($(5) - (0,5mm)$) --(6)     ;
      \end{scope}
      \end{tikzpicture}\]


\item[(1-iii)]
If $L(\alpha)<L(\beta)$, then 
we define $\gamma$, where $\gamma|_{v_{i-1}}$ is isotopy to  $\alpha|_{v_{i-1}}$ and $i$ is the right endpoint of $\gamma$.
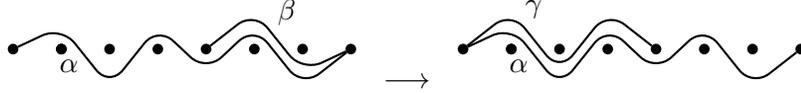
\[   \begin{tikzpicture}
      [ mycell/.style={draw, minimum size=1em},
        dot/.style={mycell,
            append after command={\pgfextra \fill (\tikzlastnode) circle[radius=.2em]; \endpgfextra}}]
      \begin{scope}[every node/.style={circle, fill=black, inner sep=.5mm, outer sep=0}]
      \node[left=5mm of 1] (-2) {};
        \node[left=5mm of 0] (-1) {};
        \node (0) {};
        \node[right=5mm of 0] (1) {};
        \node[right=5mm of 1] (2) {};
        \node[right=5mm of 2] (3) {};
        \node[right=5mm of 3] (4) {};
        \node[right=5mm of 4] (5) {};
         \node[right=5mm of 5] (6) {};
      \end{scope}
      \begin{scope}
        [thick, rounded corners=8pt]
        \coordinate (a) at (3,0.8) node at (a) [below=0] {$\beta$};
    \coordinate (b) at (0.1,0) node at (b) [below=0] {$\alpha$};
 \draw (-1)--($(0) + (0,3mm)$)-- ($(1) - (0,5mm)$)-- ($(2) + (0,3mm)$) -- ($(3) - (0,3mm)$)-- ($(4) + (0,3mm)$) --($(5) - (0,5mm)$) --(6);
        \draw  (3)-- ($(4) + (0,5mm)$) --($(5) - (0,3mm)$) --(6) ;
      \end{scope}
      \end{tikzpicture}\ \ \longrightarrow \ \ 
      \begin{tikzpicture}
      [ mycell/.style={draw, minimum size=1em},
        dot/.style={mycell,
            append after command={\pgfextra \fill (\tikzlastnode) circle[radius=.2em]; \endpgfextra}}]
      \begin{scope}[every node/.style={circle, fill=black, inner sep=.5mm, outer sep=0}]
      \node[left=5mm of 1] (-2) {};
        \node[left=5mm of 0] (-1) {};
        \node (0) {};
        \node[right=5mm of 0] (1) {};
        \node[right=5mm of 1] (2) {};
        \node[right=5mm of 2] (3) {};
        \node[right=5mm of 3] (4) {};
        \node[right=5mm of 4] (5) {};
         \node[right=5mm of 5] (6) {};
      \end{scope}
      \begin{scope}
        [thick, rounded corners=8pt]
        \coordinate (a) at (0.3,0.8) node at (a) [below=0] {$\gamma$};
            \coordinate (b) at (0.1,0) node at (b) [below=0] {$\alpha$};
 \draw (-1)--($(0) + (0,3mm)$)-- ($(1) - (0,5mm)$)-- ($(2) + (0,3mm)$) -- ($(3) - (0,3mm)$)-- ($(4) + (0,3mm)$) --($(5) - (0,5mm)$) --(6);
        \draw (-1)--($(0) + (0,5mm)$)-- ($(1) - (0,3mm)$)-- ($(2) + (0,5mm)$) -- (3);
      \end{scope}
      \end{tikzpicture}\]


\item[(2-i)] Assume that $L(\alpha)=R(\beta)=i$. 
Then we define $\gamma$, where $\gamma|_{v_{i-1}}$ is isotopy to  $\beta|_{v_{i-1}}$,   ${}_{v_{i}}|\gamma$ is isotopy to  ${}_{v_{i}}|\alpha$ and  ${}_{v_{i-1}}|\gamma|_{v_{i}}$ is the arc below to the point $i$.

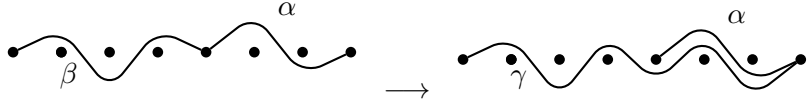
\[\begin{tikzpicture}
      [ mycell/.style={draw, minimum size=1em},
        dot/.style={mycell,
            append after command={\pgfextra \fill (\tikzlastnode) circle[radius=.2em]; \endpgfextra}}]
      \begin{scope}[every node/.style={circle, fill=black, inner sep=.5mm, outer sep=0}]
          \node[left=5mm of 1] (-2) {};
        \node[left=5mm of 0] (-1) {};
        \node (0) {};
        \node[right=5mm of 0] (1) {};
        \node[right=5mm of 1] (2) {};
        \node[right=5mm of 2] (3) {};
        \node[right=5mm of 3] (4) {};
        \node[right=5mm of 4] (5) {};
         \node[right=5mm of 5] (6) {};
      \end{scope}
      \begin{scope}
        [thick, rounded corners=8pt]
        \coordinate (a) at (3,0.8) node at (a) [below=0] {$\alpha$};
    \coordinate (b) at (0.1,0) node at (b) [below=0] {$\beta$};
     \draw (-1)--($(0) + (0,3mm)$)-- ($(1) - (0,5mm)$)-- ($(2) + (0,3mm)$) -- (3) ;
        \draw  (3)-- ($(4) + (0,5mm)$) --($(5) - (0,3mm)$) --(6) ;
      \end{scope}
      \end{tikzpicture}\ \ \longrightarrow \ \ 
      \begin{tikzpicture}
      [ mycell/.style={draw, minimum size=1em},
        dot/.style={mycell,
            append after command={\pgfextra \fill (\tikzlastnode) circle[radius=.2em]; \endpgfextra}}]
      \begin{scope}[every node/.style={circle, fill=black, inner sep=.5mm, outer sep=0}]
      \node[left=5mm of 1] (-2) {};
        \node[left=5mm of 0] (-1) {};
        \node (0) {};
        \node[right=5mm of 0] (1) {};
        \node[right=5mm of 1] (2) {};
        \node[right=5mm of 2] (3) {};
        \node[right=5mm of 3] (4) {};
        \node[right=5mm of 4] (5) {};
         \node[right=5mm of 5] (6) {};
      \end{scope}
      \begin{scope}
        [thick, rounded corners=8pt]
            \coordinate (b) at (0.1,0) node at (b) [below=0] {$\gamma$};
              \coordinate (a) at (3,0.8) node at (a) [below=0] {$\alpha$};
 \draw (-1)--($(0) + (0,3mm)$)-- ($(1) - (0,5mm)$)-- ($(2) + (0,3mm)$) -- ($(3) - (0,3mm)$)-- ($(4) + (0,3mm)$) --($(5) - (0,5mm)$) --(6);
        \draw  (3)-- ($(4) + (0,5mm)$) --($(5) - (0,3mm)$) --(6) ;
      \end{scope}
      \end{tikzpicture}\]


\item[(2-ii)]  If $R(\alpha)>R(\beta)$, then 
we define $\gamma$ such that ${}_{v_{i}}|\gamma$ is isotopy to  ${}_{v_{i}}|\alpha$ and $i$ is the left endpoint of $\gamma$.
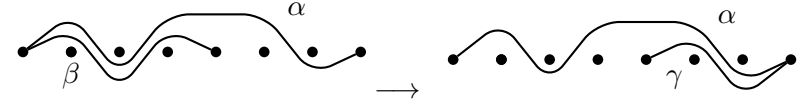
\[ \begin{tikzpicture}
      [ mycell/.style={draw, minimum size=1em},
        dot/.style={mycell,
            append after command={\pgfextra \fill (\tikzlastnode) circle[radius=.2em]; \endpgfextra}}]
      \begin{scope}[every node/.style={circle, fill=black, inner sep=.5mm, outer sep=0}]
      \node[left=5mm of 1] (-2) {};
        \node[left=5mm of 0] (-1) {};
        \node (0) {};
        \node[right=5mm of 0] (1) {};
        \node[right=5mm of 1] (2) {};
        \node[right=5mm of 2] (3) {};
        \node[right=5mm of 3] (4) {};
        \node[right=5mm of 4] (5) {};
         \node[right=5mm of 5] (6) {};
      \end{scope}
      \begin{scope}
        [thick, rounded corners=8pt]
        \coordinate (a) at (3,0.8) node at (a) [below=0] {$\alpha$};
    \coordinate (b) at (0,0) node at (b) [below=0] {$\beta$};  
 \draw (-1)--($(0) + (0,5mm)$)-- ($(1) - (0,3mm)$)-- ($(2) + (0,5mm)$) -- ($(3) + (0,5mm)$)-- ($(4) + (0,5mm)$) --($(5) - (0,3mm)$) --(6);
        \draw  (-1)--($(0) + (0,3mm)$)-- ($(1) - (0,5mm)$)-- ($(2) + (0,3mm)$) -- (3) ;
      \end{scope}
      \end{tikzpicture}\longrightarrow \ \ 
      \begin{tikzpicture}
      [ mycell/.style={draw, minimum size=1em},
        dot/.style={mycell,
            append after command={\pgfextra \fill (\tikzlastnode) circle[radius=.2em]; \endpgfextra}}]
      \begin{scope}[every node/.style={circle, fill=black, inner sep=.5mm, outer sep=0}]
      \node[left=5mm of 1] (-2) {};
        \node[left=5mm of 0] (-1) {};
        \node (0) {};
        \node[right=5mm of 0] (1) {};
        \node[right=5mm of 1] (2) {};
        \node[right=5mm of 2] (3) {};
        \node[right=5mm of 3] (4) {};
        \node[right=5mm of 4] (5) {};
         \node[right=5mm of 5] (6) {};
      \end{scope}
      \begin{scope}
        [thick, rounded corners=8pt]
\coordinate (a) at (3,0.8) node at (a) [below=0] {$\alpha$};
    \coordinate (b) at (2.3,0) node at (b) [below=0] {$\gamma$};        
 \draw  (-1)--($(0) + (0,5mm)$)-- ($(1) - (0,3mm)$)-- ($(2) + (0,5mm)$) -- ($(3) + (0,5mm)$)-- ($(4) + (0,5mm)$) --($(5) - (0,3mm)$) --(6);
        \draw  (3)-- ($(4) + (0,3mm)$) --($(5) - (0,5mm)$) --(6) ;
      \end{scope}
      \end{tikzpicture}\]


\item[(2-iii)]
If $R(\alpha)<R(\beta)$, then 
we define $\gamma$ such that ${}_{v_{i}}|\gamma$ is isotopy to  ${}_{v_{i}}|\beta$ and $i$ is the left endpoint of $\gamma$.
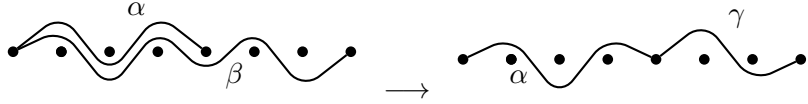
\[  \begin{tikzpicture}
      [ mycell/.style={draw, minimum size=1em},
        dot/.style={mycell,
            append after command={\pgfextra \fill (\tikzlastnode) circle[radius=.2em]; \endpgfextra}}]
      \begin{scope}[every node/.style={circle, fill=black, inner sep=.5mm, outer sep=0}]
      \node[left=5mm of 1] (-2) {};
        \node[left=5mm of 0] (-1) {};
        \node (0) {};
        \node[right=5mm of 0] (1) {};
        \node[right=5mm of 1] (2) {};
        \node[right=5mm of 2] (3) {};
        \node[right=5mm of 3] (4) {};
        \node[right=5mm of 4] (5) {};
         \node[right=5mm of 5] (6) {};
      \end{scope}
      \begin{scope}
        [thick, rounded corners=8pt]
        \coordinate (a) at (1,0.8) node at (a) [below=0] {$\alpha$};
    \coordinate (b) at (2.3,0) node at (b) [below=0] {$\beta$};  
 \draw (-1)--($(0) + (0,3mm)$)-- ($(1) - (0,5mm)$)-- ($(2) + (0,3mm)$) -- ($(3) - (0,3mm)$)-- ($(4) + (0,3mm)$) --($(5) - (0,5mm)$) --(6);
        \draw (-1)--($(0) + (0,5mm)$)-- ($(1) - (0,3mm)$)-- ($(2) + (0,5mm)$) -- (3);
      \end{scope}
      \end{tikzpicture} \ \ \longrightarrow \ \ 
      \begin{tikzpicture}
      [ mycell/.style={draw, minimum size=1em},
        dot/.style={mycell,
            append after command={\pgfextra \fill (\tikzlastnode) circle[radius=.2em]; \endpgfextra}}]
      \begin{scope}[every node/.style={circle, fill=black, inner sep=.5mm, outer sep=0}]
          \node[left=5mm of 1] (-2) {};
        \node[left=5mm of 0] (-1) {};
        \node (0) {};
        \node[right=5mm of 0] (1) {};
        \node[right=5mm of 1] (2) {};
        \node[right=5mm of 2] (3) {};
        \node[right=5mm of 3] (4) {};
        \node[right=5mm of 4] (5) {};
         \node[right=5mm of 5] (6) {};
      \end{scope}
      \begin{scope}
        [thick, rounded corners=8pt]
        \coordinate (a) at (3,0.8) node at (a) [below=0] {$\gamma$};
    \coordinate (b) at (0.1,0) node at (b) [below=0] {$\alpha$};
     \draw (-1)--($(0) + (0,3mm)$)-- ($(1) - (0,5mm)$)-- ($(2) + (0,3mm)$) -- (3) ;
        \draw  (3)-- ($(4) + (0,5mm)$) --($(5) - (0,3mm)$) --(6) ;
      \end{scope}
      \end{tikzpicture}
      \]

\end{itemize}
\end{defi}

Using the above transformation, we define a mutation of $\DAD$ as follows.

\begin{defi}\label{muta}
Let $w=(w_1w_2\cdots w_{n+1})\in W$ and 
$\D(w)=\{\alpha_{w_1w_{2}}(c_1),\alpha_{w_2w_{3}}(c_2),\cdots,\alpha_{w_{n}w_{n+1}}(c_n)\}$. 
Fix $1\leq i\leq n$ and assume that $w_i>w_{i+1}$, or equivalently, $\alpha_{w_iw_{i+1}}$ is green.
We define \emph{left mutation} $$\mu_i^-(\D(w)):=\{\mu_i^-(\alpha_{w_1w_{2}}(c_1)),\mu_i^-(\alpha_{w_2w_{3}}(c_2)),\cdots,\mu_i^-(\alpha_{w_{n}w_{n+1}}(c_{n}))\}$$ as follows.

\begin{itemize}
\item[(1)] 
For $k\neq i-1,i,i+1$, we do not change $\alpha_{w_kw_{k+1}}$ including its color.
Equivalently, we define 
$$\mu_i^-(\alpha_{w_kw_{k+1}}(c_k))=\alpha_{w_kw_{k+1}}(c_k).$$
\item[(2)]  
We only change the color of $\alpha_{w_iw_{i+1}}$ from green to red. 
Equivalently, 
we define 
$$\mu_i^-(\alpha_{w_iw_{i+1}}(0))=\alpha_{w_iw_{i+1}}(1).$$
\item[(3)] 
We define $\mu_i^-(\alpha_{w_{i-1}w_{i}}(c_{i-1}))=\sigma(\alpha_{w_{i-1}w_{i}};\alpha_{w_{i}w_{i+1}})(c_{i-1}')$,
where 
$$c_{i-1}' =
\begin{cases}
1 & \mathrm{if} \ w_{i-1}<w_{i+1}, \\
0 & \mathrm{if} \ w_{i-1}>w_{i} \mathrm{\ or\ }w_i>w_{i-1}>w_{i+1}.
\end{cases}
$$

Similarly, we define $\mu_i(\alpha_{w_{i+1}w_{i+2}}(c_{i+1}))=\sigma(\alpha_{w_{i+1}w_{i+2}};\alpha_{w_{i}w_{i+1}})(c_{i+1}')$,
where 
$$c_{i+1}' =
\begin{cases}
1 & \mathrm{if} \ w_{i}<w_{i+2}, \\
0 & \mathrm{if} \ w_{i+1}>w_{i+2} \mathrm{\ or\ }w_{i+1}<w_{i+2}<w_{i}.
\end{cases}
$$
\end{itemize}
If $w_i<w_{i+1}$, or equivalently, $\alpha_{w_iw_{i+1}}$ is red,
then we can dually define \emph{right mutation} $\mu^+_i(\D(w))$. We denote by $\mu_i$ left or right mutation.
\end{defi}

\begin{remk}
We can formulate the above mutation in terms of  half-twists. 
Namely, the new diagram is defined by applying a half-twist \cite[1.6.2 Half-twists]{KT} to $\alpha_{w_{i}w_{i+1}}$ and 
hence mutation is nothing but graded half-twists.
\end{remk}


\begin{exam}
Consider $n=3$ and the set of double noncrossing arc diagrams $\DAD_3$. 
The following graph is a mutation graph of $\DAD_3$, where 
the arrows denote left mutations.

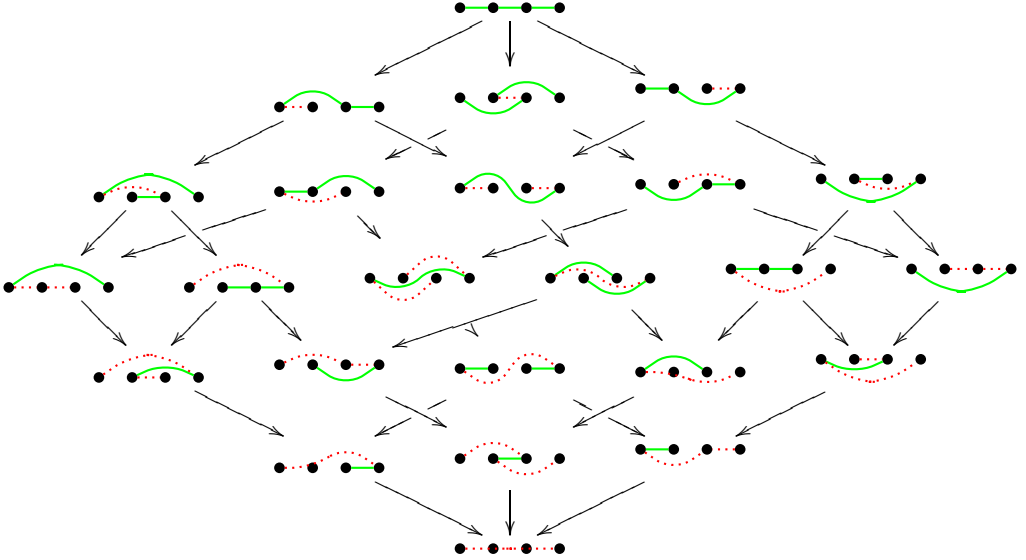
\begin{align*}
\begin{xy}
(  0,-36) *+{\begin{tikzpicture}
      [ mycell/.style={draw, minimum size=1em},
        dot/.style={mycell,
            append after command={\pgfextra \fill (\tikzlastnode) circle[radius=.2em]; \endpgfextra}}]
\begin{scope}[every node/.style={circle, fill=black, inner sep=.5mm, outer sep=0}]
\node (1) {};
\node[right=3mm of 1] (2) {};
\node[right=3mm of 2] (3) {};
\node[right=3mm of 3] (4) {};
      \end{scope}
      \begin{scope}
        [thick, rounded corners=8pt]
         \draw[red,dotted] (1)--($(2)$)--($(3)$)-- (4);
      \end{scope}
      \end{tikzpicture}}="1234",
( 24,-24) *+{\begin{tikzpicture}
      [ mycell/.style={draw, minimum size=1em},
        dot/.style={mycell,
            append after command={\pgfextra \fill (\tikzlastnode) circle[radius=.2em]; \endpgfextra}}]
\begin{scope}[every node/.style={circle, fill=black, inner sep=.5mm, outer sep=0}]
\node (1) {};
\node[right=3mm of 1] (2) {};
\node[right=3mm of 2] (3) {};
\node[right=3mm of 3] (4) {};
      \end{scope}
      \begin{scope}
        [thick, rounded corners=8pt]
        \draw[green]
        (1)--(2); 
         \draw[red,dotted] (4)-- (3);
          \draw[red,dotted] (1)--($(2) - (0,3mm)$)-- (3);
      \end{scope}
      \end{tikzpicture}}="2134",
(  0,-24) *+{\begin{tikzpicture}
      [ mycell/.style={draw, minimum size=1em},
        dot/.style={mycell,
            append after command={\pgfextra \fill (\tikzlastnode) circle[radius=.2em]; \endpgfextra}}]
\begin{scope}[every node/.style={circle, fill=black, inner sep=.5mm, outer sep=0}]
\node (1) {};
\node[right=3mm of 1] (2) {};
\node[right=3mm of 2] (3) {};
\node[right=3mm of 3] (4) {};
      \end{scope}
      \begin{scope}
        [thick, rounded corners=8pt]
        \draw[green]
        (2)-- (3); 
        \draw[red,dotted] (1)--($(2) + (0,3mm)$)-- (3);
        \draw[red,dotted] (2)-- ($(3) - (0,3mm)$)-- (4);
      \end{scope}
      \end{tikzpicture}}="1324",
(-24,-24) *+{\begin{tikzpicture}
      [ mycell/.style={draw, minimum size=1em},
        dot/.style={mycell,
            append after command={\pgfextra \fill (\tikzlastnode) circle[radius=.2em]; \endpgfextra}}]
\begin{scope}[every node/.style={circle, fill=black, inner sep=.5mm, outer sep=0}]
\node (1) {};
\node[right=3mm of 1] (2) {};
\node[right=3mm of 2] (3) {};
\node[right=3mm of 3] (4) {};
      \end{scope}
      \begin{scope}
        [thick, rounded corners=8pt]
        \draw[green]
        (3)-- (4); 
        \draw[red,dotted] (1)-- ($(2)$)--($(3) + (0,3mm)$)-- (4);
      \end{scope}
      \end{tikzpicture}}="1243",
( 48,-12) *+{\begin{tikzpicture}
      [ mycell/.style={draw, minimum size=1em},
        dot/.style={mycell,
            append after command={\pgfextra \fill (\tikzlastnode) circle[radius=.2em]; \endpgfextra}}]
\begin{scope}[every node/.style={circle, fill=black, inner sep=.5mm, outer sep=0}]
\node (1) {};
\node[right=3mm of 1] (2) {};
\node[right=3mm of 2] (3) {};
\node[right=3mm of 3] (4) {};
      \end{scope}
      \begin{scope}
        [thick, rounded corners=8pt]
        \draw[green]
        (1)-- ($(2) - (0,2mm)$)-- (3); 
                  \draw[red,dotted] (1)-- ($(2) - (0,3mm)$)--($(3) - (0,3mm)$)-- (4);
  \draw[red,dotted] (2)-- (3);
      \end{scope}
      \end{tikzpicture}}="2314",
( 24,-12) *+{\begin{tikzpicture}
      [ mycell/.style={draw, minimum size=1em},
        dot/.style={mycell,
            append after command={\pgfextra \fill (\tikzlastnode) circle[radius=.2em]; \endpgfextra}}]
\begin{scope}[every node/.style={circle, fill=black, inner sep=.5mm, outer sep=0}]
\node (1) {};
\node[right=3mm of 1] (2) {};
\node[right=3mm of 2] (3) {};
\node[right=3mm of 3] (4) {};
      \end{scope}
      \begin{scope}
        [thick, rounded corners=8pt]
        \draw[green]
        (1)-- ($(2) + (0,3mm)$)-- (3); 
          \draw[red,dotted] (1)-- ($(2)$)--($(3) - (0,2mm)$)-- (4);
      \end{scope}
      \end{tikzpicture}}="3124",
(  0,-12) *+{\begin{tikzpicture}
      [ mycell/.style={draw, minimum size=1em},
        dot/.style={mycell,
            append after command={\pgfextra \fill (\tikzlastnode) circle[radius=.2em]; \endpgfextra}}]
\begin{scope}[every node/.style={circle, fill=black, inner sep=.5mm, outer sep=0}]
\node (1) {};
\node[right=3mm of 1] (2) {};
\node[right=3mm of 2] (3) {};
\node[right=3mm of 3] (4) {};
      \end{scope}
      \begin{scope}
        [thick, rounded corners=8pt]
        \draw[green]
        (1)-- (2); 
         \draw[green] (3)-- (4); 
          \draw[red,dotted] (1)-- ($(2) - (0,3mm)$)--($(3) + (0,3mm)$)-- (4);
      \end{scope}
      \end{tikzpicture}}="2143",
(-24,-12) *+{\begin{tikzpicture}
      [ mycell/.style={draw, minimum size=1em},
        dot/.style={mycell,
            append after command={\pgfextra \fill (\tikzlastnode) circle[radius=.2em]; \endpgfextra}}]
\begin{scope}[every node/.style={circle, fill=black, inner sep=.5mm, outer sep=0}]
\node (1) {};
\node[right=3mm of 1] (2) {};
\node[right=3mm of 2] (3) {};
\node[right=3mm of 3] (4) {};
      \end{scope}
      \begin{scope}
        [thick, rounded corners=8pt]
        \draw[green]        (2)--($(3) - (0,3mm)$)-- (4); 
      \draw[red,dotted] (3)-- (4); 
        \draw[red,dotted] (1)-- ($(2) + (0,2mm)$)-- (3);
      \end{scope}
      \end{tikzpicture}}="1342",
(-48,-12) *+{\begin{tikzpicture}
      [ mycell/.style={draw, minimum size=1em},
        dot/.style={mycell,
            append after command={\pgfextra \fill (\tikzlastnode) circle[radius=.2em]; \endpgfextra}}]
\begin{scope}[every node/.style={circle, fill=black, inner sep=.5mm, outer sep=0}]
\node (1) {};
\node[right=3mm of 1] (2) {};
\node[right=3mm of 2] (3) {};
\node[right=3mm of 3] (4) {};
      \end{scope}
      \begin{scope}
        [thick, rounded corners=8pt]
        \draw[green]        (2)-- ($(3) + (0,2mm)$)-- (4); 
        \draw[red,dotted] (2)-- (3); 
        \draw[red,dotted] (1)-- ($(2) + (0,3mm)$)--($(3) + (0,3mm)$)-- (4);
      \end{scope}
      \end{tikzpicture}}="1423",
( 60,  0) *+{\begin{tikzpicture}
      [ mycell/.style={draw, minimum size=1em},
        dot/.style={mycell,
            append after command={\pgfextra \fill (\tikzlastnode) circle[radius=.2em]; \endpgfextra}}]
\begin{scope}[every node/.style={circle, fill=black, inner sep=.5mm, outer sep=0}]
\node (1) {};
\node[right=3mm of 1] (2) {};
\node[right=3mm of 2] (3) {};
\node[right=3mm of 3] (4) {};
      \end{scope}
      \begin{scope}
        [thick, rounded corners=8pt]
        \draw[green]
        (1)-- ($(2) - (0,3mm)$)--($(3) - (0,3mm)$)-- (4); 
        \draw[red,dotted] (2)-- (3)-- (4); 
      \end{scope}
      \end{tikzpicture}}="2341",
( 36,  0) *+{\begin{tikzpicture}
      [ mycell/.style={draw, minimum size=1em},
        dot/.style={mycell,
            append after command={\pgfextra \fill (\tikzlastnode) circle[radius=.2em]; \endpgfextra}}]
\begin{scope}[every node/.style={circle, fill=black, inner sep=.5mm, outer sep=0}]
\node (1) {};
\node[right=3mm of 1] (2) {};
\node[right=3mm of 2] (3) {};
\node[right=3mm of 3] (4) {};
      \end{scope}
      \begin{scope}
        [thick, rounded corners=8pt]
        \draw[green]
        (1)-- (2)-- (3); 
        \draw[red,dotted] (1)-- ($(2) - (0,3mm)$)-- ($(3) - (0,3mm)$)-- (4); 
      \end{scope}
      \end{tikzpicture}}="3214",
( 12,  0) *+{\begin{tikzpicture}
      [ mycell/.style={draw, minimum size=1em},
        dot/.style={mycell,
            append after command={\pgfextra \fill (\tikzlastnode) circle[radius=.2em]; \endpgfextra}}]
\begin{scope}[every node/.style={circle, fill=black, inner sep=.5mm, outer sep=0}]
\node (1) {};
\node[right=3mm of 1] (2) {};
\node[right=3mm of 2] (3) {};
\node[right=3mm of 3] (4) {};
      \end{scope}
      \begin{scope}
        [thick, rounded corners=8pt]
        \draw[green]
        (1)-- ($(2) + (0,3mm)$)-- (3); 
          \draw[green] (2)-- ($(3) - (0,3mm)$)-- (4); 
           \draw[red,dotted](1)-- ($(2) + (0,2mm)$)-- ($(3) - (0,2mm)$)-- (4); 
      \end{scope}
      \end{tikzpicture}}="3142",
(-12,  0) *+{\begin{tikzpicture}
      [ mycell/.style={draw, minimum size=1em},
        dot/.style={mycell,
            append after command={\pgfextra \fill (\tikzlastnode) circle[radius=.2em]; \endpgfextra}}]
\begin{scope}[every node/.style={circle, fill=black, inner sep=.5mm, outer sep=0}]
\node (1) {};
\node[right=3mm of 1] (2) {};
\node[right=3mm of 2] (3) {};
\node[right=3mm of 3] (4) {};
      \end{scope}
      \begin{scope}
        [thick, rounded corners=8pt]
         \draw[green] (1)-- ($(2) - (0,2mm)$)-- ($(3) + (0,2mm)$)-- (4); 
          \draw[red,dotted](2)-- ($(3) + (0,4mm)$)-- (4);
           \draw[red,dotted](1)-- ($(2) - (0,4mm)$)-- (3);
      \end{scope}
      \end{tikzpicture}}="2413",
(-36,  0) *+{\begin{tikzpicture}
      [ mycell/.style={draw, minimum size=1em},
        dot/.style={mycell,
            append after command={\pgfextra \fill (\tikzlastnode) circle[radius=.2em]; \endpgfextra}}]
\begin{scope}[every node/.style={circle, fill=black, inner sep=.5mm, outer sep=0}]
\node (1) {};
\node[right=3mm of 1] (2) {};
\node[right=3mm of 2] (3) {};
\node[right=3mm of 3] (4) {};
      \end{scope}
      \begin{scope}
        [thick, rounded corners=8pt]
        \draw[green]
        (2)-- (3)-- (4); 
         \draw[red,dotted](1)-- ($(2) + (0,3mm)$)-- ($(3) + (0,3mm)$)-- (4);
      \end{scope}
      \end{tikzpicture}}="1432",
(-60,  0) *+{\begin{tikzpicture}
      [ mycell/.style={draw, minimum size=1em},
        dot/.style={mycell,
            append after command={\pgfextra \fill (\tikzlastnode) circle[radius=.2em]; \endpgfextra}}]
\begin{scope}[every node/.style={circle, fill=black, inner sep=.5mm, outer sep=0}]
\node (1) {};
\node[right=3mm of 1] (2) {};
\node[right=3mm of 2] (3) {};
\node[right=3mm of 3] (4) {};
      \end{scope}
      \begin{scope}
        [thick, rounded corners=8pt]
        \draw[green] (1)-- ($(2) + (0,3mm)$)-- ($(3) + (0,3mm)$)-- (4) ; 
         \draw[red,dotted] (1)-- (2);
          \draw[red,dotted] (2)-- (3);
      \end{scope}
      \end{tikzpicture}}="4123",
( 48, 12) *+{\begin{tikzpicture}
      [ mycell/.style={draw, minimum size=1em},
        dot/.style={mycell,
            append after command={\pgfextra \fill (\tikzlastnode) circle[radius=.2em]; \endpgfextra}}]
\begin{scope}[every node/.style={circle, fill=black, inner sep=.5mm, outer sep=0}]
\node (1) {};
\node[right=3mm of 1] (2) {};
\node[right=3mm of 2] (3) {};
\node[right=3mm of 3] (4) {};
      \end{scope}
      \begin{scope}
        [thick, rounded corners=8pt]
        \draw[green]        (2)-- (3); 
        \draw[green] (1)-- ($(2) - (0,3mm)$)-- ($(3) - (0,3mm)$)-- (4);
        \draw[red,dotted] (2)-- ($(3) - (0,2mm)$)-- (4);
      \end{scope}
      \end{tikzpicture}}="3241",
( 24, 12) *+{\begin{tikzpicture}
      [ mycell/.style={draw, minimum size=1em},
        dot/.style={mycell,
            append after command={\pgfextra \fill (\tikzlastnode) circle[radius=.2em]; \endpgfextra}}]
\begin{scope}[every node/.style={circle, fill=black, inner sep=.5mm, outer sep=0}]
\node (1) {};
\node[right=3mm of 1] (2) {};
\node[right=3mm of 2] (3) {};
\node[right=3mm of 3] (4) {};
      \end{scope}
      \begin{scope}
        [thick, rounded corners=8pt]
        \draw[green]   (3)-- (4);
     \draw[green] (1)-- ($(2) - (0,3mm)$)--  (3); 
     \draw[red,dotted] (2)-- ($(3) + (0,2mm)$)-- (4);
      \end{scope}
      \end{tikzpicture}}="2431",
(  0, 12) *+{\begin{tikzpicture}
      [ mycell/.style={draw, minimum size=1em},
        dot/.style={mycell,
            append after command={\pgfextra \fill (\tikzlastnode) circle[radius=.2em]; \endpgfextra}}]
\begin{scope}[every node/.style={circle, fill=black, inner sep=.5mm, outer sep=0}]
\node (1) {};
\node[right=3mm of 1] (2) {};
\node[right=3mm of 2] (3) {};
\node[right=3mm of 3] (4) {};
      \end{scope}
      \begin{scope}
        [thick, rounded corners=8pt]
         \draw[green] (1)-- ($(2) + (0,3mm)$)-- ($(3) - (0,3mm)$)-- (4);
         \draw[red,dotted] (1)-- (2);
         \draw[red,dotted] (3)-- (4);
      \end{scope}
      \end{tikzpicture}}="3412",
(-24, 12) *+{\begin{tikzpicture}
      [ mycell/.style={draw, minimum size=1em},
        dot/.style={mycell,
            append after command={\pgfextra \fill (\tikzlastnode) circle[radius=.2em]; \endpgfextra}}]
\begin{scope}[every node/.style={circle, fill=black, inner sep=.5mm, outer sep=0}]
\node (1) {};
\node[right=3mm of 1] (2) {};
\node[right=3mm of 2] (3) {};
\node[right=3mm of 3] (4) {};
      \end{scope}
      \begin{scope}
        [thick, rounded corners=8pt]
        \draw[green]   (1)-- (2);
        \draw[green] (2)-- ($(3) + (0,3mm)$)-- (4);
        \draw[red,dotted] (1)-- ($(2) - (0,2mm)$)-- (3);
      \end{scope}
      \end{tikzpicture}}="4213",
(-48, 12) *+{\begin{tikzpicture}
      [ mycell/.style={draw, minimum size=1em},
        dot/.style={mycell,
            append after command={\pgfextra \fill (\tikzlastnode) circle[radius=.2em]; \endpgfextra}}]
\begin{scope}[every node/.style={circle, fill=black, inner sep=.5mm, outer sep=0}]
\node (1) {};
\node[right=3mm of 1] (2) {};
\node[right=3mm of 2] (3) {};
\node[right=3mm of 3] (4) {};
      \end{scope}
      \begin{scope}
        [thick, rounded corners=8pt]
        \draw[green]       (2)-- (3);
         \draw[green] (1)-- ($(2) + (0,3mm)$)--($(3) + (0,3mm)$)-- (4);
         \draw[red,dotted] (1)-- ($(2) + (0,2mm)$)-- (3);
      \end{scope}
      \end{tikzpicture}}="4132",
( 24, 24) *+{\begin{tikzpicture}
      [ mycell/.style={draw, minimum size=1em},
        dot/.style={mycell,
            append after command={\pgfextra \fill (\tikzlastnode) circle[radius=.2em]; \endpgfextra}}]
\begin{scope}[every node/.style={circle, fill=black, inner sep=.5mm, outer sep=0}]
\node (1) {};
\node[right=3mm of 1] (2) {};
\node[right=3mm of 2] (3) {};
\node[right=3mm of 3] (4) {};
      \end{scope}
      \begin{scope}
        [thick, rounded corners=8pt]
        \draw[green]
        (1)-- (2); 
        \draw[green] (2)-- ($(3) - (0,3mm)$)-- (4);
        \draw[red,dotted] (3)-- (4);
      \end{scope}
      \end{tikzpicture}}="3421",
(  0, 24) *+{\begin{tikzpicture}
      [ mycell/.style={draw, minimum size=1em},
        dot/.style={mycell,
            append after command={\pgfextra \fill (\tikzlastnode) circle[radius=.2em]; \endpgfextra}}]
\begin{scope}[every node/.style={circle, fill=black, inner sep=.5mm, outer sep=0}]
\node (1) {};
\node[right=3mm of 1] (2) {};
\node[right=3mm of 2] (3) {};
\node[right=3mm of 3] (4) {};
      \end{scope}
      \begin{scope}
        [thick, rounded corners=8pt]
        \draw[green] (2)-- ($(3) + (0,3mm)$)-- (4);
        \draw[green] (1)-- ($(2) - (0,3mm)$)-- (3);
        \draw[red,dotted] (2)-- (3);
      \end{scope}
      \end{tikzpicture}}="4231",
(-24, 24) *+{\begin{tikzpicture}
      [ mycell/.style={draw, minimum size=1em},
        dot/.style={mycell,
            append after command={\pgfextra \fill (\tikzlastnode) circle[radius=.2em]; \endpgfextra}}]
\begin{scope}[every node/.style={circle, fill=black, inner sep=.5mm, outer sep=0}]
\node (1) {};
\node[right=3mm of 1] (2) {};
\node[right=3mm of 2] (3) {};
\node[right=3mm of 3] (4) {};
      \end{scope}
      \begin{scope}
        [thick, rounded corners=8pt]
        \draw[green] (3)-- (4);
        \draw[green] (1)-- ($(2) + (0,3mm)$)-- (3);
        \draw[red,dotted] (1)-- (2);
      \end{scope}
      \end{tikzpicture}}="4312",
(  0, 36) *+{\begin{tikzpicture}
      [ mycell/.style={draw, minimum size=1em},
        dot/.style={mycell,
            append after command={\pgfextra \fill (\tikzlastnode) circle[radius=.2em]; \endpgfextra}}]
\begin{scope}[every node/.style={circle, fill=black, inner sep=.5mm, outer sep=0}]
\node (1) {};
\node[right=3mm of 1] (2) {};
\node[right=3mm of 2] (3) {};
\node[right=3mm of 3] (4) {};
      \end{scope}
      \begin{scope}
        [thick, rounded corners=8pt]
        \draw[green]
        (1)-- (2)-- (3)--(4); 
      \end{scope}
      \end{tikzpicture}}="4321",
\ar "2134";"1234"
\ar "1324";"1234"
\ar "1243";"1234"
\ar "2314";"2134"
\ar|\hole "2143";"2134"
\ar "3124";"1324"
\ar "1342";"1324"
\ar|\hole "2143";"1243"
\ar "1423";"1243"
\ar "2341";"2314"
\ar "3214";"2314"
\ar "3214";"3124"
\ar "3142";"3124"
\ar|\hole "2413";"2143"
\ar "3142";"1342"
\ar "1432";"1342"
\ar "1432";"1423"
\ar "4123";"1423"
\ar "3241";"2341"
\ar|\hole "2431";"2341"
\ar "3241";"3214"
\ar "3412";"3142"
\ar|\hole "2431";"2413"
\ar "4213";"2413"
\ar "4132";"1432"
\ar|\hole "4213";"4123"
\ar "4132";"4123"
\ar "3421";"3241"
\ar|\hole "4231";"2431"
\ar "3421";"3412"
\ar "4312";"3412"
\ar|\hole "4231";"4213"
\ar "4312";"4132"
\ar "4321";"3421"
\ar "4321";"4231"
\ar "4321";"4312"
\end{xy}.
\end{align*}  
\end{exam}

The following lemma is straightforward. 

\begin{prop}\label{comm}
Let $s_i=(i\ i+1)$ be the transposition of $W$. Then the left action of $s_i$ on $w$ is compatible with  mutation $\mu_i$, 
that is, we have the following commutative diagram 

\begin{align*}
\begin{xy}
( 0,  8) *+{\DAD}   ="01",
( 0, -8) *+{\DAD}    ="00",
(-45,  8) *+{W}   ="02",
(-45,  -8) *+{W}   ="12",
\ar^{\mu_i}   "01";"00"
\ar^{\D} "02";"01"
\ar^{\D} "12";"00"
\ar_{s_i\cdot}  "02";"12"
\end{xy}
\end{align*}
In particular, we have $\mu_{i}(\D(w))\in\DAD$ and 
$\mu^-_i\circ\mu^+_i=\id=\mu^+_i\circ\mu^-_i$.
\end{prop}

\begin{proof}
Let $w=w_1\cdots w_{n+1}$ and $$\D(w) = \{\alpha_{w_1w_{2}}(c_1),\alpha_{w_2w_{3}}(c_2),\cdots,\alpha_{w_{n}w_{n+1}}(c_{n})\}.$$
Let $u:=s_iw=u_1\cdots u_{n+1}$ and 
$$\D(u) = \{\alpha_{u_1u_{2}}(d_1),\alpha_{u_2u_{3}}(d_2),\cdots,\alpha_{u_{n}u_{n+1}}(d_{n})\}.$$ 
We will show that 
$$\{\alpha_{u_ku_{k+1}}(d_k)\ |\ 1\leq k\leq n\}=\mu_i(\{\alpha_{w_kw_{k+1}}(c_k))\ |\ 1\leq k\leq n\}).$$
It is also enough to show that $\alpha_{u_ku_{k+1}}(d_k)=\mu_i(\alpha_{w_kw_{k+1}}(c_k))$ for any $1\leq k\leq n.$

Since 
$u_1\cdots u_{n+1}=w_1w_2\cdots w_{i-1}w_{i+1}w_iw_{i+2}\cdots w_{n+1}$, 
we have 
$$\alpha_{u_ku_{k+1}}(d_k)=\alpha_{w_kw_{k+1}}(c_k)=\mu_i(\alpha_{w_kw_{k+1}}(c_k))$$
for any $k\neq i-1,i,i+1$.
By case-by-case analysis, we can easily check the case of $k=i-1$,$k=i$ and $k=i+1$.
\end{proof}

By Propositions \ref{h^0} and \ref{comm}, we obtain the following corollary. 

\begin{cor}\label{dad partial order}
Let $\D,\D'\in\DAD$. 
We write $\D\leq\D'$ if $D$ is obtained from $\D'$ by a sequence of left mutation. Then it gives a partial order on $\DAD$ and it is isomorphic to $W$ as posets. 
\end{cor}

Moreover we define a partial order on $\NAD$ as follows. 

\begin{defi}
Let $\G,\G'\in\NAD$. We define a partial order on $\NAD$ by 
$\G\leq\G'$ if $\mathbb{G}^{-1}(\G)\leq\mathbb{G}^{-1}(\G')$. In particular, $\NAD\cong\DAD\cong W$.
\end{defi}



Next we will show the compatibility of mutation of $\DAD$ and $\twosmc\Pi$.
For this purpose, we recall mutation of simple minded collections based on \cite{KY,BY}. 

\begin{defi}\label{mut simple}\cite{KY}
Let $A$ be a finite dimensional algebra such that $|A|=n$. 
Let $\xX:=\{X_1,X_2,\cdots,X_{n}\}$ be a SMC of $\Db(\mod A).$ 
Fix $1\leq i\leq n$. 
We define a left mutation of SMCs 
$$\mu_i^-(\xX):=\{\mu_i^-(X_1),\mu_i^-(X_2),\cdots,\mu_i^-(X_{n})\}\subset\Db(\mod A)$$ as follows. 
Let $\E:=\Filt{X_i}$ be the extension closure of $X_i$ of $\Db(\mod A).$
Take a minimal left $\E$-approximation of $X_k[-1]$
\[\xymatrix{
X_k[-1] \ar[r]^{\ \ \ 
f_k}& E_k .
}\]
Then we define $\mu_i^-(X_k)=\mathrm{cone}(f_k)$, where $\mathrm{cone}(f_k)$ is the mapping cone of $f_k$. 
Dually we can define the right mutation $\mu^+_i$ and 
we have $\mu_i^+\circ\mu_i^-=\id$ and $\mu_i^-\circ\mu_i^+=\id$. We denote by $\mu_i$ a left or right mutation.
\end{defi}

Even if $\xX$ is 2-term, $\mu_i^-(\xX)$ is not necessary 2-term in general. With regard to this fact, we have the following nice characterization of 2-term  SMCs \cite[Subsection 3.7]{BY}.

\begin{lemm}\cite{BY}\label{mut 2-simple}
Let $\xX:=\{X_1,X_2,\cdots,X_n\}$ be a 2-term SMCs of $\Db(\mod A).$
\begin{itemize}
\item[(1)] We have $\mu_i^-(\xX)\in\twosmc A$ 
if and only if we have $X_i \in \mod A$. 
\item[(2)]  If the above equivalent condition (1) holds, then 
 $\mu_i^-(\xX)$ has the following description.
\begin{itemize}
\item[(i)] We have 
$$\mu_i^-(X_i)=X_i[1].$$
\item[(ii)] Assume $X_k\in\mod A$ $(k\neq i)$. 
Then we have $\mu_i^-(X_k)=\mathrm{cone}(f_k)\in \mod A$ and 
we have an exact sequence $0 \to E_k \to \mathrm{cone}(f_k) \to X_k \to 0$ in $\mod A$. 
\item[(iii)] Assume $X_k\in\mod A[1]$.  
Then we have one of the following two cases:

(Case1). We have an exact sequence in $\mod A$
$$\xymatrix{0\ar[r]&
X_k[-1] \ar[r]^{\ \ f_k}
& E_k \ar[r] & \Coker(f_k)\ar[r]&0.
}$$ 
In this case, we have $\mu_i^-(X_k)=\Coker(f_k)$.

(Case2).We have an exact sequence in $\mod A$
$$\xymatrix{0\ar[r]&
\ker(f_k)\ar[r]&
X_k[-1] \ar[r]^{\ \ f_k}
& E_k \ar[r] \ar[r]&0.
}$$ 
In this case, we have $\mu_i^-(X_k)=\ker(f_k)[1]$.

\end{itemize}
\end{itemize}
\end{lemm}

The following proposition is a key result
of this section.

\begin{prop}\label{main2}
The map $\Psi:\DAD\to\twosmc\Pi$ is well-defined and 
the mutation of $\DAD$ and $\twosmc\Pi$ are compatible, that is, we have the following commutative diagram 

\begin{align*}
\begin{xy}
( 0,  8) *+{\DAD}   ="01",
(45,  8) *+{\twosmc\Pi }   ="11",
( 0, -8) *+{\DAD}    ="00",
(45, -8) *+{\twosmc\Pi}     ="10",
\ar^{\Psi} "01";"11"
\ar^{\Psi}  "00";"10"
\ar_{\mu_i}   "01";"00"
\ar^{\mu_i}   "11";"10"
\end{xy}
\end{align*}
\end{prop}

\begin{proof}
Let 
$\D(w) = \{\alpha_{w_1w_{2}}(c_1),\alpha_{w_2w_{3}}(c_2),\cdots,\alpha_{w_{n}w_{n+1}}(c_{n})\}$. Fix $1\leq i\leq n.$ 
We will show the compatibility of left mutation. 
We assume that $w_i>w_{i+1}$, or equivalently, $c_i=0$ and 
we will show that 
$$\Psi(\mu_i^-(\{\alpha_{w_kw_{k+1}}(c_k)\}_{1\leq k\leq n}))=\mu_i^-(
\Psi(\{\alpha_{w_kw_{k+1}}(c_k)\}_{1\leq k\leq n})).
$$
In our notation, it is enough to show that
$\Psi(\mu_i^-(\alpha_{w_kw_{k+1}}(c_k)))=\mu_i^-(
\Psi(\alpha_{w_kw_{k+1}}(c_k)))$ for any $1\leq k\leq n$. 
Then, because of $\Psi(\D(\id))=\{S_1,S_2,\cdots,S_n\}\in\twosmc\Pi$ and Lemma \ref{mut 2-simple}, it also implies that  $\Psi:\DAD\to\twosmc\Pi$ is well-defined.

For simplicity, we write 
$S(\alpha_{w_kw_{k+1}})$ by $S(w_kw_{k+1})$.
Then we have 
\begin{eqnarray*}
\Psi(\D(w))=
\{S({w_1w_{2}})[c_1],S({w_2w_{3}})[c_2],\cdots,S({w_{n}w_{n+1}})[c_{n}]\}.
\end{eqnarray*}

Let $\E:=\Filt\{S({w_iw_{i+1}})\}$ be the extension closure of $S({w_iw_{i+1}})$. 
Then, by Proposition \ref{arc-brick bij} (2), we have 
$\E=\add\{S({w_iw_{i+1}})\}.$ 

\begin{itemize}
\item[(1)] Fix $1\leq k\leq n$ such that $k\neq i-1,i,i+1$. 
Then the two arcs $\alpha_{w_kw_{k+1}}$ and $\alpha_{w_iw_{i+1}}$ are noncrossing. 
If $c_k=0$, then Lemma \ref{ext lemm} implies that
$$\Hom_{\Db(\mod\Pi)}(S({w_kw_{k+1}})[-1],S({w_iw_{i+1}}))=\Ext^1_{\Pi}(S({w_kw_{k+1}}),S({w_iw_{i+1}}))=0.$$

If $c_k=1$, then Lemma \ref{hom lemm} implies that  
$$\Hom_{\Db(\mod\Pi)}(S({w_kw_{k+1}}),S({w_iw_{i+1}}))=0.$$

Therefore, in both cases, 
\[\xymatrix{
S({w_kw_{k+1}})[c_k-1] \ar[r]^(0.8){f_k\ \ }& 0,
}\]
is a left minimal $\mathcal{E}$-approximation and hence we have 
$$\mu_i^-(S({w_kw_{k+1}})[c_k])=\mathrm{cone}(f_k)=S({w_kw_{k+1}})[c_k]=\Psi(\mu_i^-(\alpha_{w_kw_{k+1}}(c_k))).$$ 
Thus we get 
$$\mu_i^-(\Psi(\alpha_{w_kw_{k+1}}(c_k)))=\Psi(\mu_i^-(\alpha_{w_kw_{k+1}}(c_k))).$$

\item[(2)] Since $c_i=0$, we have 
$$\mu_i^-(S({w_iw_{i+1}})[0])=S({w_iw_{i+1}})[1].$$
Thus we get 
$\mu_i^-(\Psi(\alpha_{w_iw_{i+1}}(c_i)))=\Psi(\mu_i^-(\alpha_{w_iw_{i+1}}(c_i))).$
\item[(3)] We will calculate $\mu_i^-(S({w_{i-1}w_{i}})[c_{i-1}])$. 

\begin{itemize}
\item[(i)] 
First assume $w_{i-1}>w_{i}$, or equivalently, $c_{i-1}=0$. 
Since the dimension vectors of 
$S({w_{i-1}w_{i}})$ and $S({w_{i}w_{i+1}})$ do not share any non-zero entry and they are adjacent, we have $(\underline{\dim} S({w_{i-1}w_{i}}),\underline{\dim} S({w_{i}w_{i+1}}))=-1$. 
On the other hand, we have 
$\Hom_\Pi(S({w_{i-1}w_{i}}),S({w_{i}w_{i+1}}))=0=\Hom_\Pi(S({w_{i}w_{i+1}}),S({w_{i-1}w_{i}}))$ 
by Proposition \ref{noncrossing-hom}. 
Thus, Proposition \ref{cw} implies  $$\dim\Ext_\Pi^1(S({w_{i-1}w_{i}}),S({w_{i}w_{i+1}}))=1.$$
Therefore there exists a unique (up to scalar) non-zero map 
\[\xymatrix{S({w_{i-1}w_{i}})[-1] \ar[r]^(0.55){f_{i-1}}& S({w_{i}w_{i+1}}), }\] which is a left minimal $\mathcal{E}$-approximation. 
Then, by Lemma \ref{mut 2-simple} (ii), we have  $\mathrm{cone}(f_{i-1})\in\mod\Pi$ which is given by the following non-split exact sequence 
$$0 \to S({w_{i}w_{i+1}}) \to \mathrm{cone}(f_{i-1}) \to S({w_{i-1}w_{i}}) \to 0.$$ 
Thus we get 
$\mathrm{cone}(f_{i-1})\cong S(\sigma(\alpha_{w_{i-1}w_{i}};\alpha_{w_{i}w_{i+1}}))=
\Psi(\mu_i^-(\alpha_{w_{i-1}w_{i}}(c_{i-1})))$. 
Therefore we obtain 
$\mu_i^-(\Psi(\alpha_{w_{i-1}w_{i}}(c_{i-1})))=\Psi(\mu_i^-(\alpha_{w_{i-1}w_{i}}(c_{i-1}))).$

\item[(ii)] 
Next assume $w_{i-1}<w_{i}$, or equivalently, $c_{i-1}=1$. 
As same as Lemma \ref{hom lemm} (2), if $w_{i-1}>w_{i+1}$, then we have a unique inclusion from $S({w_{i-1}w_{i}})$ to $S({w_{i}w_{i+1}})$ and 
if $w_{i-1}<w_{i+1}$, then we have a unique  surjection from $S({w_{i-1}w_{i}})$ to $S({w_{i}w_{i+1}})$. Thus, 
we have  $$\dim\Hom_\Pi(S({w_{i-1}w_{i}}),S({w_{i}w_{i+1}}))=1.$$
Thus there exists a unique 
non-zero map 
\[\xymatrix{
S({w_{i-1}w_{i}})[1][-1]\cong S({w_{i-1}w_{i}}) \ar[r]^(0.65){f_{i-1}}& S({w_{i}w_{i+1}}),
}\]
which is a left minimal $\mathcal{E}$-approximation.

{\bf (Case 1).}
Assume $w_{i-1}>w_{i+1}$. Then we have an exact sequence 
$$\xymatrix{0\ar[r]&
S({w_{i-1}w_{i}}) \ar[r]^{f_{i-1}}
& S({w_{i}w_{i+1}}) \ar[r] & \Coker(f_{i-1})\ar[r]&0.
}$$ 
By Lemma \ref{mut 2-simple} (iii), we have 
$$\mu_i^-(S({w_{i-1}w_{i}})[1])=\Coker(f_{i-1})=S(\sigma(\alpha_{w_{i-1}w_{i}};\alpha_{w_{i}w_{i+1}}))[c_{i-1}'].$$

{\bf (Case 2).}
Assume $w_{i-1}<w_{i+1}$. 
Then we have an exact sequence 
$$\xymatrix{0\ar[r]&\Ker(f_{i-1})\ar[r]&S({w_{i-1}w_{i}})\ar[r]^{{f_{i-1}}} 
& S({w_{i}w_{i+1}}) \ar[r] & 0.
}$$ 
By Lemma \ref{mut 2-simple} (iii), we have 
$$\mu_i^-(S({w_{i-1}w_{i}})[1])=\ker(f_{i-1})[1]=S(\sigma(\alpha_{w_{i-1}w_{i}};\alpha_{w_{i}w_{i+1}}))[c_{i-1}'].$$

Therefore, we obtain 
$\mu_i^-(\Psi(\alpha_{w_{i-1}w_{i}})[c_{i-1}])=\Psi(\mu_i^-(\alpha_{w_{i-1}w_{i}}(c_{i-1}))).$
\end{itemize}


\item[(4)] We can calculate $\mu_i^-(S({w_{i+1}w_{i+2}})[c_{i+1}])$ as same as (3) and we get $$\mu_i^-(\Psi(\alpha_{w_{i+1}w_{i+2}}(c_{i+1})))=\Psi(\mu_i^-(\alpha_{w_{i+1}w_{i+2}}(c_{i+1}))).$$ 
We leave to the reader the calculation. 

\end{itemize}
By (1),(2),(3) and (4), we get the conclusion.
\end{proof}

As a consequence, we have the following result. 

\begin{thm}\label{two poset iso} 
\begin{itemize}
\item[(1)] We have a poset isomorphism
$$\Psi:\DAD\longrightarrow\twosmc\Pi.$$
\item[(2)] We have a poset isomorphism 
$$\Phi:\NAD\to\sbrick\Pi.$$ 
\end{itemize}

\end{thm}

\begin{proof}
(1) First we will show that the map is injective. 
Let $\D,\D'\in \DAD$ and 
$$\D = \{\alpha_{1}(c_1),\alpha_{2}(c_2),\cdots,\alpha_{n}(c_{n})\}, \ \ \ \D' = \{\alpha_{1}'(c_1'),\alpha_{2}'(c_2'),\cdots,\alpha_{n}'(c_{n}')\}.$$

Assume that $\Psi(\D)=\Psi(\D')$. 
It is equivalent to saying that 
$$\{S(\alpha_{i})[c_i]\}=\{S(\alpha_{i}')[c_i']\}.\  \ \ ({1\leq i\leq n})$$
Then, Proposition \ref{arc-brick bij} 
implies $\D=\D'$. 
Thus the map is injective. 

Next we will show that the map is an order embedding. 
Let $\D,\D'\in \DAD$ and 
assume that $\D'\gtrdot\D$, or equivalently, 
$\D$ is obtained by left mutation from $\D'$. 
By Proposition \ref{main2}, it is also equivalent to saying that $\Psi(\D)$ is obtained  by left mutation from $\Psi(\D')$. 
Equivalently, we have $\Psi(\D')\gtrdot\Psi(\D)$ by  \cite[Theorem 2.35]{AI} and \cite[Theorem 7.12]{KY}. 
Therefore, the map $W\to\twosmc\Pi$ is an order embedding. 

Finally we will show that the map is surjective, which follows from the same argument as \cite[Corollary 2.38]{AIR}. 
From the above argument, $\{\Psi(\D)\ |\ \D\in\DAD\}$ consists of the connected component in the Hasse quiver of $\twosmc\Pi$, which is isomorphic to $W$. 
It is enough to show that any element $\xX\in\twosmc\Pi$ belongs to this component. 
Then, since $\xX\in\twosmc\Pi$, we have 
$\Psi(\D(\id))\geq\xX$. 
Then \cite[Proposition 2.35]{AI} and \cite[Theorem 7.12]{KY} implies that there exists a sequence $\Psi(\D(\id))>\xX_1 >\xX_2 >\cdots$ of left mutation of $\twosmc\Pi$ such that $\xX_i\geq \xX$ for any $i$. 
Since the component $\{\Psi(\D)\ |\ \D\in\DAD\}$ is finite and $\Psi(\D(\id))\in\{\Psi(\D)\ |\ \D\in\DAD\}$, this sequence must be finite. Thus we have $\xX=\xX_j$ for some $j$ and hence $\xX$ belongs to $\{\Psi(\D)\ |\ \D\in\DAD\}$.

(2) We have the following commutative diagram and all maps are bijections by Proposition \ref{h^0}, Theorems \ref{smc-fsbrick bij} and \ref{bij nad}.
\begin{align*}
\begin{xy}
( 0,  8) *+{\DAD}   ="01",
(45,  8) *+{\twosmc\Pi }   ="11",
( 0, -8) *+{\NAD}    ="00",
(45, -8) *+{\sbrick \Pi,}     ="10",
\ar^{\mathbb{G}}   "01";"00"
\ar^{\Psi} "01";"11"
\ar^{\Phi}  "00";"10"
\ar^{H^0}   "11";"10"
\end{xy}
\end{align*}

Moreover, since 
the maps $\mathbb{G},\Psi,H^0$ are poset isomorphisms, 
we obtain the conclusion.
\end{proof}


The next corollary allows us to compare the orders  of arc diagrams in a simple way.

\begin{cor}\label{poset def}
$\D(w)\leq \D(v)$
(or equivalently, $\G(w)\leq \G(v)$) if and only if there exists no graph map 
from any $M\in\Phi(\G(w))$ to any $N\in\Phi(\R(v))$.

\end{cor}

\begin{proof}
By Theorem \ref{two poset iso}, 
$\D(w)\leq\D(v)$ if and only if $\Psi(\D(w))\leq\Psi(\D(v))$.

Since any object of $\Psi(\D(w))$ and $\Psi(\D(v))$ is 2-term, 
we have 
$\Psi(\D(w))\leq\Psi(\D(v))$ if and only if  
$$\Hom_{\Db(\mod\Pi)}(X,Y[-1])=0$$
for any $X\in\Psi(\D(w))$ and  $Y\in\Psi(\D(v))$.
Because $$\Psi(\D(w))=\Phi(\G(w))\sqcup\Phi(\R(w))[1],$$ it is also equivalent to saying that 
$\Hom_{\Pi}(M,N)=0$ for any 
$M\in\Phi(\G(w))$ and $N\in\Phi(\R(v))$.
Then the conclusion follows from Theorem \ref{CWthm}. 
\end{proof}


At the end of this section, we will discuss a consequence of the above results.

Let $A$ be a finite dimensional algebra.  
Even though we know that there exists a bijection between  $\fLsbrick A$ and  $\twosmc A$, it is not known that a calculatable map from $\fLsbrick A$ to  $\twosmc A$ in general. 

Theoretically, for a given $\sS\in\fLsbrick A$, we obtain the corresponding SMC $\xX_\sS$ by the following steps (see \cite[Figure 1]{A1} for more details). 

\begin{itemize} 
\item[(i)] 
From $\sS$, we calculate 
the smallest torsion class $\T(\sS)$ containing $\sS$. 
\item[(ii)] 
From the torsion class $\T(\sS)$, we calculate a Ext-projective of $\T(\sS)$, which we denote by $T$. 
It turns out that $T$ is a support $\tau$-tilting module and it is identified with a support $\tau$-tilting pair $(T,P)$ for some projective module $P$ \cite[Theorem 2.7]{AIR}. 

\item[(iii)]  
From $(T,P)$, we calculate $D(T,P)^\dagger:=(\tau T\oplus \nu P,\nu T_{pr})$, where $\nu$ is the Nakayama functor and $T_{pr}$ is a maximal projective direct summand of $T$. It turns out to be a support $\tau^{-}$-tilting pair \cite[Theorem 2.14]{AIR}. 

\item[(iv)] 
From $U:=\tau T\oplus \nu P$, 
we calculate the set $\sS'=\{X_1,\ldots,X_i\}$ of isoclasses of indecomposable direct summands of
$\Soc_{\End_A(U)}(U)$.

\end{itemize}
Then, we conclude 
$\sS\sqcup\sS'[1]$ is a 2-term SMC, and 
this correspondence gives a bijection between 
$\fLsbrick A$ and  $\twosmc A$ \cite{A1}.

However, it is very hard to calculate all modules of $\T(\sS)$ and the Ext-projective of $\T(\sS)$. 
From this viewpoint, we pose the following question. 
\begin{question}
For a given left finite semibricks, how can we calculate an explicit description of the corresponding SMC ?
\end{question}

In our situation, the map $\fLsbrick\Pi=\sbrick\Pi\to\twosmc\Pi$ is given by simple combinatorics of arc diagrams, that is, all we have to do is to calculate $\D(w)$ from $\G(w)$, which is done in \cite{R3} as pointed out in Remark \ref{G,R,D}.


\section{Quotient algebras of the preprojective algebra}
In this section, 
we discuss semibricks for several quotient algebras of preprojective algebras. 
As an application, we study semibricks of some important  classes of algebras and recover some known results. 

As before, fix an integer $n\geq1$ and let $\Pi=\Pi_n$ be the preprojective algebra of type $A_{n}$ and 
$W=W_{n}$ the symmetric group of degree $n+1$.

Let $I$ be a two-sided ideal of $\Pi$. 
Then we define 
$$\NAD_I:=\{\G\in\NAD\ |\ S(\alpha)\in\mod(\Pi/I),\ \forall\alpha\in\G\}.$$

Then we naturally regard $\NAD_I$ as an induced subposet of $\NAD$, that is, for $\G,\G'\in\NAD_I$, we have 
$\G\leq\G'$ on $\NAD_I$ if and only if $\G\leq\G'$ on $\NAD$.

As before, for $\G\in\NAD_I$, we can define 
\begin{eqnarray*}
\Phi(\G)=\{S(\alpha)\ |\ \alpha\in\G \}\subset\mod (\Pi/I).
\end{eqnarray*}


Moreover, for $\mathcal{C}\subset\mod(\Pi/I)$, we denote  
$\T_{\Pi}(\mathcal{C})$  (resp. $\T_{\Pi/I}(\mathcal{C})$) 
by the smallest torsion class containing $\mathcal{C}$ in $\mod\Pi$ (resp. in $\mod(\Pi/I)$). 
Note that $\T_{\Pi}(\mathcal{C})=\Filt_\Pi(\Fac\mathcal{C})$ (see \cite[Lemma 3.1]{MS}). We let  

$$\sbrick(\Pi)\cap\mod(\Pi/I):=\{\sS\in\sbrick\Pi\ |\ S\in\mod(\Pi/I),\ \forall S\in\sS\}.$$

\begin{lemm}\label{sbrick iso}
We have a poset isomorphism 
$$\sbrick(\Pi)\cap\mod(\Pi/I)\to\sbrick(\Pi/I).$$
\end{lemm}

\begin{proof}
Because $\mod(\Pi/I)$ is a full subcategory of $\mod\Pi$, 
we can identify $\brick(\Pi/I)$ with 
$\brick(\Pi)\cap\mod(\Pi/I)=\{S\in\brick\Pi\ |\ SI=0\}$ and 
$\sbrick(\Pi/I)$ with 
$\sbrick(\Pi)\cap\mod(\Pi/I)$.

Take $\sS,\sS'\in\sbrick(\Pi)\cap\mod(\Pi/I)$ and assume that 
$\sS\leq\sS'$, that is, $\T_{\Pi}(\sS)\subset\T_{\Pi}(\sS').$ 
Since $\T_{\Pi}(\mathcal{C})=\Filt_\Pi(\Fac\mathcal{C})$, we have 
$\T_\Pi(\sS)\cap\mod(\Pi/I)=\T_{\Pi/I}(\sS)$. Thus, we have 
$\T_{\Pi/I}(\sS)\subset\T_{\Pi/I}(\sS').$ 

Conversely, 
take $\sS,\sS'\in\sbrick(\Pi/I)$ and assume that 
$\sS\leq\sS'$, that is, $\T_{\Pi/I}(\sS)\subset\T_{\Pi/I}(\sS').$ 
Then since $\T_\Pi(\T_{\Pi/I}(\sS))=\T_{\Pi}(\sS)$, we have 
$\T_{\Pi}(\sS)\subset\T_{\Pi}(\sS').$ 
Thus the map is an isomorphism.
\end{proof}

\begin{thm}\label{main3}
Let $I$ be a two-sided ideal of $\Pi$. 
Then we have a poset isomorphism
$$\Phi: \NAD_I\longrightarrow\sbrick(\Pi/I).$$
\end{thm}



\begin{proof} 
By Theorem \ref{two poset iso},
we have a poset isomorphism 
$$\NAD_I\longrightarrow\sbrick(\Pi)\cap\mod(\Pi/I).$$
Moreover, Lemma \ref{sbrick iso} gives a poset isomorphism 
$\sbrick(\Pi)\cap\mod(\Pi/I)\to\sbrick(\Pi/I)$. 
\end{proof}

As a first application of Theorem \ref{main3}, 
we consider the case that semibricks of the algebras do not change. 

\begin{cor}\label{cycle}
Let $I_{\textnormal{cyc}}$ be the ideal of $\Pi$ generated by all 2-cycles. 
Then we have a poset isomorphism
$$\Phi:\NAD\longrightarrow\sbrick (\Pi/I_{\textnormal{cyc}}).$$
\end{cor}

\begin{proof}
Since we have $\NAD=\NAD
_{I_{\textnormal{cyc}}}$, 
we get $\NAD\cong\NAD_{I_{\textnormal{cyc}}}\cong\sbrick (\Pi/I_{\textnormal{cyc}})$ by Theorem \ref{main3}.
\end{proof}

\begin{remk}
Corollary \ref{cycle}  was also shown in \cite[Lemma 4.2.8]{BCZ} and \cite[Proposition 6.7]{DIRRT}.
\end{remk}

As a second application of Theorem \ref{main3},
we study semibricks of a path algebra in terms of a special class of arc diagrams.

\begin{defi}\cite{R3}\label{RNAD}
A \emph{right arc} is an arc that does not pass to the left (=above) of any point, and we call a noncrossing arc diagram having only right arcs a \emph{right noncrossing arc diagram}. We denote the set of right  noncrossing arc diagrams by 
$\RNAD=\RNAD_n$ and regard it as a subposet of $\NAD$.
\end{defi}

\begin{remk}
\begin{itemize} 
\item[(i)]
The right noncrossing arc diagrams are nothing but \emph{noncrossing partitions}, and its enumeration is well-known : The number of right noncrossing arc  diagrams on $n$ points is the Catalan number $\frac{1}{n+1}\left(\begin{smallmatrix}
2n\\
n
\end{smallmatrix}
\right)$.
\item[(ii)] 
The name of \emph{right arc} follows from \cite{R3}, which makes sense if we arrange an arc diagram vertically.   
\end{itemize} 
\end{remk}

As a corollary of Theorem \ref{main3}, we have the following result. 

\begin{cor}\label{linear}
Let $\vec{Q}$ be a linear quiver of type $A_n$. 
Then we have a poset isomorphism
$$\Phi:\RNAD\longrightarrow\sbrick K\vec{Q}.$$
\end{cor}

\begin{proof}
For $I:=\langle a_1^-,a_2^-,\ldots,a_{n-1}^-\rangle$, 
we have $\Pi/I\cong K\vec{Q}$. 
On the other hand,  it is easy to check 
$$\NAD_I=\RNAD.$$
Thus Theorem \ref{main3} implies the assertion.
\end{proof}

\begin{remk}
Corollary \ref{linear} can be easily generalized to an arbitrary quiver of type $A$ by the notion of $c$-sortable arcs ($c$-sortable permutations) (see \cite{BaR,R1,R2} for definitions and further background). 
More precisely, let $Q$ be a quiver of type $A$ and $c$-$\NAD$ the set of noncrossing arc diagrams consisting of $c$-sortable arcs corresponding $Q$. 
Then we have a poset isomorphism 
$$\Phi:c\textnormal{-}\NAD\longrightarrow\sbrick KQ.$$
This result together with Theorem \ref{poset} recovers one of the main results by Ingalls-Thomas \cite{IT}, which relate $c$-sortable elements and torsion classes.
\end{remk}

To give a last application, we introduce the notion of alternating arc diagrams.

\begin{defi}\cite{BaR}\label{ANAD}
A \emph{right-even alternating arc} is an arc that passes to the right (=below) of even points and to the left (=above) of odd points. A \emph{left-even alternating arc} is an arc that passes to the left of even points and to the right of odd points. 

An \emph{alternating arc} is an arc that is either right-even alternating or left-even alternating or both. We call a noncrossing arc diagram consisting of alternating arcs an \emph{alternating noncrossing arc diagram}. 
We denote the set of alternating noncrossing arc diagrams by $\ANAD=\ANAD_n$. 
\end{defi}

Let $\rad(\Pi)$ be a radical of $\Pi$ and 
consider the quotient algebra $\Pi/\rad^2(\Pi)$, where $\rad^2(\Pi)$ is the radical square. 
It is equivalent to saying that $\rad^2(\Pi)$ 
is the ideal of $\Pi$ generated by all paths of length two. 

Then we have the following result. 

\begin{cor}\label{cor2}
We have a poset isomorphism
$$\ANAD\longrightarrow\sbrick (\Pi/\rad^2(\Pi)).$$
\end{cor}

\begin{proof}
It is easy to check 
$$\NAD_{\rad^2(\Pi)}=\ANAD.$$
Thus Theorem \ref{main3} implies the assertion.
\end{proof}

\begin{remk}
Let $B$ be a Brauer line algebra with the multiplicity 1 (see, for example, \cite[section 4]{RZ} and \cite[section 2]{AMN} for the definition).   
Then we have a poset isomorphism
$$\sbrick B\cong  \sbrick(\Pi/\rad^2(\Pi))$$
(for example, this follows from  \cite[section1.4]{A1} or \cite[Proposition 4.3]{Ad}). 
Thus, Corollary \ref{cor2} gives a classification of semibricks of a Brauer line algebra in terms of alternating arc diagrams. 
Moreover the result of \cite{BaR} gives a bijection a between $\ANAD$ and 
the quotient lattice $W/\theta_{biC}$, where $\theta_{biC}$ denotes by the biCambrian congruence. 
Since there exists a bijection between the set of torsion classes, the set of 2-term silting complexes and the set of semibricks \cite{AIR,A1}, 
the above result also implies results of \cite[Theorem 
7.10]{DIRRT}, \cite{Ao} and \cite[Theorem 5.2]{AMN}. 
\end{remk}


\textbf{Acknowledgements.}
The author would like to thank Takahide Adachi, Aaron Chan, Osamu Iyama, Toshiya Yurikusa for helpful comments and variable discussion on the first version of this paper. 
He would like to thank the referee for carefully reading our manuscript and detailed comments that  helped his to improve the manuscript.


\end{document}